\documentclass[reqno,11pt]{amsart}
\usepackage[colorlinks=true, linkcolor=blue, citecolor=blue]{hyperref}

\usepackage{amssymb}
\usepackage{amsmath, graphicx, rotating} 
\usepackage{color}
\usepackage{soul}
\usepackage[dvipsnames]{xcolor}
  
\usepackage{ifthen}
\usepackage{xkeyval}
\usepackage{todonotes}
\usepackage[T1]{fontenc}
\usepackage{lmodern}
\usepackage[english]{babel}

\usepackage{ upgreek }
\usepackage{stmaryrd}
\SetSymbolFont{stmry}{bold}{U}{stmry}{m}{n}
\usepackage{amsthm}
\usepackage{float}

\usepackage{ bbm }
\usepackage{ stmaryrd }
\usepackage{ mathrsfs }
\usepackage{ frcursive }
\usepackage{ comment }

\usepackage{pgf, tikz}
\usetikzlibrary{shapes}
\usepackage{varioref}
\usepackage{enumitem}
\usepackage{longtable}

\usepackage{mathtools}

\usepackage{dsfont}

\definecolor{rouge}{rgb}{0.7,0.00,0.00}
\definecolor{vert}{rgb}{0.00,0.5,0.00}
\definecolor{bleu}{rgb}{0.00,0.00,0.8}

\usepackage[margin=1.26in]{geometry}

\newtheorem{theorem}{Theorem}[section]
\newtheorem*{theorem*}{Theorem}
\newtheorem{lemma}[theorem]{Lemma}

\newtheorem{corollary}[theorem]{Corollary}
\newtheorem{proposition}[theorem]{Proposition}

\labelformat{hypothesis}{\textbf{M\kern-0.1mm#1}}

\newtheorem{condition}{Condition}
\newtheorem{conditionA}{A\kern-0.1mm}
\labelformat{conditionA}{\textbf{A\kern-0.1mm#1}}

\newtheorem{conditionEmpty}{\kern-0.1mm}
\labelformat{conditionEmpty}{\textbf{\kern-0.1mm#1}}

\renewcommand\dots{\hbox to 1em{.\hss.\hss.}}

\newenvironment{conditionAstar}[1]
{\renewcommand{\theconditionEmpty}{\ref{#1}*}
	\begin{conditionEmpty}}
	{\end{conditionEmpty}}  

\theoremstyle{definition}

\newtheorem{remark}[theorem]{Remark}

\numberwithin{equation}{section}

\newcommand*{\abs}[1]{\left\lvert#1\right\rvert}
\newcommand*{\norm}[1]{\left\lVert#1\right\rVert}

\def\bb#1{\mathbb{#1}}

\def\bf#1{\mathbf{#1}}
\def\scr#1{\mathscr{#1}}

\def\wt{\widetilde}
\def\wh{\widehat}
\def\ds#1{\mathds{#1}}

\newcommand\ee{\varepsilon}

\DeclarePairedDelimiter\floor{\lfloor}{\rfloor}

\def\Rd {\mathbb{R}^d}
\def\Rd*{(\mathbb{R}^d)^*}
\def\Pd{{\mathbb{P}}^{d-1}}
\def\Pd*{(\mathbb{P}^{d-1})^*}

\begin{document}

\title[Seneta-Heyde scaling for matrix branching random walks]  
{Spinal decomposition, martingale convergence and the Seneta-Heyde scaling for matrix branching random walks}

\author{Ion Grama}
\author{Sebastian Mentemeier}\thanks{S.M. was supported by DFG grant ME 4473/2-1.} 
\author{Hui Xiao}

\curraddr[Grama, I.]{Univ Bretagne Sud, CNRS UMR 6205, LMBA, Vannes, France}
\email{ion.grama@univ-ubs.fr}

\curraddr[Mentemeier, S.]{Universität Hildesheim, Institut für Mathematik und Angewandte Informatik, Hildesheim, Germany}
\email{mentemeier@uni-hildesheim.de}

\curraddr[Xiao, H.]{Academy of Mathematics and Systems Science, Chinese Academy of Sciences, Beijing 100190, China}
\email{xiaohui@amss.ac.cn}


\begin{abstract}
We consider a matrix branching random walk on the semi-group of nonnegative matrices, where we are able to derive, under general assumptions, an analogue of Biggins' martingale convergence theorem for the additive martingale $W_n$, a spinal decomposition theorem, convergence of  the derivative martingale $D_n$, and finally, the Seneta-Heyde scaling stating that in the boundary case $c \sqrt{n} W_n \to D_\infty$ a.s., where $D_\infty$ is the limit of the derivative martingale and $c$ is a positive constant.  

As an important tool that is of interest in its own right, we provide explicit duality results for the renewal measure of centered Markov random walks, relating the renewal measure of the process, killed when the random walk component becomes negative, to the renewal measure of the ascending ladder process.
\end{abstract}

\date{\today}
\subjclass[2020]{Primary 60J80; 
	Secondary 60B15, 
	60J05, 
	60K15, 
	60G42, 
	}
\keywords{Branching random walks; Derivative martingale; Products of random matrices, Seneta-Heyde scaling; Markov renewal theory; Duality}

\maketitle

\tableofcontents

\section{Introduction}
To introduce a matrix branching random walk which is the object of our study, 
let us first recall the definition of a branching random walk on $\bb R^d$. 
The process begins with a single particle at the origin at time $n=0$, 
which at time $1$ produces a random number of offspring, that is displaced from their parent's position. 
The probability law of the displacements as well as the number of children is fully determined by a point process on $\bb R^d$. 
The children of the initial particle form the first generation of the branching random walk, 
which subsequently reproduce independently according to the same law given by the point process. 
Each subsequent generation is obtained according to the same rule. 

In a matrix branching random walk, the additive group $\bb R^d$ is replaced by a multiplicative (semi-)group of $d \times d$ matrices ($d \ge 2$), hence the \emph{positions} of the particles become matrices, with the initial value being the identity matrix.
 Here, we consider a matrix branching random walk on the semigroup of nonnegative matrices, and its action on nonnegative vectors.  
This is a very versatile model: on one hand, it can be seen as the natural extension of branching random walk from additive abelian groups to multiplicative, noncommutative (semi-)groups. On the other hand, when considering the action of the matrix branching random walk on vectors, this generalizes (additive) branching random walk in $\bb R^d$ in the sense that the displacement is no longer additive and independent of the parent's position, but given by multiplying the parent's position with a random matrix, thereby also introducing a Markovian dependence. The latter relates the model also to multitype branching process with continuous state space, as treated for instance in \cite{BK04}.

In recent years, there has been significant interest in studying branching random walks 
and Brownian motions on $\bb R^d$, 
particularly their extremal behavior. For an overview, we refer for example to \cite{BG23,BCMZ24,KLZ23,Mal15,SBM21}.
The extremal position for a branching matrix random walk was considered in the authors' paper \cite{GMX22}. 
In the present paper, we shall complement \cite{GMX22} by obtaining new results.
 
 A fundamental object in the study of branching random walks is the additive martingale, for which we are going to prove a full analogue of Biggins' martingale convergence theorem \cite{Big77, IKM20, Lyo97}, in the sense that we obtain explicit equivalent conditions for the mean convergence of the additive martingale. Previous works in this direction \cite{BGL20, Men16} have only provided necessary or sufficient conditions separately. As a well-established tool for proving this and further results, we prove a spinal decomposition theorem (cf. \cite[Section 4.4]{Shi12} and references therein). Our main objective is to study martingale convergence in the so-called boundary case, where the mean drift of particles is zero. Then the limit of the additive martingale vanishes almost surely, and we introduce for our setting the analogue of the so-called derivative martingale \cite{BK04,Kyp98}. Firstly, we show its convergence to a positive limit, conditioned on the systems survival. Secondly, we prove a Seneta-Heyde scaling \cite{AS14, BK97} for the additive martingale in the boundary case, which upon multiplication by $\sqrt{n}$ will converge to the limit of the derivative martingale. Here, we will follow the recent strategy developed in \cite{BM19}. In our setting, this amounts to the study of renewal theory for a {\em centered} Markov random walk driven by products of random matrices. We prove an explicit duality relation, relating the renewal measure of the process killed when becoming negative to the renewal measure of the ascending ladder process; with explicit bounds for the Markov renewal function.


\textit{Related works and history of the problem.}
For general information about branching random walks (on the real line), we refer to the textbook \cite{Shi12} or the survey article \cite{Zei16}. The study of random walks on matrix (semi-)groups goes back to \cite{FK60}, see the recent book \cite{BQ16b} for a survey of the topic. Matrix branching random walks have been studied previously in 
\cite{BGL20,BDGM14,GMX22, IM15, KM15,Men16}. 

The {\em additive martingale} with a parameter, evaluated at time $n$, can be seen as the Laplace transform of the intensity measure of the point process generated by the branching random walk in generation $n$ and therefore -- if it exists on an open parameter set -- encodes all information about the branching random walk. It figures prominently e.g. when studying the relative frequencies of particles that are at a distance from the origin that grows at a fixed linear rate in the number of the generation, see \cite{Big79}. A corresponding result for matrix branching random walk is in \cite{BGL20}. The leftmost particle moves with a linear speed an logarithmic correction term, see \cite{AR09,Hu16,HS09} and it converges after centering to a limiting law that is a Gumbel distribution with a random shift  given by the logarithm of the limit of the derivative martingale \cite{Aid13}. Therefore, the {\em derivative martingale} figures in the description of the extremal behaviour of the branching random walk and this motivates our study of the derivative martingale. The weak law of large numbers for the minimum of a matrix branching random walk is studied in \cite{GMX22}; a first convergence result for the derivative martingale in a matrix branching random walk was obtained in \cite{KM15}, under restrictive assumptions. The maximal displacement of a spherically symmetric branching random walk in $\bb R^d$ is studied in \cite{BG23}.

In the context of branching Brownian motion on $\bb R^d$, \cite{SBM21} consider the fluctuations of the minimum in a given direction and therefore introduce a version of the derivative martingale with a parameter on the sphere. Their objective is to study the preferred direction when conditioning on the initial behaviour of the branching Brownian motion up to a small time -- for the initial problem is spherically symmetric. In our setting of a matrix branching random walk, the derivative martingale naturally has an initial value from the sphere as a parameter, see Def. \ref{def-derivative-martingale} below.

The {\em spinal decomposition} in multitype branching is constructed in \cite[Section 12]{BK04}, with the spine being called trunk there. The use of spinal decomposition goes back to \cite{Lyo97}. 

{\em Renewal theory for Markov random walks driven by products of random matrices} was pioneered in \cite{Kes73} in the case with drift. The main obstacle in deriving results for centered processes is the lack of a simple duality lemma, due to the non-commutative structure. This is why so far, no explicit results for centered processes exist. There are only abstract versions  of Wiener-Hopf factorization for Markov random walks, see \cite{AS73} or results for finite state space Markov random walks \cite{Asm03}.

\section{Notation and conditions}\label{Sec-notation-condi}

\subsection{Basic notation}
Let $\bb R_{+} := [0, \infty)$ and denote by  $\bb N=\{0,1,2,\ldots\}$ the set of non-negative integers.
For $t \in \bb R$, denote $t^+ = \max \{ t, 0 \}$ and $\log^+ t = \max \{ \log t,  0\}$. 
Let $d\geq 1$ be an integer 
and let $\bb R^d$  be the  $d$-dimensional Euclidean space  
equipped with the scalar product $\langle x, y \rangle :=\sum_{i=1}^d x_i y_i$ for $x = (x_1, \ldots, x_d)$ and $y = (y_1, \ldots, y_d)$. 
Denote by $\bb R^d_+$ the positive quadrant of $\bb R^d$. 
We fix an orthonormal basis $(e_i)_{1 \leq i \leq d}$ of $\bb R^d$ 
and define the associated norm by $\|v\| := \sum_{i=1}^d |\langle v, e_i\rangle |$ for $v \in \bb R^d$. 
A $d \times d$ matrix $g$ with non-negative entries is said to be \emph{allowable},
if every row and every column of $g$ has a strictly positive entry. 
Let $\bb M$ be the semigroup of $d\times d$  matrices with non-negative entries which are allowable. 
Further let $\bb M_+$ be the subsemigroup of $\bb M$ with strictly positive entries.
The action of $g \in \bb M$ on a vector $v \in \bb R_+^d$ is denoted by $gv$.
For any $g \in \bb M$, let $\| g \| := \sup_{v \in \bb R^d_+ \setminus \{0\} } \frac{\| g v \|}{\|v\|}$
and $\iota(g) : = \inf_{v \in \bb R^d_+ \setminus \{0\} } \frac{\| g v \|}{\|v\|}$.

The unit sphere in $\bb R^d$ is denoted by $\bb S^{d-1}:=\{ v \in \bb R^d : \| v \| =1 \}.$ 
Let  $\bb S_+^{d-1}$ be the intersection of the unit sphere with the positive quadrant: $\bb S_+^{d-1} := \bb S^{d-1} \cap \bb R_+^{d}$.  For any  $g\in \bb M_+$, we can define its action on $x \in \bb S_+^{d-1} $ by $g\cdot x := \frac{g x}{\| g x \|}$. Note that this is well defined by the requirement that $g$ is allowable.
For any  $g \in \bb M$ and $x \in \bb S_+^{d-1}$, let
\begin{align}\label{Def-dot-cocycle}
\sigma \, :\, \bb M \times \bb S_+^{d-1} \to \bb R, \qquad \sigma(g, x) := \log \|gx\|, 
\end{align}
which satisfies the {\em cocycle property}: for  $g_2, g_1 \in \bb M$ and $x \in \bb S_+^{d-1}$, 
\begin{align}\label{eq-cocycle-property}
\sigma(g_2 g_1, x) = \sigma(g_2, g_1 \cdot x) + \sigma(g_1, x).
\end{align}

Throughout the paper, the symbols $c, C, c', C'$ denote positive and finite constants that may vary at each occurrence. 

\subsection{Branching random walk, probability measures and filtrations}
To define a branching random walk on the semigroup $\bb M$, we  
assume that we are given a point process $\mathscr N$ on the allowable matrices in $\bb M$, 
that is a random counting measure on the Borel subsets of the semigroup $\bb M$. 
Let $\bb U:= \bigcup_{n=0}^\infty \bb N^n$ be the infinite tree 
with Ulam-Harris labelling and root $\o$, where by convention $\bb N^0:=\{\o\}$.  
We say that $u \in \bb U$ belongs to generation $n$ (shortly $|u|=n$) iff $u \in \bb N^n$.
Hence, each node of generation $n$ will be identified with an $n$-tuple  $u=(u_1, \ldots, u_n)$, where $u_i\in \bb N$. For notational simplicity, we write $u=u_1\dots u_n$ and further denote by $u|k:=u_1 \dots u_k$ the restriction of $u$ to its first $k$ components. Given $u=u_1 \dots u_n$ and $v=v_1 \dots v_m$, we write $uv=u_1 \dots u_nv_1 \dots v_m$ for the concatenation.

  Let $(\mathscr{N}_u)_{u \in \bb U}$ be a family of independent and identically distributed (i.i.d) copies of $\mathscr{N}$
  indexed by the elements of $\bb U$.
For each $u \in \bb U$, let $N_u:=\mathscr{N}_u(\bb M)$ denote the number of points of $\mathscr{N}_u$, which we will consider as the number of children of the particle pertaining to  the node $u$.
In particular $N:=N_{\o}$ represents the number of children pertaining to the node $\o$. 
Define $g_{\o}$ as the identity matrix. 
Then, upon using the representation 
\begin{align}\label{def-basic-point-process}
\mathscr{N}_u = \sum_{j=1}^{N_u} \delta_{g_{uj}}, 
\end{align}
we obtain an associated family $(g_v)_{v \in \bb U}$ of random matrices. 
From $(\mathscr{N}_u)_{u \in \bb U}$, we construct a random subtree $\bb T \subset \bb U$ as follows: 
let $\bb T_0:=\{\o\}$, $\bb T_1:=\{i \in \bb N \, : \, 1\leq i \leq N_{\o}\}$ and thereupon, we define recursively
$$ \bb T_n := \{ vi \, : \, v \in \bb T_{n-1},\ 1 \leq i \leq N_v\}. $$
Then $\bb T := \bigcup_{n=0}^\infty \bb T_n$. 
For $u, v \in \bb T$, we write $u < v$ if there is $k < |v|$ with $u = v|k$
and $u \nless v$ otherwise. 
For any $u \in \bb T \setminus \o$, we write $\overset{\leftarrow}{u}$ for its parent, 
and ${\rm brot}(u)$ for the set of brothers of $u$, i.e., $v \in {\rm brot}(u)$ means that $v$ and $u$ have the same parent
and $v \neq u$.
Note that the set ${\rm brot}(u)$ can be empty.

 The branching random walk $(G_u)_{u\in \bb T}$  on the semigroup $\bb M$ 
 is defined then by taking the left product of the random matrices  
along the branch leading from $\o$ to $u \in \bb T$:  
\begin{align}\label{def-product-Gu}
G_u := g_u g_{u|n-1}\ldots g_{u|2} g_{u|1}, \ u\in \bb T.
\end{align} 
Note that these factors are independent, but not necessarily identically distributed.

The family $(\mathscr{N}_u)_{u \in \bb U}$ can be redefined on a canonical probability space. 
Write $\mathscr{R}:=\mathscr{R}(\bb M)$ for the set 
of Radon measures on $\bb M$ and note that the  point process $\mathscr{N}$ is a $\mathscr{R}$-valued random variable.
Let $\mathscr{A}$ be the Borel $\sigma$-field on $\mathscr{R}$ and let $\eta$ denote the law of $\mathscr{N}$.
Then, the family of the point process $(\mathscr{N}_u)_{u\in \bb T}$ 
has the same law as the coordinate process on the "branching product" space
$$ 
(\Omega := \mathscr{R}^{\bb U}, \mathscr{F}:= \mathscr B ( \mathscr{R}^{\bb U} ), \bb P:=\eta^{\otimes{\bb U}}),
$$
where $\mathscr{R}^{\bb U}$ is understood as the product space
$\mathscr R \times {\mathscr R}^{\bb N} \times {\mathscr R}^{\bb N^2}\times \cdots $
and $ \mathscr B ( \mathscr{R}^{\bb U} )$ is the corresponding Borel $\sigma$-algebra. 
In order to define 
the probability $\eta^{\otimes{\bb U}}$,  
one can apply the general form of the Kolmogorov extension theorem for compatible systems of probability measures: 
first, for $k\geq 1$, define the measures $\eta_k= \eta ^{\otimes{\bb N}^k}$,
and then set $\eta^{\otimes{\bb U}}= \eta\otimes \eta_1\otimes  \eta_2\otimes \cdots $.
The expectation corresponding to the probability $\mathbb P$ will be denoted by $\mathbb E.$  

The event corresponding to the system's survival after $n$ generations is denoted by
$\mathscr S_n   = \{ \sum_{|u|=n} 1 \mbox >0 \}.$ 
Then $\mathscr S = \cap_{n = 1}^{\infty} \mathscr S_n$ is the event corresponding to the system's ultimate survival.
We will introduce the assumption $\bb E N>1$ (see \ref{Condi_ms}) which guarantees that $\bb P(\scr S)>0$. This allows to define the probability measure $\bf P=\bb P( \cdot | \scr{S})$.
Below, we say that a sequence of random variables $(\xi_n)_{n\geq 1}$ converges in probability to the random variable $\xi$
under the system's survival if for any $\ee >0$ it holds 
$
\lim_{n\to\infty} \bf P( |  \xi_n - \xi | \geq \ee ) =0. 
$
The convergence almost surely under the system's survival is defined in the same way: 
$
\bf P( \lim_{n\to\infty} \xi_n = \xi ) =1.
$

To consider the action of the branching product of random matrices on vectors in $\bb R^d_+$, we will also need to include information about an initial direction $X_{\o} \in \bb S^{d-1}_+$ and position $S_{\o} \in \bb R$. 
To do so, consider the space $\Omega^{(1)} := \Omega \times \bb S^{d-1}_+ \times \bb R$ with corresponding product Borel $\sigma$-field $\mathscr{F}^{(1)}$ with coordinate maps $X_{\o} : \Omega^{(1)} \to \bb S^{d-1}_+$ and $S_{\o} : \Omega^{(1)} \to \bb R$. 
For each $x \in \bb S^{d-1}_+$ and $b \in \bb R$, define a probability measure $\bb P_{x,b}$ on $\Omega^{(1)}$ by setting 
$ \bb P_{x,b} = \bb P \otimes \delta_{(x,b)}$. 
In particular, $\bb P_{x,b}((X_{\o},S_{\o})=(x,b))=1$. 
In the same way, we define $\bf P_{x,b}:=\bf P \otimes \delta_{(x,b)}$.
Moreover, we will use the convention that $\bb P_{x, 0} = \bb P_x$ and $\bf P_{x,0} = \bf P_x$.

For each node $u\in \bb T$, 
we introduce the following functions of the branching process:  
\begin{align} \label{Def-Xku_Sku}
	X_u = G_u \cdot X_{\o}   \quad \mbox{and} \quad  S_u =  S_{\o} + (- \sigma (G_{u}, X_{\o}) ), 
\end{align} 
where $G_u$ is defined in \eqref{def-product-Gu}. 
There is a natural filtration given by 
\begin{align}\label{def-filtration-Fn}
\scr{F}_n:=\sigma ( X_{\o}, S_{\o}, \{ u, g_u \, : \, u \in \bb T, |u| \leq n\} ).
\end{align}
Hence $\mathscr{F}_\infty := \bigcup_{n=0}^\infty \scr{F}_n$ contains all the information of the matrix branching random walk.

We now introduce the following shift operator which will help us to write conveniently the branching property all over the paper. 
For any function $F=F\big(X_{\o}, S_{\o}, (g_u)_{u \in \bb T} \big)$ 
of the branching random walk $(G_u)_{u\in \bb T}$ and any node $w \in \bb T$, we define 
\begin{equation}\label{eq:shift-operator}
	[F]_w:=F \big( X_w, S_w, (g_{wu})_{u \in \bb T_w} \big), 
\end{equation}
where $\bb T_w$ denotes the random subtree rooted at $w$, its first generation being formed by the children of the particle at $w$. That is, $[F]_w$ is evaluated on the subtree started at $w$.
In particular, we can apply this notation when the function $F$ selects a single branch of this tree.  
For example, $[G_u]_w$ means the product $g_{wu} g_{wu|n-1} \ldots g_{wu|2} g_{wu|1}$ for $u=u_1 u_2 \dots u_n \in \bb T_{w}$, 
i.e. the product of the elements of $\bb M$ along the branch leading from $w$ to $wu$.

\subsection{Assumptions}
We denote by $\Gamma$ the smallest closed subsemigroup of $d\times d$ 
matrices with nonnegative entries that covers $ \{G_u:  |u|=1\}$ with probability 1. 
\begin{conditionA}\label{CondiAP}  
We assume the following allowability and positivity conditions: 
\begin{itemize}
\item[(i)] (Allowability) Every $g\in\Gamma$ is allowable.
\item[(ii)] (Positivity) $\Gamma$ contains at least one matrix belonging to $\bb M_+$.
\end{itemize}
\end{conditionA}

We shall also use the following Furstenberg-Kesten condition 
which implies condition \ref{CondiAP}
and is introduced in their pioneering work \cite{FK60}. 

\begin{conditionAstar}{CondiAP}\label{Condi-Furstenberg-Kesten}
There exists a constant $\varkappa > 1$ such that  for any $g = (g^{i,j})_{1 \leq i, j \leq d} \in \Gamma$, 
\begin{align*} 
g>0 \quad \mbox{ and } \quad 
\frac{\max_{1\leq i, j  \leq d} g^{i,j} }{ \min_{1\leq i, j  \leq d} g^{i,j} } \leq \varkappa.
\end{align*}
\end{conditionAstar}

It follows from the Perron-Frobenius theorem that every  $g \in \bb M_+$
has a dominant eigenvalue $\lambda_g>0$, with the corresponding eigenvector $v_g \in \bb S_+^{d-1}$. 
Let
\begin{equation} \label{def-VGamma-pos}
\Lambda (\Gamma) = \overline{\{v_{g} \in \bb S_+^{d-1}: g\in \Gamma \cap \bb M_+ \}}. 
\end{equation}

We say that the semigroup $\Gamma$
is \emph{arithmetic}, if there exists $t>0$ together with $\beta \in[0,2\pi)$ and a continuous function
$\vartheta: \bb S_+^{d-1} \to \mathbb{R}$ such that
$\exp[2 \pi it\sigma(g, x)-i\beta + i\vartheta(g\!\cdot\!x)-i \vartheta(x)]=1 $
for any $g\in \Gamma$ and $x\in \Lambda (\Gamma)$.
In other words, $\sigma(g, x)$ is contained in $t^{-1} \bb Z$ up to a shift that may depend on $g$ and $x$ through the function $\vartheta$.
We impose the following non-arithmetic condition:
\begin{conditionA}\label{CondiNonarith}
	{\rm (Non-arithmeticity) }
	The semigroup $\Gamma$ is non-arithmetic.
\end{conditionA}
A simple sufficient condition introduced in \cite{Kes73} for $\Gamma$
to be non-arithmetic is that the additive subgroup of $\mathbb{R}$ generated by the set
$\{ \log \lambda_{g} : g \in \Gamma \cap \bb M_+ \}$
is dense in $\mathbb{R}$, see \cite[Lemma 2.7]{BM16}.

Let
$$ I := \Big\{ s \in \bb R_+ \, : \,  \bb E \Big( \sum_{|u| = 1}  \norm{G_u}^s \Big) < \infty  \Big\}.$$ 
We will require, as a part of assumption \ref{Condi_ms} below, that $I$ has nonempty interior. This implies in particular that the intensity measure $\mu$ of $\scr{N}$,
\begin{align}\label{def-measure-mu}
\mu(A):= \bb E \Big[ \sum_{ |u| = 1 } \mathds{1}_A(G_u) \Big]
\end{align}
is $\sigma$-finite, since it integrates an everywhere positive function (cf. \cite{Kal17} p.\ 21).

Let $\scr C(\bb S_+^{d-1})$ be the space of continuous functions on $\bb S_+^{d-1}$. 
For any $s \in I$, define the transfer operator $P_s$ as follows: 
for $\varphi \in \scr C(\bb S_+^{d-1})$ and $x \in \bb S_+^{d-1}$,
\begin{align}\label{Def-Ps}
P_s \varphi(x) := \bb E \Big( \sum_{|u| = 1}  e^{s \sigma (G_u, x)} \varphi(G_u \cdot x) \Big) 
= \int_{\Gamma} e^{s\sigma(g,x)} \varphi(g \cdot x) \mu(dg).
\end{align}
Spectral gap properties of $P_s$ have been studied in \cite{BDGM14} under the additional assumption that $\mu$ is a probability measure. The analysis is still valid for $\sigma$-finite measures, since we only consider parameters $s \in I$ such that the integral in \eqref{Def-Ps} is finite. Note, however, that the point $0$ might not belong to $I$. 
We will provide all necessary information needed to adopt the arguments of \cite{BDGM14}  in Section \ref{sect:spectralgap}. 
In particular, by Proposition \ref{prop:Ps} below, for any $s \in I$, 
	there exists a unique dominant eigenvalue $\mathfrak{m}(s)>0$ of $P_s$ such that
	the function $s \mapsto \mathfrak{m}(s)$ is differentiable on $I^\circ$ (the interior of $I$). The corresponding eigenfunction $r_s$ satisfying $$P_s r_s = \mathfrak{m}(s) r_s$$ is strictly positive on $\bb S^{d-1}_+$.

We introduce the following condition, saying that the branching product of random matrices is in the boundary case and that the underlying branching process is supercritical.
\begin{conditionA}\label{Condi_ms}
	It holds $\bb E N>1$ and there is  a constant $\alpha \in I^{\circ}$ 
	such that 
	$\mathfrak m (\alpha)  = 1$ and $\mathfrak m '(\alpha) = 0$. 
\end{conditionA}

This condition will fix the constant $\alpha$ throughout the paper. Note that we cannot transform (as it is possible in the one-dimensional case) to a setting where $\alpha=1$, because this transformation would require taking powers of the matrices.

At some point, we will require a condition ensuring that a harmonic function for a process that is killed when leaving a certain set, is strictly positive. The precise condition is formulated in terms of a change of measure to be introduced below (cf.\  \eqref{precise-condi-Qx-alpha}), here we formulate a condition that can be checked directly on the point process $\scr{N}$ and implies \eqref{precise-condi-Qx-alpha}, see Lemma \ref{lem:condition_harmonic}.

\begin{conditionA}\label{Condi_harmonic}
There is $\delta>0$ such that for the constant $\varkappa$ from \ref{Condi-Furstenberg-Kesten}, it holds with $\overline{\varkappa}:=2 \log \varkappa$, 
\begin{equation*}
\bb E \bigg( \sum_{|u| = 1}  \mathds{1}_{(\overline{\varkappa} + \delta, \infty )}(- \log \norm{G_u}) \bigg) > 0.
\end{equation*} 
\end{conditionA}

For any constant $s \in I$, 
using the eigenfunction $r_s$ of $P_s$, we define a sequence of nonnegative random variables $W_n(s)$ as follows:
\begin{align}\label{def-addi-martigale}
W_n(s) =  \frac{1}{\mathfrak m (s)^n}  \sum_{|u| = n} e^{-s S_u} r_{s}(X_u),   \quad  n \geq 1. 
\end{align}
Then, for each fixed $x \in \bb S_+^{d-1}$ and $b \in \bb R$, this defines a $\bb P_{x,b}$-martingale with mean $r_s(x)e^{s b}$ with respect to the natural filtration $\scr{F}_n$ given by \eqref{def-filtration-Fn}. 
Denote $W_{\infty}(s) = \lim_{n \to \infty} W_n(s)$, $\bb P_{x,b}$-a.s. 
This is the analogue of the \textit{additive martingale} in the one-dimensional branching random walk 
(see e.g.\ \cite[Section 3.2]{Shi12}). 
We will study its convergence in Theorem \ref{Thm:Biggins-bb} for arbitrary parameter $s \in I$, 
proving an analogue of the Biggins martingale convergence theorem. 

In most parts of the paper, we will fix $s = \alpha$, where $\alpha$ is from condition \ref{Condi_ms}. 
In this case, we abbreviate $W_n$ for $W_n(\alpha)$, and $W_{\infty}$ for $W_{\infty}(\alpha)$.

The following moment assumption will be used to prove that the limit of the derivative martingale $D_n$ (defined in \eqref{def-derivative-martingale} below) is non-trivial. Recall that $t^+ = \max \{ t, 0 \}$ and $\log^+ t = \max \{ \log t,  0\}$ for $t \in \bb R$. 

\begin{conditionA}\label{condi:momentsW1} 
With $W_1$ defined by \eqref{def-addi-martigale} and $Z_1 := \sum_{|u| = 1}  S_u^+ e^{-\alpha S_u}$, for any $x \in \bb S_+^{d-1}$, it holds
\begin{align*}
\bb E_x \left( W_1  ( \log^+ W_1 )^2 \right) < \infty,  \qquad 
\bb E_x \left( Z_1 \log^+ Z_1 \right) < \infty. 
\end{align*}
\end{conditionA}

\subsection{Martingales and change of measure}\label{Subsection-change-of-measure}
In the main part of the paper, we shall assume \ref{Condi_ms}, 
meaning that $\bb E N>1$, $\mathfrak m (\alpha) = 1$ and $\mathfrak m'(\alpha)=0$. 
Then it holds that $W_\infty$, which is the $\bb P_{x,b}$-a.s. limit of $W_n$, is equal to $0$, $\bf P_{x,b}$-a.s., see Proposition \ref{Thm:Biggins} below.

Using this,  we may introduce a family of probability measures 
 $\bb Q_{x,b}^\alpha$ on $\bb S^{d-1}_+ \times \bb R \times \bb M^{\bb N}$ as follows.
Let $(X_0, S_0, g_i)_{i \ge 1}$ be the coordinate maps on $\bb S^{d-1}_+ \times \bb R \times \bb M^{\bb N}$ and set $\bb Q_{x, b}^\alpha(X_0=x, S_0 = b)=1$ and further let the law of $(g_i)_{i \ge 1}$ be given as the projective limit of $(\bb Q_{x, b, n}^\alpha)_{n \geq 1}$, where $\bb Q_{x, b, n}^\alpha$ is given, for any $n \ge 1$ and bounded measurable function $f$ on $\bb M^n$, by 
\begin{align}\label{def:Qsxn}
\int_{\Gamma^n}  f(g_1, \dots, g_n) \, \bb Q_{x, b, n}^\alpha( dg_1, \dots, dg_n)   
:=  \frac{e^{\alpha b}}{r_\alpha(x) \mathfrak m (\alpha)^n }\bb E_{x,b} \bigg[ \sum_{|u| = n}  e^{- \alpha S_u}  r_\alpha(X_u)
f (g_{u|1}, \dots, g_{u|n} ) \bigg], 
\end{align}
see Section \ref{sect:spectralgap} for details.
Define 
\begin{align}\label{def-Sn-Xn}
S_n := S_0 - \log \| g_n \cdots g_1 X_0 \| \  \text{ and } \ X_n:=(g_n \cdots g_1) \cdot X_0. 
\end{align}

We now describe some properties under assumptions \ref{CondiAP} and \ref{Condi_ms}
which in particular means that $\mathfrak m (\alpha) = 1$ and $\mathfrak m'(\alpha)=0$.
It will be shown in Proposition \ref{Prop-LLN-change-of-mea} that  
$\lim_{n \to \infty} \frac{S_n}{n} =0$, $\bb Q_{x,b}^\alpha$-almost surely.
For brevity, we use the notation $\bb Q_{x, 0}^\alpha = \bb Q_{x}^\alpha$. 
By Proposition \ref{Prop-ell-sigma-alpha}, 
the following limit  
\begin{align}\label{def-ell-alpha-x}
\ell_{\alpha}(x):= \lim_{n \to \infty} \bb E_{\bb Q_x^\alpha} S_n
\end{align}
exists for all $x \in \bb S^{d-1}_+$ and defines a continuous function $\ell_{\alpha} \in \scr C(\bb S_+^{d-1})$. 
Moreover, by Corollary \ref{cor:expect-Su-ell-Xu},  
it holds that, for any $x \in \bb S_+^{d-1}$, 
\begin{align}\label{expect-Su-ell-Xu}
\bb E_x \sum_{u=1} (S_u + \ell_{\alpha}(X_u)) e^{-\alpha S_u} r_\alpha(X_u) = r_\alpha(x) \ell_{\alpha}(x).
\end{align}
Using this, the following analogue 
of the \textit{derivative martingale} (see \cite[Section 3.4]{Shi12})  has been introduced in this context in \cite{KM15}: 
\begin{align}\label{def-derivative-martingale}
D_n = \sum_{|u|=n} \big( S_u + \ell_{\alpha} (X_u) \big) e^{- \alpha S_u}  r_{\alpha}(X_u).
\end{align}
Namely, for each $x \in \bb S_+^{d-1}$, this defines a $\bb P_{x}$-martingale with mean $ r_\alpha(x)\ell_{\alpha}(x)$ with respect to the natural filtration $\scr{F}_n$, see Lemma \ref{Lem-derivative-mart}. 
In addition to $\ell_{\alpha}(x)$ and $D_n$, we consider
\begin{align}\label{def-sigma-alpha}
 \sigma^2_{\alpha} := \lim_{n \to \infty} \frac{1}{n} \bb E_{\bb Q_x^\alpha} S_n^2.
\end{align}
By Proposition \ref{Prop-ell-sigma-alpha}, 
this limit is well defined, independent of $x \in \bb S^{d-1}_+$ and strictly positive due to the non-arithmeticity assumption \ref{CondiNonarith}.

Recall that $\Omega^{(1)} = \Omega \times \bb S^{d-1}_+ \times \bb R$. 
For $x \in \bb S^{d-1}_+$ and $b \in \bb R$, 
define a probability measure $\wh{\bb P}_{x,b}$ on $(\Omega^{(1)},\mathscr{F}^{(1)})$ by setting
\begin{align}\label{def-hat-P-x-b-intro}
\wh{\bb P}_{x,b} (A) 
=  \frac{ e^{\alpha b} }{  r_{\alpha}(x)  }  \bb E_{x,b}  \left( W_n\mathds 1_A \right) 
\end{align}
for any $A \in \mathscr F_n$ and $n \geq 0$. 
For brevity, we use the notation 
$\wh{\bb P}_{x} = \wh{\bb P}_{x, 0}$ and for the corresponding expectation we write $\wh{\bb E}_{x} = \wh{\bb E}_{x, 0}$. 
To construct a spinal decomposition of the branching product of random matrices, we want to redefine $ \wh{\bb P}_{x,b}$ on an extended probability space $\Omega^{(2)}:= \Omega^{(1)} \times \bb N^{\bb N}$ to carry a random sequence $w \in \bb N^{\bb N}$, called the {\em spine}, whose law is given by
\begin{align*}
	\wh{\bb P}_{x, b} \big( w|n = z  \big| \mathscr F_n \big)
	:=  \frac{   e^{-\alpha S_z} r_{\alpha} (X_z) }{ W_n },  
\end{align*}
for any $n \in \bb N$ and $z \in \bb T$. By the martingale property of $W_n$ and the definition of $\wh{\bb P}_{x,b}$ in \eqref{def-hat-P-x-b-intro}, this induces a consistent family of distributions. Hence, by Kolmogorov's extension theorem, we obtain the conditional distribution of $w$ given $\mathscr{F}_\infty$, {\em i.e.}, $\wh{\bb P}_{x,b}(w = \cdot \, | \, \mathscr{F}_\infty)$.

\section{Main results}

\subsection{Spinal decomposition}

Our first result proves the existence of a spinal decomposition (cf.\  \cite[Section 4.4]{Shi12}) for matrix branching random walks. Here, we formulate the result for the measure $\wh {\bb P}_{x,b}$ introduced above. In Section \ref{sect:spinal-decomp} we will formulate a general result that also includes a setup where we condition to stay nonnegative. There, we also show that the process with a spine can be constructed by first generating the spine particles and their brothers according to a size-biased point process, and thereupon, attaching unbiased processes to each of the children outside the spine. See Theorem \ref{Thm-spinal-decom} for details.

\begin{theorem}\label{Thm-spinal-decom-intro}
Assume conditions \ref{CondiAP} and \ref{CondiNonarith}. In addition, assume $\bb E N>1$
and that there is $\alpha \in I^\circ$ such that $\mathfrak{m}(\alpha)=1$. 
Then the process $(X_{w|n}, S_{w|n})_{n \geq 0}$ under $\wh{\bb P}_{x}$ is distributed as the Markov random walk
$(X_n, S_n)_{n \geq 0}$, under the measure $\bb Q_{x}^{\alpha}$. 
\end{theorem}

\subsection{Biggins' martingale convergence theorem}

Using the spinal decomposition theorem (cf.\ Theorem \ref{Thm-spinal-decom}), 
we establish the following analogue of Biggins' martingale convergence theorem
for the additive martingale $W_n(s)$ defined by \eqref{def-addi-martigale}. Let
\begin{align}\label{def-psi-s}
\mathfrak M(s) =  \log \mathfrak m(s),  \quad  s \in  I. 
\end{align}

\begin{theorem}\label{Thm:Biggins-bb}
Assume condition \ref{Condi-Furstenberg-Kesten} and $\bb E N >1$. 
Then, for any $x \in \bb S_+^{d-1}$ and $s \in I$,
\begin{align*}
\bb E_{x} (W_{\infty}(s)) = 1  
&  \Leftrightarrow  W_{\infty}(s) > 0, \quad  \bf P_{x}\mbox{-a.s.}  \notag\\ 
&  \Leftrightarrow 
\mathfrak M(s) > s \mathfrak M'(s) \mbox{ and } 
\bb E_x ( W_1(s) \log^+ W_1(s)) < \infty. 
\end{align*}
\end{theorem}

\subsection{Derivative martingale and Seneta-Heyde scaling}

Our next result gives the almost sure convergence of the derivative martingale $D_n$ defined by \eqref{def-derivative-martingale}.

\begin{theorem}\label{Thm-conv-deriv-mart}
Under conditions  \ref{Condi-Furstenberg-Kesten}, \ref{CondiNonarith}, \ref{Condi_ms} and \ref{Condi_harmonic}, 
there exists a non-negative random variable $D_{\infty}$ such that
\begin{align*}
\lim_{n \to \infty} D_n = D_{\infty}, \quad \bb P_{x}\mbox{-a.s.,}
\end{align*}
for any $x \in \bb S_+^{d-1}$.
Assuming in addition \ref{condi:momentsW1},  
we have $D_{\infty} > 0$, $\bf P_{x}$-a.s., for any $x \in \bb S_+^{d-1}$. 
\end{theorem}

Theorem \ref{Thm-conv-deriv-mart} is proved in Section \ref{Sec-conver-deriva-mart}. 
In \cite[Theorem 2.6]{KM15}, this result was proved under the additional assumptions that $\bb E N <\infty$ and particularly, that the law of $g_1$ has a density with respect to Lebesgue measure on $\bb R^{d^2}$.
Having established the existence of $D_\infty$, we then prove that it appears as the limit of the rescaled additive martingale in the boundary case. Namely, we prove a version of the Seneta-Heyde scaling for matrix branching random walks. 

\begin{theorem}\label{Thm-Seneta-Heyde}
Assume conditions \ref{Condi-Furstenberg-Kesten}, \ref{CondiNonarith}, \ref{Condi_ms}, \ref{Condi_harmonic} and \ref{condi:momentsW1}. 
Then, we have  
\begin{align*}
\lim_{n \to \infty} n^{1/2} W_n = \left( \frac{2}{\pi \sigma_\alpha^2} \right)^{1/2} D_{\infty}  
\end{align*}
in probability with respect to $\bf P_{x}$, for any $x \in \bb S_+^{d-1}$. 
\end{theorem}

The proof of Theorem \ref{Thm-Seneta-Heyde} is given in Section \ref{Sec-Seneta-Heyde}. 
For the one-dimensional branching random walk, this result was first proved by A\"id\'ekon and Shi \cite[Theorem 1.1]{AS14}. Our method of proof will follow the arguments given by Boutaud and Maillard \cite{BM19} that allow us to avoid second-moment estimates and peeling lemmas. Instead, we need to prove truncated first moment estimates (cf.\  Proposition \ref{Prop-Wn-ee})
and estimates for the renewal function (Green function) of products of random matrices, conditioned to stay positive 
(cf.\ Theorem \ref{Thm-Green-function}).

\subsection{Duality and renewal theory for products of random matrices}

The duality lemma for a (one-dimensional) random walk $(R_n)_{n \geq 1}$ \cite[XII.2]{Feller71} states that the renewal measure for the (ascending) ladder height process of an interval $I \subset \bb{R}$ equals the expected number of visits of $R_n$ to $I$ before $R_n$ becomes negative. Note that this property heavily relies on the independence of the increments, a property that is lost when considering (logarithms of norms of) products of random matrices; which exhibit a Markovian dependence. 

In our paper, we prove {\em approximative} duality results for {\em oscillating} Markov random walks generated by products of random matrices. To be more precise, we can bound the renewal measure of $I$ of the ladder height process by the expected number of visits to $I$ before becoming negative, and vice versa, see Lemma \ref{lem:dualityrenewal}.

Since duality results are useful tools in many areas of applied probability, we formulate these results in a general setup, that includes non-branching products of random matrices.
To be precise, we generalize  \ref{Condi_ms} to
\begin{conditionAstar}{Condi_ms}\label{Condi_ms2}
	It either holds $N \equiv 1$ a.s. or $\bb E N>1$.  In addition, there is  a constant $\alpha \in I^{\circ}$ 
	such that 
	$\mathfrak m (\alpha)  = 1$ and $\mathfrak m '(\alpha) = 0$. 
\end{conditionAstar}

The case $N \equiv 1$ corresponds to the non-branching case. Under \ref{Condi_ms2}, the process $S_n$ is oscillating under the measure $\bb Q_x^\alpha$. For $y \in \bb R$, let 
$$\tau_y^- = \inf \{ k \geq 1: y + S_k < 0 \}$$
be the first time that $(y+S_k)_{k \geq 1}$ becomes negative. Moreover, introduce the weakly ascending ladder times 
by setting $\mathscr T_0=S_0=0$ and
\begin{align}\label{def-T-n-001}
\scr{T}_n = \inf\{ j > \scr{T}_{n-1} \, :\, S_j \geq S_{\scr{T}_{n-1}} \}.
\end{align}

We now state bounds for the renewal function of the ladder process in the oscillating case.

\begin{theorem}\label{Thm-ladder-height}
Assume conditions \ref{Condi-Furstenberg-Kesten}, \ref{CondiNonarith} and \ref{Condi_ms2}. 
Then, there exists $c>0$ and for all $y \geq 0$, there exists $C = C(y) >0$ such that for all $x \in \bb S_+^{d-1}$, $t \in \bb R$ and $a >0$,
\begin{align*}
\bb E_{\bb Q_{x}^\alpha} \left[ \sum_{n=0}^{\tau_{y}^- -1} \mathds{1}_{[t,t+a]} (S_{n})    \right] \leq C \max\{a, 1\},
\qquad  
\bb E_{\bb Q_{x}^{\alpha}} \left[ \sum_{j=0}^\infty \mathds{1}_{[t,t+a]} (S_{\scr{T}_j})    \right]  \leq c \max\{a, 1\}. 
\end{align*}
\end{theorem}

As stated before, a main ingredient in the proof of the Seneta-Heyde scaling, Theorem \ref{Thm-Seneta-Heyde}, will be the following estimate. 

\begin{theorem}\label{Thm-Green-function}
Assume conditions \ref{Condi-Furstenberg-Kesten}, \ref{CondiNonarith} and \ref{Condi_ms2}. 
Let $f: \bb R_+ \to \bb R_+$ be a bounded and non-increasing function satisfying $\int_0^{\infty} y f(y) dy < \infty$. Then, for any $x \in \bb S^{d-1}_+$, we have
\begin{align*}
\lim_{b \to \infty} \frac1b \sum_{n=0}^\infty \bb E_{\bb Q_{x}^\alpha}  \big[ (b+S_n) f(b+S_n) \mathds{1}_{\{\tau_b^- >n\}}  \big] =0.
\end{align*}
\end{theorem}

Theorems \ref{Thm-ladder-height} and \ref{Thm-Green-function} are proved in Section \ref{Sec-Duality}.

\subsection{Applications}  \ 
\medskip

\textit{Fixed point equation of the multivariate smoothing transform - critical case.}
In \cite{KM15,Men16} the following problem (see \cite{ABM12} for an account of the original one-dimensional equation) was considered (adopted to the present notation). 
Given $(G_v)_{|v|=1}$, determine all possible distributions for a nonnegative random vector $Y$, 
such that the following equality holds:
\begin{align}\label{equation-fixed-point}
Y \stackrel{\mathcal{L}}{=} \sum_{|v|=1}^N G_v^* Y_v, 
\end{align}
where $\stackrel{\mathcal{L}}{=}$ denotes the equality in law, 
and $Y_v$ are i.i.d.\ copies of $Y$ that are also independent of $(G_v)_{|v|=1}$ 
and $G_v^*$ is the transpose. 
An equivalent formulation in terms of the Laplace transform $\phi(ry)=\bb E \exp{(- r \langle y, Y\rangle)}$, 
$y \in \bb S_+^{d-1}$, $r \ge 0$, reads
\begin{equation}\label{eq:FPMST-Laplace}
	\phi(ry) = \bb E \Big[ \prod_{|v|=1} \phi(r G_v y) \Big] = \bb E_y \Big[ \prod_{|v|=1} \phi(r e^{-S_v} Y_v) \Big].
\end{equation}
Note that the appearance of the transpose $G_v^*$ in equation \eqref{equation-fixed-point} allows to consider
the original matrices in equation \eqref{eq:FPMST-Laplace}. 	
Using the results of this paper, we obtain the following set of solutions to equation \eqref{eq:FPMST-Laplace}.

\begin{proposition}
	Assume conditions \ref{Condi-Furstenberg-Kesten}, \ref{CondiNonarith}, \ref{Condi_ms} and \ref{Condi_harmonic}. 
	Then, for any $K>0$, the function 
	\begin{equation}\label{eq:solutionSFPE}
		\phi(ry):= \bb E_y \exp(-r^\alpha KD_\infty), \qquad r\ge 0,\ y \in \bb S_+^{d-1},
	\end{equation}
	is the Laplace transform of a random vector on $\bb R_+^d$ and satisfies \eqref{eq:FPMST-Laplace}.
\end{proposition}

\begin{proof}
Recalling the shift notation $[\, \cdot \, ]_v$ from \eqref{eq:shift-operator}, we can rewrite the formula \eqref{def-derivative-martingale} for $D_n$ as follows, by decomposing at the first generation:
	\begin{align*}
		D_{n+1}& = \sum_{|v|=1} \sum_{|u|=n} \big( S_{vu} + \ell_{\alpha} (X_{vu}) \big) e^{- \alpha S_{vu}}  r_{\alpha}(X_{vu})   \notag\\
		& =  \sum_{|v|=1} S_{v}  e^{- \alpha S_{v}}   \Big(  \sum_{|u| = n}   
		e^{- \alpha (S_{vu} - S_{v}) } r_{\alpha}(X_{vu})  \Big) \notag\\
		& \quad +  \sum_{|v|=1} e^{- \alpha S_{v}}  \Big(  \sum_{|u| = n}  \big( (S_{vu} - S_{v})  + \ell_{\alpha} (X_{vu}) \big) 
		e^{- \alpha (S_{vu} - S_{v}) } r_{\alpha}(X_{vu}) \Big)  \\
		& = \sum_{|v|=1} S_v e^{-\alpha S_v} \big[W_n]_v + \sum_{|v|=1} e^{-\alpha S_v} [D_n]_v.
	\end{align*}
	By Theorem \ref{Thm-conv-deriv-mart}, for any $x \in \bb S_+^{d-1}$, the derivative martingale $D_n$ converges to a nonnegative limit $D_\infty$, $\bb P_x$-a.s., while by Theorem \ref{Thm:Biggins-bb}, it holds that $\lim_{n \to \infty} W_n=0$, $\bb P_x$-a.s. under condition \ref{Condi_ms}. As a consequence, the limit $D_\infty$ satisfies the following equation: $\bb P_x$-a.s., 
	\begin{equation}\label{eq:as-identity-D}
		D_\infty = \sum_{|v|=1} e^{-\alpha S_v} [D_\infty]_v.
	\end{equation}
	Hence, for the function $\phi$ given by \eqref{eq:solutionSFPE}, it follows that
	\begin{align*}
		\phi(ry) &= \bb E_y \Big[ \exp (-r^\alpha K \sum_{|v|=1)} e^{-\alpha S_v} [D_\infty]_v ) \Big] =  \bb E_y \Big[ \bb E \big[ \exp (-r^\alpha K \sum_{|v|=1)} e^{-\alpha S_v} [D_\infty]_v ) \big| \mathscr{F}_1 \big] \Big] \\
		& = \bb E_y \Big[ \prod_{|v|=1} \bb E \big[ \exp (-(r e^{- S_v})^\alpha K [D_\infty]_v ) \big| \mathscr{F}_1 \big] \Big] 
		= \bb E_y \Big[ \prod_{|v|=1} \phi(re^{-S_v} X_v) \Big].
	\end{align*}
Thus we have found a (one-parameter) family of solutions to \eqref{eq:FPMST-Laplace}.  By Theorem \ref{Thm-conv-deriv-mart}, the function $\phi$  is well defined and nontrivial; it remains to show  that $\phi$ is indeed the Laplace transform of a probability measure. This is proved as follows. By the arguments given in the proof of \cite[Theorem 5.1]{Men16}, $\phi_n(ry):= \bb E_y \exp(-r^\alpha K \sqrt{\frac12 \pi n \sigma_\alpha^2} W_n)$ is a Laplace transform for any $n \ge 1$. By Theorem \ref{Thm-Seneta-Heyde}, $\lim_{n \to \infty} \phi_n(ry) = \phi(ry)$ for all $r \ge 0$, $y \in \bb S_+^{d-1}$, hence $\phi$ is a Laplace transform as well.
\end{proof}

The fact that $\phi$ as given by \eqref{eq:solutionSFPE} is a solution to \eqref{eq:FPMST-Laplace} was proved in \cite{KM15} subject to strong assumptions, namely that $\mu$ is a finite measure and has a density, or finite support which is moreover concentrated on rank-one matrices. These assumptions were particularly used in the proof for the convergence of $D_n$.

\medskip

\textit{Invariant measure for stochastic recursions on $\bb R_+^d$ in the critical case.}
Let $A$ be a random nonnegative matrix and $B$ a random nonnegative vector. Let $(A_n, B_n)_{n \ge 1}$ be a sequence of i.i.d.~copies of $(A,B)$. The affine stochastic recursion  
$$ X_{n+1}= A_n X_n + B_n$$ has attracted much attention since the seminal paper by Kesten \cite{Kes73}, see e.g. the book \cite{BDM16} for a recent survey. 
In the critical case, which corresponds to our assumption \ref{Condi_ms2} (with $A$  taking the role of $G_1$), the Markov chain $(X_n)_{n \geq 1}$ admits no stationary probability distribution but possesses a unique invariant Radon measure. 
This was proved only very recently in \cite{BPP21}.
As the authors of that paper state (cf. \cite[Remarks after Theorem 1.1]{BPP21}) a crucial ingredient to study further the tails of the invariant measure, by extending the approach taken in \cite[Section 4]{BB15}, would be the development of renewal theory for centered Markov random walks, driven by nonnegative random matrices; a fundamental step towards that direction is provided by our Theorem \ref{Thm-ladder-height}.

\section{Auxiliary statements}


\subsection{Transfer operators}\label{sect:spectralgap}
	
In this section, we provide spectral gap properties of the transfer operator $P_s$ defined in \eqref{Def-Ps}. In particular, we show how the analysis of \cite{BDGM14,BM16,GL16} (for $\mu$ being a probability measure) carries over to our setting (where $\mu$ is $\sigma$-finite).

\begin{lemma}\label{lem:exchangeability}
	Assume that $I$ is nonempty (hence $\mu$ is $\sigma$-finite). Then, it holds that for every $n \in \bb N$ and nonnegative measurable function $f$,
	\begin{align}
		 \bb E \Big[ \sum_{|u|=n} f(g_{u|1}, g_{u|2}, \dots, g_{u|n-1},g_u) \Big] 
		 & = \int f(g_1, \dots, g_n) \mu^{\otimes n}(dg_1, \dots, dg_n)  \label{eq:exchangeability} \\
		 & =  \bb E \Big[ \sum_{|u|=n} f(g_{u}, g_{u|n-1}, \dots, g_{u|2},g_{u|1}) \Big]  \label{eq:exchangeability-002}
	\end{align}
\end{lemma}

\begin{proof}
	Denote by $\scr{B}(\bb M)$ the Borel $\sigma$-field on $\bb M$. Then, for every $n \in \bb N$, the product $\sigma$-field $\scr{B}(\bb M^n)$  is generated by the $\pi$-system  $\{A_1 \times \cdots \times A_n \, : \, A_i \in \scr{B}(\bb M),\ 1\le i \le n\}$ of rectangles.
	Using the monotone class theorem \cite[Theorem 5.2.2]{Dur19}, it suffices to show \eqref{eq:exchangeability} for indicator functions of rectangles.
Recall that, by \eqref{def-filtration-Fn}, $\mathscr{F}_n$ denotes the natural filtration with respect to generations in the matrix branching random walk. 
Recall also that, by \eqref{def-measure-mu}, $\mu(A)=\bb E \big[ \sum_{|u|=1} \mathds{1}_A(g_u) \big]$ for any $A \in \scr{B}(\bb M)$.
Then, by taking conditional expectation, we get 
	\begin{align*}
		&\bb E \Big[ \sum_{|u|=n} \mathds{1}_{A_1}(g_{u|1}) \mathds{1}_{A_2}(g_{u|2})\cdots \mathds{1}_{A_{n-1}}(g_{u|n-1}) \mathds{1}_{A_n}(g_u) \Big] \\
		& = \bb E \bigg[ \bb E \Big[ \sum_{|u|=n-1}  \mathds{1}_{A_1}(g_{u|1}) \mathds{1}_{A_2}(g_{u|2})\cdots \mathds{1}_{A_{n-1}}(g_{u}) \sum_{|w|=1} \mathds{1}_{A_n}(g_{uw}) \Big| \mathscr{F}_{n-1} \Big]\bigg] \\
		& = \bb E \bigg[   \sum_{|u|=n-1}  \mathds{1}_{A_1}(g_{u|1}) \mathds{1}_{A_2}(g_{u|2})\cdots \mathds{1}_{A_{n-1}}(g_{u}) 
		 \bb E \Big[\sum_{|w|=1} \mathds{1}_{A_n}(g_{uw}) \Big| \mathscr{F}_{n-1} \Big] \bigg] \\
		& = \bb E \bigg[   \sum_{|u|=n-1}  \mathds{1}_{A_1}(g_{u|1}) \mathds{1}_{A_2}(g_{u|2})\cdots \mathds{1}_{A_{n-1}}(g_{u}) \mu(A_n) \bigg] \\
		& = \dots = \mu(A_1)\mu(A_2)\cdots \mu(A_n). 
	\end{align*}
This shows \eqref{eq:exchangeability}. 
	If we start with the reversed order of the sets at the beginning, 
	we also arrive at $\prod_{i=1}^n \mu(A_i)$, hence the second identity \eqref{eq:exchangeability-002} holds as well.
\end{proof}
	
The following corollay gives the $n$-th iterate of the operator $P_s$ defined by \eqref{Def-Ps}. 
Write $G_n:=g_n \cdots g_1$. 

\begin{corollary}\label{lem:Psn} 
For any $s \in I$, 
$n \geq 1$ and bounded measurable function $\varphi$ on $\bb S_+^{d-1}$, it holds
\begin{align}\label{eq:Psn}
P_s^n \varphi(x) = \int_{\Gamma^n} e^{s \sigma (G_n,x)} \varphi(G_n \cdot x) \mu^{\otimes n}(dg_1, \dots, dg_n) 
=  \bb E \bigg( \sum_{|u| = n}  e^{s \sigma (G_u, x)} \varphi(G_u \cdot x) \bigg).
\end{align}
\end{corollary}

\begin{proof} We prove the first identity by induction, the second one being a direct consequence of Lemma \ref{lem:exchangeability}.
 The case $n=1$ is the definition of $P_s$ (see \eqref{Def-Ps}). 
Assume \eqref{eq:Psn} holds for some $n \geq 1$, 
then, using the cocycle property \eqref{eq-cocycle-property}, we have
\begin{align*}
P_s^{n+1} \varphi(x) & = \int_{\Gamma}  e^{s \sigma(g_0,x)} P_s^n \varphi(g_0 \cdot x) \mu(d g_0) \\ 
& =  \int_{\Gamma}  e^{s \sigma(g_{0},x)} \left[ \int_{\Gamma^n} e^{s \sigma (G_n,g_0 \cdot x)} \varphi(G_n \cdot (g_0 \cdot x)) \mu^{\otimes n}(dg_1, \dots, dg_n) \right] \mu(d g_0) \\
& =  \int_{\Gamma^{n+1}} e^{s \sigma (G_n g_0,x)} \varphi(G_n g_0 \cdot x) \mu^{\otimes n+1}(dg_0, dg_1, \dots, dg_n). 
\end{align*}
This completes the proof of the corollary. 
\end{proof}

To state the main properties of $P_s$, we will also study the conjugate transfer operator $P_s^*$ defined as follows: 
for any bounded measurable function $\varphi$ on $\bb S_+^{d-1}$ and $x \in \bb S_+^{d-1}$, 
\begin{equation}\label{eq:defnPsstar}
P_s^* \varphi(x) := \bb E \Big( \sum_{|u| = 1}  e^{s \sigma (G_u^*, x)} \varphi(G_u^* \cdot x) \Big) 
 = \int_{\Gamma} e^{s \sigma(g^*,x)} \varphi(g^* \cdot x) \mu(dg),
\end{equation}
where $G_u^*$ denotes the transpose of $G_u$.

Recall from \eqref{def-VGamma-pos} the definition of $\Lambda (\Gamma)$, 
and note that $\Lambda (\Gamma)$ is the unique minimal $\Gamma$-invariant subset of $\bb S_+^{d-1}$
(see \cite[Lemma 4.2]{BDGM14}). In the same way, define $\Gamma^*$ to be the smallest closed subsemigroup of nonnegative $d \times d$ matrices that covers $\{G_u^* \, : \, |u|=1\}$ with $\bb P$-probability 1. Then $\Lambda(\Gamma^*)$ is the unique minimal $\Gamma^*$-invariant subset of $\bb S_+^{d-1}$.

\begin{lemma}\label{lem:ExistenceEigenfunctionsPs}
	Assume condition \ref{CondiAP} and let $s \in I$.
	Then the spectral radius of $P_s$ equals 
	$$m(s):= \lim_{n \to \infty} \bigg( \bb E \Big( \sum_{|u| = n}  \norm{G_u}^s \Big) \bigg)^{1/n}$$
	 and there is an eigenfunction $r_s \in \scr C(\bb S_+^{d-1})$ and a probability eigenmeasure $\nu_s$ on $\bb S^{d-1}_+$ with
	$$ P_s r_s = \mathfrak{m}(s) r_s, \quad P_s \nu_s = \mathfrak{m}(s)\nu_s.$$
	The function $r_s$ is strictly positive and $\min\{s,1\}$-H\"older with respect to the Euclidean distance $|\cdot|$. 
	The support of $\nu_s$ contains $\Lambda(\Gamma)$.
	Similarly, the spectral radius of $P_s^*$ equals $\mathfrak{m}(s)$ and there is a strictly positive $\min\{s,1\}$-H\"older continuous eigenfunction $r_s^*$ on $\bb S^{d-1}_+$ and a probability eigenmeasure $\nu_s^*$ on $\bb S^{d-1}_+$ such that 
	$$ P_s^* r_s^* = \mathfrak{m}(s) r_s^*, \quad P_s^* \nu_s^* = \mathfrak{m}(s)\nu_s^*.$$
\end{lemma}	

\begin{proof}
This can be proved along the same lines as \cite[Proposition 4.4]{BDGM14}, so we only give an outline and sketch the differences. It is crucial that for $s \in I$, the operators $P_s$ and $P_s^*$ are bounded operators. 
	
	By an appeal to the Schauder-Tychonoff theorem for the nonlinear operator $\wt{P}_s^* \nu^*:= P_s^* \nu^* / (P_s^*\nu^*)(1)$, there is an invariant probability measure $\nu_s^*$ for $\wt{P}_s^*$, which in turn is an eigenmeasure of $P_s^*$. We denote the corresponding eigenvalue by $k(s)>0$ and we will show later that this is the dominant one.
	Upon defining
	$$ r_s(x) := \int_{ \bb S^{d-1}_+ } \langle x, y \rangle^s \nu_s^*(dy), $$
	we first observe that $r_s(x)>0$ for all $x \in \bb S^{d-1}_+$. To wit, note that the support of $\nu_s^*$ is $\Gamma^*$-invariant, hence contains $\Lambda(\Gamma^*)$. Note that $\Gamma^*$ satisfies \ref{CondiAP} as well, hence there is (by the existence of a positive matrix in $\Gamma^*$) a strictly positive vector in $\Lambda(\Gamma^*)$, which implies $r_s(x)>0$ for all $x \in \bb S^{d-1}_+$. The expression for $r_s$ also shows the asserted H\" older continuity.
	
The following calculation shows that $r_s$ is an eigenfunction of $P_s$ with eigenvalue $\mathfrak{m}(s)$: 
\begin{align*}
P_s r_s(x) 
& = \bb E \Big[ \sum_{|u|=1} \norm{G_u x}^s  \int_{ \bb S^{d-1}_+ }  \langle G_u \cdot x,y \rangle^s \nu_s^*(dy) \Big] 
= \bb E \Big[ \sum_{|u|=1}  \int_{ \bb S^{d-1}_+ }  \langle G_u x,y \rangle^s \nu_s^*(dy) \Big]  \notag\\
& = \bb E \Big[ \sum_{|u|=1}  \int_{ \bb S^{d-1}_+ }  \langle x, G_u^* y \rangle^s \nu_s^*(dy) \Big] 
 = \bb E \Big[ \sum_{|u|=1} \norm{G_u^* y}^s \int_{ \bb S^{d-1}_+ }  \langle x, G_u^* \cdot y \rangle^s \nu_s^*(dy) \Big]  \notag\\
& = \int_{ \bb S^{d-1}_+ }  P_s^* \big( \langle x, \cdot \rangle^s\big)(y) \nu_s^*(dy)  
 =  \int_{ \bb S^{d-1}_+ }   \langle x, y \rangle^s (P_s^*\nu_s^*)(dy) = k(s) r_s(x). 
\end{align*}

Turning to the identification $\mathfrak{m}(s)=k(s)=m(s)$, we first show that $m(s)$ is well defined. Upon defining $Y_n:=\sum_{|u|=n} \norm{G_u}^s$, we note that by submultiplicativity of the matrix norm, 
\begin{align*}
	Y_{n+m} = \sum_{|u|=n} \sum_{|v|=m} \norm{G_{uv}}^s \leq \sum_{|u|=n} \sum_{|v|=m} \norm{G_u}^s \norm{[G_v]_u}^s = \sum_{|u|=n}  \norm{G_u}^s \sum_{|v|=m}\norm{[G_v]_u}^s.
\end{align*}
Hence, $ \bb E Y_{n+m} \leq \bb E Y_n \cdot \bb E Y_m$ and thus the sequence $y_n:=\log \bb E Y_n$ is subadditive 
and converges to $\inf_{n \geq 1} \frac{y_n}{n}$ by Fekete's subadditivity lemma.

Next, by Lemma \ref{lem:Psn}, we can bound the norm of $P_s^n$ as follows:
\begin{align*}
\abs{P_s^n \varphi(x)} 
\leq  \norm{\varphi}_\infty \bb E \Big[  \sum_{|u| = n} \norm{G_u x}^s\Big] 
\leq \norm{\varphi}_\infty \bb E \Big[  \sum_{|u| = n} \norm{G_u}^s\Big].
\end{align*}
Hence, by the Gelfand formula, the spectral radius $\mathfrak{m}(s)$ of $P_s$ satisfies
$$ k(s) \leq \mathfrak{m}(s) = \lim_{n \to \infty} \norm{P_s^n}^{1/n} \leq \lim_{n \to \infty} \bigg(\bb E \Big[  \sum_{|u| = n} \norm{G_u}^s\Big]\bigg)^{1/n}=m(s).$$
This shows in particular that $m(s)>0$. 

To show the equality, it now suffices to prove that $k(s) \geq m(s)$. We use \cite[Lemma 4.5]{BDGM14} which shows that there is a constant $d_s>0$ such that for all $g \in \bb M$, 
$$ \int_{ \bb S^{d-1}_+ } \norm{g^*y}^s \nu_s^*(dy) \geq d_s \norm{g}^s.$$
Consequently,
\begin{align*}
	k(s)^n =& (P_s^* \nu_s^*)(1) = 
	\bb E \Big[ \sum_{|u|=n} \int_{\bb S^{d-1}_+} \norm{G_u^*y}^s \nu_s^*(dy) \Big] 
	\geq d_s \bb E \Big[ \sum_{|u|=n}  \norm{G_u}^s  \Big], 
\end{align*}
thus $k(s)\ge m(s)$.

In the same way, the existence of $\nu_s$ and $r_s^*$ is proved. Finally, note that as well $m(s) = \lim_{n \to \infty} \big( \bb E \big[ \sum_{|u| = n}  \norm{G_u^*}^s \big] \big)^{1/n}$, hence the spectral radii of $P_s$ and $P_s^*$ coincide by the arguments given above.
\end{proof}

In order to show that $\mathfrak{m}(s)$ is an algebraically simple eigenvalue, we consider the following
Markov operator on $\scr C(\bb S_+^{d-1})$, defined by
$$ 
Q_s \varphi(x) := \frac{1}{\mathfrak{m}(s) r_s(x)} P_s \big(\varphi r_s \big)(x). 
$$
By Theorem 4.11 in \cite{BDGM14},  the probability measure 
\begin{align} \label{invar mes for Q_s-001}
\pi_s(dx) = \frac{r_s(x)}{ \mathfrak m (s)} \nu_s (dx) 
\end{align}
is the unique invariant measure for the Markov operator  $Q_s$.  
Upon setting
\begin{equation}\label{eq:qnsx}
	 q_n^s(x,g):= \frac{|gx|^s r_s(g \cdot x)}{\mathfrak{m}(s)^n r_s(x)},
\end{equation}
we can show as in the proof of Corollary \ref{lem:Psn} that
$Q_s^n \varphi(x) = \int q_n^s(x,G_n) \mu^{\otimes n}(dg_1, \dots, dg_n).$ It follows that $q_n^s(x,\cdot) \mu^{\otimes n}$ is a projective system, which allows to apply Kolmogorov's extension theorem to obtain a  probability measure $\bb Q_x^s$  on $\bb M^{\bb N}$. This generalizes the definition of $\bb Q_x^\alpha$ in \eqref{def:Qsxn} to the case of general $s \in I$. Recalling the definition of $(X_n,S_n)$ from \eqref{def-Sn-Xn}, we fix initial values by using the convention $\bb Q_{x,b}^s(X_0=x, S_0=b)=1$, with the identification $\bb Q_x^s=\bb Q_{x,0}^s$.

As a consequence of Lemma \ref{lem:exchangeability}, we have the following identity.

\begin{corollary}[Many-to-one formula]\label{cor:many-to-one}
For any  $s \in I$ and $x \in \bb S_+^{d-1}$, $b \in \bb R$,   
$n \geq 1$ and any bounded measurable function $f: (\bb S_+^{d-1} \times \bb R)^{n+1} \to \bb R$, 
\begin{align}\label{Formula_many_to_one}
	&  \bb E_{x,b} \bigg[ \sum_{|u| = n}  
	f \left( X_{\o},  S_{\o},  X_{u|1}, S_{u|1}, \ldots, X_u, S_u \right) \bigg]  \notag\\
	& = r_{s}(x)  \mathfrak m(s)^n  
	\bb E_{\bb Q_{x, b}^s} \bigg[ \frac{1}{r_{s} (X_n)} e^{ s (S_n - S_0) } f \big( X_0, S_0, X_1, S_1, \ldots, X_n, S_n \big) \bigg]  
\end{align}
or, equivalently,
\begin{align}\label{Formula_many_to_oneII}
	&  \bb E_{x,b} \bigg[ \sum_{|u| = n}  \frac{e^{-s (S_u-S_{\o})}r_s(X_u)}{\mathfrak{m}(s)^n r_{s} (X_{\o})}
	f \left( X_{\o},  S_{\o},  X_{u|1}, S_{u|1}, \ldots, X_u, S_u \right) \bigg]  \notag\\
	& =  
	\bb E_{\bb Q_{x, b}^s} \bigg[ f \big( X_0, S_0, X_1, S_1, \ldots, X_n, S_n \big) \bigg].  
\end{align}
\end{corollary}

\begin{proposition}\label{prop:Ps} 
Assume condition \ref{CondiAP}. Let $s \in I$.
The spectral radius $\mathfrak{m}(s)$ is the  unique dominant eigenvalue of $P_s$; 
any other eigenvalue $\lambda$ of $P_s$ is strictly smaller in absolute value. 
The eigenspace corresponding to $\mathfrak{m}(s)$ is one-dimensional, 
hence $r_s$ is unique up to scaling. The probability measure $\nu_s$ is unique and its support equals  $\Lambda(\Gamma)$.	
\end{proposition}

\begin{proof}[Source]
	This is proved in Section 4.4 of \cite{BDGM14} in the case where $\mu$ is a probability measure, but this property is in fact not necessary for the proofs to work. What is needed is that $\bb Q_x^s$ is a probability measure, and that that $\int \norm{g}^s \mu(dg)< \infty$, which is true for any $s \in I$.
\end{proof}

Recalling the definition of the Markov random walk $(X_n, S_n)$ from \eqref{def-Sn-Xn}, the following result shows that $\bb Q_{x}^s$ changes the drift of $S_n$. In particular, under \ref{Condi_ms}, $S_n$ is centered.

\begin{proposition}\label{Prop-LLN-change-of-mea}
Assume condition \ref{CondiAP}. 
Then the function $s \mapsto \mathfrak{m}(s)$ is differentiable on $I^\circ$. 
Moreover, for any $s \in I^\circ$ satisfying $\int \|g\|^s (\log \|g\| + \log \iota(g)) \mu(dg) < \infty$,
 it holds $\lim_{n \to \infty} \frac{S_n}{n} =\frac{\mathfrak{m}'(s)}{\mathfrak{m}(s)}$, $\bb Q_{x}^s$-almost surely.
\end{proposition}

\begin{proof}[Source]
This is proved in \cite[Theorem 6.1]{BDGM14}. Again, the arguments carry over given that $\bb Q_x^s$ is a probability measure.
\end{proof}

The following lemma shows that the matrix norm $\|g\|$ and the vector norm $|gx|$ are comparable under condition \ref{Condi-Furstenberg-Kesten}. 

\begin{lemma}[\cite{GX23}] \label{lemma kappa 1}
Assume \ref{Condi-Furstenberg-Kesten}. 
Then we have, for any $g\in \Gamma$ and $x \in \bb S_+^{d-1}$, 
\begin{align*}
\frac{1}{\varkappa^2} \|g\| \leq  \| gx \| \leq  \|g\|, 
\end{align*}
where $\varkappa > 1$ is given by condition \ref{Condi-Furstenberg-Kesten}. 
\end{lemma}

\begin{remark}\label{Rem-moment-condi}
If condition \ref{Condi-Furstenberg-Kesten} holds, then $\varkappa^{-2} \|g\| \leq \iota(g) \leq \|g\|$ by Lemma \ref{lemma kappa 1}. 
Therefore, under condition \ref{Condi-Furstenberg-Kesten}, $\int \|g\|^s (\log \|g\| + \log \iota(g)) \mu(dg) < \infty$ is always satisfied for $s \in I^{\circ}$. 
\end{remark}

\begin{proposition}\label{Prop-ell-sigma-alpha}
		Assume condition \ref{CondiAP}, $s \in I^\circ$ and that $\int \|g\|^s \log \iota(g)  \mu(dg) < \infty$. 
		Writing $q_s:=\frac{\mathfrak{m}'(s)}{\mathfrak{m}(s)}$, the following quantities are well defined for all $x \in  \bb S_+^{d-1}$:
	 $$ \ell_{s}(x):= \lim_{n \to \infty} \bb E_{\bb Q_x^s} (S_n - nq_s), 
	 \qquad \sigma^2_{s} := \lim_{n \to \infty} \frac{1}{n} \bb E_{\bb Q_x^s} (S_n-nq_s)^2,$$
	 where $\sigma_s^2$ does not depend on $x$. 
	 In addition, it holds that $\ell_s \in \scr C(\bb S_+^{d-1})$ and 
	 $$ \ell_s(x) = \bb E_{\bb Q_x^s} \big( (S_1 - q_s) + \ell_s(X_1) \big).$$
	 Moreover, if \ref{CondiNonarith} holds, then $\sigma_s^2$ is strictly positive.
\end{proposition}

\begin{proof}[Source]
	This proposition is proved in \cite[Lemmas 7.1 and 7.2]{BM16} in the case when $\mu$ is a probability measure. 
	Similar results can be obtained in our setting since the proofs are carried out using holomorphic perturbation theory for operators $Q_s$. The perturbation theory is based on \cite[Lemma 6.1]{BM16}, which does not require finiteness of $\mu$, only that $s \in I^\circ$. 
	In addition, one can also establish an analogue of \cite[Proposition 4.2]{BM16}, which gives the spectral gap properties
	of the Markov operator $Q_s$. 
\end{proof}

\begin{corollary}\label{cor:expect-Su-ell-Xu}
Assume \ref{CondiAP} and \ref{Condi_ms}. Then, for any $x \in \bb S_+^{d-1}$,
\begin{align*}
\bb E_x \sum_{u=1} (S_u + \ell_{\alpha}(X_u)) e^{-\alpha S_u} r_\alpha(X_u) = r_\alpha(x) \ell_{\alpha}(x).
\end{align*}
\end{corollary}

\begin{proof}
The identity is obatined by combining the previous propositions with the many-to-one formula \eqref{Formula_many_to_oneII}.
\end{proof}

\subsection{Harmonic function and change of measure}\label{sec-change-proba}

Let $B$ be a measurable set of $\bb R$ such that for any $x \in \bb S_+^{d-1}$ and $b \in B$,
\begin{align}\label{condition-on-B}
\bb P_{x, b} (S_1 \in B) > 0. 
\end{align}
Let $h$ be a positive harmonic function for $(X_n, S_n)$ killed outside the set $\bb S_+^{d-1} \times B$, i.e. 
$h > 0$ and for any $x \in \bb S_+^{d-1}$ and $b \in B$, 
\begin{align}\label{harmonicity-h}
\bb E_{\bb Q_{x, b}^{\alpha}} \left[ h (X_1, S_1) \mathds 1_{ \{ S_1 \in B \} } \right] = h (x, b).
\end{align}
We denote, for any $y \in \bb S_+^{d-1}$ and $s \in \bb R$, 
\begin{align}\label{def-H-alpha-ys}
H_{\alpha}(y, s) = r_{\alpha}(y)  h(y,s)  e^{- \alpha s}. 
\end{align}
Define 
\begin{align}\label{def-martingale-Dn}
M_n^h:= \sum_{ |u| = n }    H_{\alpha} \left( X_u, S_u \right)  \mathds{1}_{ \left\{ S_{u|k} \in B,  \, \forall  \,   0 \leq k \leq n  \right\} },  
\quad  n \geq 0. 
\end{align}
By definition, for any $x \in \bb S_+^{d-1}$ and $b \in B$, under the measure $\bb P_{x, b}$, 
we have $M_0^h = H_{\alpha}(x, b)$. 
The following result shows that $(M_n^h,  \mathscr F_n)_{n \geq 0}$ is a martingale 
with $\mathscr F_n$ introduced in \eqref{def-filtration-Fn}, 
which will be used to perform a change of probability measure. 

\begin{lemma}\label{Lem-martiangle-Dn}
Assume that there exists a constant $\alpha \in I^{\circ}$ such that $\mathfrak m (\alpha)  = 1$. 
Then, for any $x \in \bb S_+^{d-1}$ and $b \in B$, 
we have that $(M_n^h,  \mathscr F_n)_{n \geq 0}$ is a martingale with respect to the probability measure $\bb P_{x, b}$. 
\end{lemma}

\begin{proof}
Since for any $n \geq 0$, 
\begin{align*}
M_{n+1}^h 
& = \sum_{ |v| = n }  \,  \sum_{u: |u| = n + 1,  \overset{\leftarrow}{u} = v }  H_{\alpha} \left( X_u, S_u \right)  
   \mathds{1}_{ \left\{ S_{u|k} \in B,  \, \forall  \,   0 \leq k \leq n + 1 \right\} }  \notag\\
& = \sum_{ |v| = n }  \mathds{1}_{ \left\{ S_{v|k} \in B,  \, \forall  \,   0 \leq k \leq n  \right\} }
   \sum_{u: |u| = n + 1,  \overset{\leftarrow}{u} = v }  H_{\alpha} \left( X_u, S_u \right)  
     \mathds{1}_{ \left\{ S_u \in B  \right\} },
\end{align*}
we get that for any $x \in \bb S_+^{d-1}$ and $b \in B$, under the measure $\bb P_{x, b}$, 
\begin{align}\label{martingale-pf-ab}
\bb E_{x,b} \left( \left. M_{n+1}^h  \right|  \mathscr F_n \right)   
 =  \sum_{ |v| = n }  \mathds{1}_{ \left\{ S_{v|k} \in B,  \, \forall  \,   0 \leq k \leq n  \right\} }
\bb E_{ X_v, S_v }  \bigg(  \sum_{|w| = 1}  H_{\alpha} \left( X_w, S_w \right)  
    \mathds{1}_{ \left\{ S_w \in B  \right\} }   \bigg).  
\end{align}
By the many-to-one formula \eqref{Formula_many_to_one}, the assumption $\mathfrak m (\alpha)  = 1$, 
 \eqref{def-H-alpha-ys} 
 and the harmonicity of the function $h$ (cf.\ \eqref{harmonicity-h}), 
we get that for any $x \in \bb S_+^{d-1}$ and $b \in B$,
\begin{align*}
 \bb E_{ x, b }  \bigg(  \sum_{|w| = 1}  H_{\alpha} \left( X_w, S_w \right)  
   \mathds{1}_{ \left\{ S_w \in B  \right\} }   \bigg)   
& = r_{\alpha}(x)   \bb E_{ \bb Q^{\alpha}_{x,b} } 
      \left[ \frac{1}{ r_{\alpha}(X_1) }  e^{- \alpha (S_1 - S_0) }   H_{\alpha}(X_1, S_1) \mathds 1_{ \{ S_1 \in B \} } \right]   \notag\\
& = r_{\alpha}(x)   \bb E_{ \bb Q^{\alpha}_{x,b} } \left[ h(X_1, S_1) \mathds 1_{ \{ S_1 \in B \} } \right]  e^{-\alpha b}  \notag\\
& = r_{\alpha}(x)  h(x, b) e^{-\alpha b}.  
\end{align*}
Applying this to \eqref{martingale-pf-ab} completes the proof of the lemma. 
\end{proof}

From Lemma \ref{Lem-martiangle-Dn},  under the measure $\bb P_{x, b}$, 
we have that $(M_n^h,  \mathscr F_n)_{n \geq 0}$ is a martingale with the mean 
\begin{align}\label{martingale-property-Dn}
\bb E_{x, b} (M_n^h) =  H_{\alpha}(x, b).  
\end{align}
By Kolmogorov's extension theorem, 
there exists a unique probability measure $\wh{\bb P}^h_{x, b}$ on $\mathscr F_{\infty}$ such that, 
for any $A \in \mathscr F_n$ and $n \geq 0$, 
\begin{align}\label{def-hat-P-x-b}
\wh{\bb P}^h_{x, b} (A) 
=  \frac{1}{  H_{\alpha}(x, b) }  \bb E_{x,b}  \left( M_n^h\mathds 1_A \right). 
\end{align}
Denote by $\wh{\bb E}^h_{x, b}$ the corresponding expectation. 
This is consistent with the definition of $\wh{\bb P}_{x, b}$ (cf.\ \eqref{def-hat-P-x-b-intro}) which corresponds to $\wh{\bb P}^h_{x, b}$ with $h \equiv1$ and $B=\bb R$. 
Note that $M_n^h> 0$, $\wh{\bb P}^h_{x, b}$-a.s.

Since $(M_n^h,  \mathscr F_n)_{n \geq 0}$ is a non-negative martingale under $\bb P_{x, b}$, 
there exists a random variable $M_{\infty}^h \geq 0$ such that $\lim_{n \to \infty} M_n^h = M_{\infty}^h$, $\bb P_{x, b}$-a.s.

\begin{lemma}\label{lem-condi-limsup}
If there exists a $\sigma$-field $\mathscr G \subset \mathscr F$ such that 
\begin{align}\label{property-Dn}
\liminf_{n \to \infty} \wh{\bb E}^h_{x, b} \left( \left. M_n^h  \right|  \mathscr G  \right)  < \infty,  \quad   \wh{\bb P}_{x, b}\mbox{-a.s.}
\end{align}
then the $\bb P_{x, b}$-martingale $(M_n^h, \mathscr F_n)_{n \geq 0}$ is uniformly integrable.
In particular, we have 
\begin{align}\label{expect-Dn-uniform}
\bb E_{x,b} M_{\infty}^h = H_{\alpha}(x, b) = r_{\alpha}(x) h(x, b) e^{-\alpha b}. 
\end{align}
\end{lemma}

\begin{proof}
We first show that $(\frac{1}{M_n^h}, \mathscr F_n)_{n \geq 0}$ is a $\wh{\bb P}^h_{x, b}$-supermartingale. 
Indeed, since $M_n^h> 0$, $\wh{\bb P}^h_{x, b}$-a.s., using \eqref{def-hat-P-x-b} and Lemma \ref{Lem-martiangle-Dn}, 
we get that for any $n \geq j \geq 0$ and $A \in \mathscr F_j$, 
\begin{align*}
\wh{\bb E}^h_{x, b} \left( \frac{1}{M_n^h} \mathds 1_A \right)
& = \wh{\bb E}^h_{x, b} \left( \frac{1}{M_n^h} \mathds 1_A; M_n^h>0 \right)
 =  \frac{1}{  H_{\alpha}(x, b) }  \bb P_{x,b} \left( M_n^h> 0; A \right)   \notag\\
&  \leq  \frac{1}{  H_{\alpha}(x, b) }  \bb P_{x,b} \left( M_j^h > 0; A \right)  
 = \wh{\bb E}^h_{x, b} \left( \frac{1}{M_j^h} \mathds 1_A \right). 
\end{align*}
Therefore, there exists a finite random variable $M^h_{\infty} \geq 0$ such that 
$\lim_{n \to \infty} \frac{1}{M_n^h} = M^h_{\infty}$, $\wh{\bb P}^h_{x, b}$-a.s.
Since $M_n^h> 0$, $\wh{\bb P}^h_{x, b}$-a.s., we get $\lim_{n \to \infty} M_n^h= \frac{1}{M^h_{\infty}}$, $\wh{\bb P}^h_{x, b}$-a.s.

By the conditional Fatou's lemma and \eqref{property-Dn}, we get that, $\wh{\bb P}^h_{x, b}$-a.s. 
\begin{align*}
\wh{\bb E}^h_{x, b}  \left( \frac{1}{M_{\infty}} \Big|   \mathscr G  \right)
=  \wh{\bb E}^h_{x, b} \left( \lim_{n \to \infty} M_n^h \Big|  \mathscr G  \right)
\leq \liminf_{n \to \infty} \wh{\bb E}^h_{x, b} \left( M_n^h \Big|  \mathscr G  \right)
 < \infty,
\end{align*}
so that $\wh{\bb E}^h_{x, b} ( \frac{1}{M_{\infty}} ) < \infty$. 
Since $M^h_{\infty} \geq 0$, this implies that $\frac{1}{M^h_{\infty}} < \infty$, $\wh{\bb P}^h_{x, b}$-a.s., 
and hence $\sup_{n \geq 0} M_n^h< \infty$, $\wh{\bb P}^h_{x, b}$-a.s.

Using \eqref{def-hat-P-x-b}, the fact that $H_{\alpha}(x, b)$ is bounded by a constant $c = c_{\alpha, x, b}$
and $\sup_{n \geq 0} M_n^h< \infty$, $\wh{\bb P}^h_{x, b}$-a.s., 
we get that for any $t > 0$, 
\begin{align*}
\bb E_{x,b} \Big( M_n^h\mathds 1_{ \{ M_n^h> t \} } \Big) 
=   H_{\alpha}(x, b)  \wh{\bb P}^h_{x, b} \Big( M_n^h> t \Big) 
\leq  c  \wh{\bb P}^h_{x, b} \Big( \sup_{n \geq 0} M_n^h> t \Big), 
\end{align*}
which tends to $0$ as $t \to \infty$, uniformly in $n$. 
This shows that the $\bb P_{x, b}$-martingale $(M_n^h, \mathscr F_n)_{n \geq 0}$ is uniformly integrable. 
Hence we have $\bb E_{x,b} M_{\infty}^h = M_0^h = H_{\alpha}(x, b)$. 
\end{proof}

\subsection{Construction of the spinal decomposition} \label{sect:spinal-decomp}

The law $\wh{\bb P}^h_{x,b}$ can be extended on a larger probability space to carry a random variable 
$w \in \bb N^{\bb N}$ called the {\it spine},  
the law of which  is  given by 
	\begin{align}\label{law-of-spine-new}
		\wh{\bb P}^h_{x, b} \big( w|n = z  \big| \mathscr F_n \big)
		: =  \frac{ H_{\alpha} \left( X_z, S_z \right)   \mathds{1}_{ \left\{ S_{z|k} \in B,  \, \forall  \,   0 \leq k \leq n  \right\} } }{ M_n^h}. 
	\end{align}

In this section we introduce a spinal decomposition which gives an alternative construction of the law $\wh{\bb P}^h_{x, b}$ 
(this is the content of Theorem \ref{Thm-Spinal-decom} below). 
The subsequent construction shows that the branching product of random matrices with a spine can also be interpreted in the following way.  
Initially, we start with the root particle which becomes a member of the spine and reproduces according to a size-biased point process $\wt{\Theta}$.
Among its children, the next member of the spine is chosen with the probability given by the right hand side of \eqref{law-of-spine-new} with $n=1$. 
Spine particles reproduce according to the point process $\wt{\Theta}$, and all other particles reproduce according to the original point process $\Theta$. 
Each particle produces offsprings in the same manner independent of all other particles. 
 In each generation, there is one member of the spine that is the descendant of the spine particle in the previous generation.

Now we introduce a law $\wt{\bb P}_{x, b}$ on $\Omega^{(2)}$, where $\Omega^{(2)}$
 is defined at the end of Section \ref{Subsection-change-of-measure}. 
To do so, we first need an additional point process which corresponds to a change of measure, 
using the harmonic function $H_{\alpha}$ defined by  \eqref{def-H-alpha-ys}. 
Let us fix $x \in \bb S_+^{d-1}$ and $b \in \bb R $. The branching product of random matrices gives rise to a point process $\Theta_{x,b}$ on $\bb S_+^{d-1} \times \bb R$: 
\begin{align} \label{def-point-process-root}
	\Theta_{x,b}: = \sum_{j=1}^N \delta_{ \left( g_j \cdot x,  \,  b + \sigma (g_j, x) \right) }, 
\end{align}
where $(g_j)_{1 \leq j \leq N}$ is obtained from the point process $\mathscr N$ (see the text before \eqref{def-basic-point-process}). 
Thereupon, 
we define the new point process by 
\begin{align}\label{def-wt-Theta-xb}
\wt{\Theta}_{x, b} (dy, ds)
: = \frac{ H_{\alpha}(y, s)  }{ H_{\alpha}(x, b) } \mathds 1_{ \{ s \in B \} }
    \Theta_{x, b} (dy, ds).  
\end{align}
In other words, for any nonnegative measurable function $f$ on $\bb S_+^{d-1} \times \bb R$, 
\begin{align*}
& \int_{ \bb S_+^{d-1} \times \bb R }  f(y,s) \wt{\Theta}_{x, b} (dy, ds)    \notag\\
& =  \frac{ 1 }{ H_{\alpha}(x, b) }   
   \sum_{j=1}^N  f\left( g_j \!\cdot\! x,  b + \sigma (g_j, x) \right)       
    H_{\alpha} \left( g_j \!\cdot\! x,  b + \sigma (g_j, x) \right)    \mathds 1_{ \{ b + \sigma (g_j, x) \in B \} }. 
\end{align*}
If $\mathfrak m(\alpha) = 1$, then, by the many-to-one formula \eqref{Formula_many_to_one} and \eqref{def-H-alpha-ys},
we get
\begin{align}\label{new-point-pro-aa}
& \bb E \int_{ \bb S_+^{d-1} \times \bb R }  f(y,s) \wt{\Theta}_{x, b} (dy, ds)  
 =   \frac{ 1 }{ H_{\alpha}(x, b) }   
\bb E_{x,b} \bigg[ \sum_{|u| = 1}  f \left( X_{u}, S_{u} \right)  H_{\alpha} \left( X_{u}, S_{u} \right)  \mathds 1_{ \{ S_{u} \in B \} } \bigg]     \notag\\
& = \frac{r_\alpha(x)}{ H_{\alpha}(x, b) }  \bb E_{ \bb Q_{x, b}^{\alpha} }  
  \bigg[  \frac{1}{ r_\alpha(X_1) }  e^{- \alpha (S_1 - S_0) }  f(X_1, S_1)  H_{\alpha}(X_1, S_1)  \mathds 1_{ \{ S_1 \in B \} } \bigg]   \notag\\
& = \frac{1}{ h(x, b) }  \bb E_{ \bb Q_{x, b}^{\alpha} }  
  \left[      f(X_1, S_1)  h(X_1, S_1)  \mathds 1_{ \{ S_1 \in B \} } \right]. 
\end{align}

In order to define a branching process with a spine, we introduce a Markov chain 
\begin{align*}
\left( \wt{N}(k),  (\wt{X}_j(k),  \wt{S}_j(k) )_{ \{ 1 \leq j \leq \wt{N}(k) \} }, w(k)  \right)
 =: \left( \wt Z_k,  w(k)   \right),  \quad  k \geq 1, 
\end{align*}
 that will generate the spine. Its law is given as follows. 
 Under $\bb P_{x,b}$, generate $\wt{Z}_1=\big(\wt{N}(1),(\wt{X}_j(1),\wt{S}_j(1))_{1 \le j \le \wt{N}(1)}\big)$ according to $\widetilde{\Theta}_{x,b}$. Thereupon, choose one particle that becomes the spine particle: generate $w(1)$ according to the
 following probability distribution 
 \begin{align*}
 	\bb P \left(  w(1) = h \big|  \wt Z_1 \right)  
 	=  \mathds 1_{ \{ h \leq \wt{N}(1) \} }  
 	\frac{  H_{\alpha} \left( \wt{X}_{h}(1), \wt{S}_{h}(1) \right)   \mathds 1_{ \left\{  \wt{S}_{h}(1) \in B  \right\} } }
 	{  \sum_{j=1}^{ \wt{N}(1) }  H_{\alpha} \left( \wt{X}_{j}(1), \wt{S}_{j}(1) \right)   \mathds 1_{ \left\{  \wt{S}_{j}(1) \in B  \right\} }   }. 
 \end{align*}
 Given $\wt{Z}_{k-1}$ 
  together with $w(k-1)$, generate $\wt{Z}_k=\big(\wt{N}(k),(\wt{X}_j(k),\wt{S}_j(k))_{1 \le j \le \wt{N}(k)}\big)$ according to $\widetilde{\Theta}_{\wt{X}_{w(k-1)}(k-1),\wt{S}_{w(k-1)}(k-1)}$. Thereupon, generate $w(k)$ according to the following formula
 \begin{align}\label{def-spine-wk}
 \bb P \left(  w(k) = h \big|  \wt Z_k \right)  
 =  \mathds 1_{ \{ h \leq \wt{N}(k) \} }  
 \frac{  H_{\alpha} \left( \wt{X}_{h}(k), \wt{S}_{h}(k) \right)   \mathds 1_{ \left\{  \wt{S}_{h}(k) \in B  \right\} } }
 {  \sum_{j=1}^{ \wt{N}(k) }  H_{\alpha} \left( \wt{X}_{j}(k), \wt{S}_{j}(k) \right)   
\mathds 1_{ \left\{  \wt{S}_{j}(k) \in B  \right\} }   }, \quad k \ge 1. 
 \end{align}
As a result, this gives a Markovian dependence of $(\wt{Z}_k, w(k))$ on $(\wt{Z}_{k-1},w(k-1))$.

 The sequence $w = (w(k))_{k \geq 1}$ 
 defines the spine, and in each generation $k\ge1$, the Markov chain provides one  family, namely the family consisting of
 the spine particle at position $( \wt{X}_{w(k)}(k), \wt{S}_{w(k)}(k) )$ and its brothers at positions $(\wt{X}_j(k),  \wt{S}_j(k) )_{ \{ 1 \leq j \leq \wt{N}(k) , j \neq w(k)\} }$. 
 Each such brother at node $v$, say, is now the initial particle of an independent branching process based on the point process $\mathcal{N}$ with initial positions $(y,s)$; here $(y,s)$ are the direction and position of that brother. The genealogical tree of the branching process rooted at $v$ is denoted by $\bb T_v$ and the corresponding family of random variables is denoted by $( X_{vu},  S_{vu}, N_{vu} )_{u \in \bb T_v}$. 
 
 Thus the full genealogical tree consists of the spine and the brothers in each generation, and to each brother (at node $v$) of the spine particle, there is a corresponding subtree $\bb T_v$ attached.
  
    The whole object is called branching process with spine and is thus represented by random variables 
  $(\left( X_u, S_u, N_u \right)_{u \in \bb T}, w )$ upon the identification
  \begin{align*}
  (X_u, S_u, N_u)  = (\wt X_{w(k)}(k), \wt S_{w(k)}(k), \wt N(k)) \qquad \text{if } u=w|k \text{ for some } k.
  \end{align*}
We write $\wt{\bb P}_{x,b}$ for the law of the process which we have just constructed. 
Let $\wt{\bb E}_{x,b}$ be the corresponding expectation with respect to $\wt{\bb P}_{x,b}$. 
Then, from \eqref{new-point-pro-aa} we have, for any measurable function $f$,  
\begin{align}\label{change-theta-theta-wt}
 \wt {\bb E}_{x,b} \Big[ f \big( (X_j, S_j)_{1 \le j \le N} \big) \Big] 
= \bb E_{x,b} \bigg[ f \big( (X_j, S_j)_{1 \le j \le N} \big) 
       \sum_{j = 1}^N \frac{H_\alpha(X_j,S_j)  }{ H_\alpha(x,b) } \mathds 1_{ \{ S_j \in B \} } \bigg]. 
\end{align}

In other words, under the law $\wt{\bb P}_{x,b}$, members of the spine are the particles at position $(X_{w|n}, S_{w|n})$, $n \in \bb N$,
and they have offspring according to $\wt{\Theta}_{X_{w|n}, S_{w|n}}$. 
All other particles at positions $(X_u, S_u)_{u \nless w}$ reproduce particles according to the original point process ${\Theta}_{X_u, S_u}$.

\begin{theorem}\label{Thm-Spinal-decom}
Assume conditions \ref{CondiAP} and \ref{Condi_ms}.  
Then, for any $x \in \bb S_+^{d-1}$ and $b \in \bb R$,
\begin{align*}
\wt{\bb P}_{x,b}  \left( \left( X_u, S_u, N_u \right)_{u \in \bb T} \in \cdot \right) 
= \wh{\bb P}_{x, b}  \left( \left( X_u, S_u, N_u \right)_{u \in \bb T} \in \cdot \right). 
\end{align*}
\end{theorem}

\begin{proof}
Let $( \phi_u,  u \in \bb T )$ be a family of nonnegative measurable functions: 
$\phi_u: \bb S_+^{d-1} \times \bb R \times \bb N \to \bb R_+$. 
By \eqref{def-hat-P-x-b}, we have that for any $x \in \bb S_+^{d-1}$, $b \in \bb R$ and $n \in \bb N$,
\begin{align*}
\wh{\bb E}_{x,b} \bigg[  \prod_{|u| \leq n} \phi_u \left( X_u, S_u, N_u \right) \bigg]
=  \bb E_{x,b} \bigg[  \frac{M_n^h}{ H_{\alpha}(x,b)  }  \prod_{|u| \leq n} \phi_u \left( X_u, S_u, N_u \right) \bigg],
\end{align*}
where $H_{\alpha}$ and $M_n^h$ are defined by \eqref{def-H-alpha-ys} and \eqref{def-martingale-Dn}, respectively.
Therefore, to prove the theorem, it is sufficient to show that for any $x \in \bb S_+^{d-1}$, $b \in \bb R$ and $n \in \bb N$,
\begin{align}\label{Spinal-decom-goal-2}
\wt{\bb E}_{x,b} \bigg[  \prod_{|u| \leq n} \phi_u \left( X_u, S_u, N_u \right) \bigg]
=  \bb E_{x,b} \bigg[  \frac{M_n^h}{ H_{\alpha}(x,b) }  \prod_{|u| \leq n} \phi_u \left( X_u, S_u, N_u \right) \bigg].
\end{align}
Decomposing by the addition of the spine particle, 
it is enough to prove that for any $x \in \bb S_+^{d-1}$, $b \in \bb R$, $n \in \bb N$ and $z \in \bb T$ with $|z| = n$,
\begin{align}\label{spinal-goal-2-decom}
I_1(n) : & =  \wt{\bb E}_{x,b} \bigg[  \mathds 1_{ \{ w|n = z \} } \prod_{|u| \leq n} \phi_u \left( X_u, S_u, N_u \right) \bigg]   \notag\\
& = \bb E_{x,b} 
\bigg[  
\frac{ H_{\alpha} \left( X_z, S_z \right)   \mathds{1}_{ \left\{ S_{z|k} \in B,  \, \forall  \,   0 \leq k \leq n  \right\} } }
       { H_{\alpha}(x,b) }  
\prod_{|u| \leq n} \phi_u \left( X_u, S_u, N_u \right) \bigg] =: I_2(n). 
\end{align}
(Indeed, by summing over $|z| = n$ in \eqref{spinal-goal-2-decom}, we get \eqref{Spinal-decom-goal-2}.) 
By decomposing the product $\prod_{|u| \leq n} \phi_u \left( X_u, S_u, N_u \right)$ along the spine, we get
\begin{align*}
I_1(n) = \wt{\bb E}_{x,b} \bigg[  \mathds 1_{ \{ w|n = z \} } 
            \prod_{k = 0}^n \phi_{z|k} \left( X_{z|k}, S_{z|k}, N_{z|k} \right)  
            \prod_{ u \in {\rm brot}(z|k) }  \Phi_u (X_u, S_u)  \bigg],
\end{align*}
where ${\rm brot}(z|k)$ is the set of brothers of $z|k$, and for $x \in \bb S_+^{d-1}$ and $s \in \bb R$, 
\begin{align*}
\Phi_u (x, s) = \bb E_{x, s}  \bigg[  \prod_{|v| \leq n - |u|} \phi_{uv} (X_v, S_v, N_v) \bigg]. 
\end{align*}
Recall that each brother of the spine particle starts an independent branching process with the point process $\Theta_{x, b}$. 
Similarly, we have 
\begin{align*}
I_2(n)  & =  \bb E_{x,b} 
\bigg[  \frac{ H_{\alpha} \left( X_z, S_z \right)  \mathds{1}_{ \left\{ S_{z|k} \in B,  \, \forall  \,   0 \leq k \leq n  \right\} } }
       { H_{\alpha}(x,b) }    \notag\\
& \qquad\quad  \times  \prod_{k = 0}^n \phi_{z|k} \left( X_{z|k}, S_{z|k}, N_{z|k} \right)  
            \prod_{ u \in {\rm brot}(z|k) }  \Phi_u (X_u, S_u)  \bigg]. 
\end{align*}
We shall prove $I_1(n) = I_2(n)$ by induction in $n$. 
For $n = 0$, it holds that $I_1(0) = I_2(0) = \bb E_{x, b} \phi_{\o} \left( X_{\o}, S_{\o}, N_{\o} \right)$. 
Now we assume that $I_1(n - 1) = I_2(n - 1)$ and we  shall show that $I_1(n) = I_2(n)$. 
For brevity, denote for $z \in \bb T$,  
\begin{align*}
\Psi(z) = \phi_{z} \left( X_{z}, S_{z}, N_{z} \right)  \prod_{ u \in {\rm brot}(z) }  \Phi_u (X_u, S_u). 
\end{align*}
Then we have
\begin{align}\label{identity-I-1n}
I_1(n) = \wt{\bb E}_{x,b} \bigg[  \mathds 1_{ \{ w|n = z \} }  \prod_{k = 0}^n \Psi (z|k) \bigg]
   = \wt{\bb E}_{x,b} \bigg[  \mathds 1_{ \{ w|n = z \} }  \Psi(z) \prod_{k = 0}^{n-1} \Psi (w|k) \bigg] 
\end{align}
and 
\begin{align}\label{identity-I-2n}
I_2(n) = \bb E_{x,b}  \bigg[  
\frac{ H_{\alpha} \left( X_z, S_z \right)  \mathds{1}_{ \left\{ S_{z|k} \in B,  \, \forall  \,   0 \leq k \leq n  \right\} } }
       { H_{\alpha}(x,b) }   \prod_{k = 0}^n \Psi (z|k)   \bigg]. 
\end{align}
Let $ \mathscr{G}_n$ be the natural filtration of the spine particles and their brothers, i.e.
\begin{align}\label{eq:filtrationGn}
\mathscr{G}_n = \sigma \left\{ w|k, X_{w|k}, S_{w|k}, (X_u, S_u)_{ u \in {\rm brot}(w|k) } ; 0 \le k \le n \right\}. 
\end{align}
Using \eqref{def-spine-wk}, we get that for any $z \in \bb T$ with $|z| = n$, 
\begin{align}\label{condi-exp-Psi-z}
& \wt{\bb E}_{x,b}  \Big(  \mathds 1_{ \{ w|n = z \} }  \Psi(z)  \Big|  \mathscr{G}_{n-1} \Big)   \notag\\
& =  \mathds 1_{ \{ w|n-1 = z|n-1 \} }   \wt{\bb E}_{x,b}  
   \Big(   \wt{H}_{\alpha} (X_z, S_z)  \Psi(z)  \Big|  \mathscr{G}_{n-1} \Big)   \notag\\
& =  \mathds 1_{ \{ w|n-1 = z|n-1 \} }   \wt{\bb E}_{x,b}  
   \Big(   \wt{H}_{\alpha} (X_z, S_z)  \Psi(z)  \Big|  w|n-1, X_{w|n-1}, S_{w|n-1} \Big),
\end{align}
where, for brevity we denote 
\begin{align*}
\wt{H}_{\alpha} (X_z, S_z) =  \frac{ H_\alpha(X_z,S_z) \mathds{1}_{\{S_z \in B \}}}{ H_\alpha(X_z,S_z)\mathds{1}_{\{S_z \in B\}} + \sum_{v \in {\rm brot}(z) }  H_\alpha(X_v,S_v)\mathds{1}_{\{S_v \in B \}} }. 
\end{align*}
In the expectation only children of the spine particle $w|n-1$ appear, so their distribution is given by $\wt \Theta_{X_{w|n-1}, S_{w|n-1}}$. 
Using \eqref{change-theta-theta-wt}, we obtain that on the event $\{ w|n-1 = z|n-1 \}$, 
\begin{align}\label{def-F-XS-wn-1}
&\wt {\bb E}_{x,b}  
\Big(   \wt{H}_{\alpha} (X_z, S_z)  \Psi(z)  \Big|  w|n-1, X_{w|n-1}, S_{w|n-1} \Big)   \notag\\
& = \wt {\bb E}_{X_{w|n-1}, S_{w|n-1}}  
\bigg(   \frac{ H_\alpha(X_{z_n},S_{z_n}) \mathds{1}_{\{S_{z_n} \in B \}}}{ \sum_{1 \le j \le N} H_\alpha(X_j,S_j)\mathds{1}_{\{S_j \in B \}}  }  \Psi(z_n)    \bigg)   \notag\\
& =  \bb E_{X_{w|n-1}, S_{w|n-1}}  
\bigg(   \frac{ H_\alpha(X_{z_n},S_{z_n}) \mathds{1}_{\{ S_{z_n} \in B \}}}{ \sum_{1 \le j \le N} H_\alpha(X_j,S_j)\mathds{1}_{\{S_j \in B \}}  } 
 \sum_{1 \le j \le N} \frac{H_\alpha (X_j,S_j)  }{ H_\alpha(X_{\o}, S_{\o}) } \mathds 1_{ \{ S_j \in B \} } \Psi(z_n)    \bigg)   \notag\\
& =  \bb E_{X_{w|n-1}, S_{w|n-1}}  
\bigg(   \frac{ H_\alpha(X_{z_n},S_{z_n}) \mathds{1}_{\{S_{z_n} \in B \}}}{ H_\alpha(X_{\o},S_{\o}) }   \Psi(z_n)    \bigg)  \notag\\
& = : F \left( X_{w|n-1}, S_{w|n-1} \right). 
\end{align}
Combining this with \eqref{condi-exp-Psi-z} and \eqref{identity-I-1n}, 
taking conditional expectation with respect to $\mathscr{G}_{n-1}$ 
and using the fact that $\prod_{k = 0}^{n-1} \Psi (w|k)$ is measurable with respect to $\mathscr{G}_{n-1}$,
we get
\begin{align*}
I_1(n)  & =  \wt{\bb E}_{x,b}  \bigg[
    \wt{\bb E}_{x,b} \Big(  \mathds 1_{ \{ w|n = z \} }  \Psi(z)  \prod_{k = 0}^{n-1} \Psi (w|k)  \Big|  \mathscr{G}_{n-1} \Big)
    \bigg]   \notag\\
 & = \wt{\bb E}_{x,b}  \bigg[  \mathds 1_{ \{ w|n-1 = z|n-1 \} }  F \left( X_{w|n-1}, S_{w|n-1} \right)  \prod_{k = 0}^{n-1} \Psi (w|k)   \bigg]. 
\end{align*}
By the assumption $I_1(n-1) = I_2(n-1)$, it follows that 
\begin{align*}
I_1(n) =  \bb E_{x,b}  \bigg[  
\frac{ H_{\alpha} \left( X_{z|n-1}, S_{z|n-1} \right)  \mathds{1}_{ \left\{ S_{z|k} \in B,  \, \forall  \,   0 \leq k \leq n-1  \right\} } }
       { H_{\alpha}(x,b) }   F \left( X_{w|n-1}, S_{w|n-1} \right)  \prod_{k = 0}^{n-1} \Psi (w|k) \bigg]. 
\end{align*}
By \eqref{def-F-XS-wn-1},  
we get $I_1(n) = I_2(n)$ with $I_2(n)$ defined by \eqref{identity-I-2n}. 
\end{proof}

\begin{theorem}\label{Thm-spinal-decom}
Assume conditions \ref{CondiAP} and \ref{Condi_ms}.  
Let $x \in \bb S_+^{d-1}$ and $b \in \bb R$. 
Then, for any $n \geq 0$ and any vertex $z \in \bb T$ with $|z| = n$,
\begin{align}\label{thm2-proba-spine}
\wh{\bb P}_{x, b} \big( w|n = z  \big| \mathscr F_n \big)
=  \frac{ H_{\alpha} \left( X_z, S_z \right)   \mathds{1}_{ \left\{ S_{z|k} \in B,  \, \forall  \,   0 \leq k \leq n  \right\} } }{ M_n^h},
\end{align}
where $H_{\alpha}$ and $M_n^h$ are defined by \eqref{def-H-alpha-ys} and \eqref{def-martingale-Dn}, respectively. 

Moreover, the process $(X_{w|n}, S_{w|n})_{n \geq 0}$ under $\wh{\bb P}_{x, b}$ is distributed as the Markov random walk
$(X_n, S_n)_{n \geq 0}$, under $\bb Q_{x,b}^{\alpha}$, 
conditioned to stay in the set $B$ (meaning that $S_k \in B$ for any $1 \leq k \leq n$). 
Precisely, for any nonnegative measurable function $f: (\bb S_+^{d-1} \times \bb R)^{n+1} \to \bb R_+$, we have
\begin{align}\label{thm2-spinal-condi}
& \wh{\bb E}_{x, b} \Big[  f \Big(  (X_{w|k}, S_{w|k}),  0 \leq k \leq n  \Big)  \Big]   \notag\\
& =  \bb{E}_{ \bb Q_{x,b}^{\alpha} } 
        \Big[  f \Big( (X_k, S_k), 0 \leq k \leq n  \Big) 
           \frac{ e^{\alpha b} h \left( X_n, S_n \right)  }
                  { h(x, b)   }   
  \mathds{1}_{ \left\{ S_{k} \in B,  \, \forall  \,   0 \leq k \leq n  \right\} }  \Big].   
\end{align}
\end{theorem}

\begin{proof}
As in the proof of Theorem \ref{Thm-Spinal-decom}, 
let $( \phi_u,  u \in \bb T )$ be a family of nonnegative measurable functions: 
$\phi_u: \bb S_+^{d-1} \times \bb R \times \bb N \to \bb R_+$. 
Using \eqref{spinal-goal-2-decom} and Theorem \ref{Thm-Spinal-decom} (identifying $\wt{\bb E}_{x,b}$ by $\wh{\bb E}_{x,b}$),
we get that for any $x \in \bb S_+^{d-1}$, $b \in \bb R$, $n \in \bb N$ and $z \in \bb T$ with $|z| = n$, 
\begin{align*}
&   \wh{\bb E}_{x,b} \Big[  \mathds 1_{ \{ w|n = z \} } \prod_{|u| \leq n} \phi_u \left( X_u, S_u, N_u \right) \Big]   \notag\\
& = \bb E_{x,b}  \bigg[  
\frac{ H_{\alpha} \left( X_z, S_z \right)   \mathds{1}_{ \left\{ S_{z|k} \in B,  \, \forall  \,   0 \leq k \leq n  \right\} } }
       { H_{\alpha}(x,b) }  
\prod_{|u| \leq n} \phi_u \left( X_u, S_u, N_u \right) \bigg] =: I. 
\end{align*}
By \eqref{def-hat-P-x-b}, it holds that 
\begin{align*}
I = \wh{\bb E}_{x,b}  \bigg[  
\frac{ H_{\alpha} \left( X_z, S_z \right)   \mathds{1}_{ \left\{ S_{z|k} \in B,  \, \forall  \,   0 \leq k \leq n  \right\} } }
       { M_n^h}  
\prod_{|u| \leq n} \phi_u \left( X_u, S_u, N_u \right) \bigg],
\end{align*}
which shows \eqref{thm2-proba-spine}. 

Now we prove \eqref{thm2-spinal-condi}. 
Using \eqref{thm2-proba-spine}, we have
\begin{align*}
J:& = \wh{\bb E}_{x, b} \Big[  f \Big(  (X_{w|k}, S_{w|k}),  0 \leq k \leq n  \Big)  \Big]   \notag\\
& = \wh{\bb E}_{x, b} \Big[  \sum_{|z| = n} f \Big(  (X_{z|k}, S_{z|k}),  0 \leq k \leq n  \Big)  \ds 1_{ \{ w |n = z \} } \Big]   \notag\\
& = \wh{\bb E}_{x, b} \bigg[  \sum_{|z| = n} f \Big(  (X_{z|k}, S_{z|k}),  0 \leq k \leq n  \Big)  
         \frac{ H_{\alpha} \left( X_z, S_z \right)   \mathds{1}_{ \left\{ S_{z|k} \in B,  \, \forall  \,   0 \leq k \leq n  \right\} } }{ M_n^h}  \bigg]. 
\end{align*}
By \eqref{def-hat-P-x-b}, it follows that
\begin{align*}
J =  \bb E_{x, b} \bigg[  \sum_{|z| = n}  f \Big(  (X_{z|k}, S_{z|k}),  0 \leq k \leq n  \Big)  
         \frac{ H_{\alpha} \left( X_z, S_z \right)   \mathds{1}_{ \left\{ S_{z|k} \in B,  \, \forall  \,   0 \leq k \leq n  \right\} } }
                { H_{\alpha}(x,b) }  \bigg]. 
\end{align*}
Using the many-to-one formula \eqref{Formula_many_to_one},
the fact that $\mathfrak m (\alpha)  = 1$ (cf.\ condition \ref{Condi_ms})
and the definition of $H_{\alpha}$ (cf.\ \eqref{def-H-alpha-ys}), we obtain 
\begin{align*}
J  & =  r_{\alpha}(x)  \bb E_{\bb Q_{x,b}^{\alpha}} 
 \bigg[ \frac{1}{ r_{\alpha} (X_n)} e^{-\alpha S_n}  f \Big(  (X_{k}, S_{k}),  0 \leq k \leq n  \Big)  
         \frac{ H_{\alpha} \left( X_n, S_n \right)   \mathds{1}_{ \left\{ S_{k} \in B,  \, \forall  \,   0 \leq k \leq n  \right\} } }
                { H_{\alpha}(x,b) } 
                \bigg]   \notag\\
  & =  \bb{E}_{ \bb Q_{x,b}^{\alpha} }  
        \Big[  f \Big( (X_k, S_k), 0 \leq k \leq n  \Big) 
           \frac{ e^{\alpha b} h \left( X_n, S_n \right)  }
                  { h(x, b)   }   
            \mathds{1}_{ \left\{ S_{k} \in B,  \, \forall  \,   0 \leq k \leq n  \right\} }
           \Big],
\end{align*}
which proves \eqref{thm2-spinal-condi}. 
\end{proof}

Take $B = \bb R$ and $h = 1$ in Section \ref{sec-change-proba}, 
so that \eqref{condition-on-B} and \eqref{harmonicity-h} are satisfied. 
Then $M_n^1$ defined by \eqref{def-martingale-Dn} becomes the additive martingale 
$W_n$ defined by \eqref{def-addi-martigale},
and \eqref{def-hat-P-x-b} coincides with \eqref{def-hat-P-x-b-intro}. 
As a consequence of Theorem \ref{Thm-spinal-decom}, we have the following result 
about size-biased branching random walks on $\mathcal M_+$.

\begin{corollary}\label{Cor-spine-001}
Assume conditions \ref{CondiAP} and \ref{Condi_ms}. 
Then, for any $n \geq 0$ and any vertex $z \in \bb T$ with $|z| = n$,
\begin{align*}
\wh{\bb P}_{x, b} \big( w|n = z  \big| \mathscr F_n \big)
=  \frac{ r_{\alpha} (X_z)  e^{- \alpha S_z} }{ W_n }. 
\end{align*}
Moreover, the process $(X_{w_n}, S_{w_n})_{n \geq 0}$ under $\wh{\bb P}_{x, b}$ is distributed as the Markov random walk
$(X_n, S_n)_{n \geq 0}$ under the measure $\bb Q_{x,b}^{\alpha}$. 
Precisely, for any nonegative measurable function $f: (\bb S_+^{d-1} \times \bb R)^{n+1} \to \bb R_+$, we have
\begin{align}\label{sized-BRW-spinal}
 \wh{\bb E}_{x, b} \Big[  f \Big(  (X_{w|k}, S_{w|k}),  0 \leq k \leq n  \Big)  \Big]  
 =  \bb{E}_{ \bb Q_{x,b}^{\alpha} } 
        \Big[  f \Big( (X_k, S_k), 0 \leq k \leq n  \Big)   \Big].   
\end{align}
\end{corollary}

\subsection{Conditioned local limit theorems for products of random matrices}
In this section, we collect results concerning the behavior of $S_n$ conditioned to stay above a threshold.

Consider the delayed first entrance times into $(-\infty,0)$ and $(0,\infty)$, respectively:
\begin{align}\label{def-tau-y-plus-minus}
\tau_y^- =\inf\{ j \ge 1: S_n+y<0\}, \qquad \tau_y^+ =\inf\{ j \ge 1: S_n-y>0\}. 
\end{align}
\begin{lemma}\label{Lem-positi-harmonic-func}
Assume \ref{CondiAP} and \ref{Condi_ms}. 
Then the following limit exists and is finite: for any $x \in \bb S_+^{d-1}$ and $y \geq 0$,  
\begin{align}\label{def-V-alpha-xy}
V_{\alpha}(x, y) = \lim_{n \to \infty} \bb E_{\bb Q_{x}^{\alpha}} \left( y + S_n; \tau_y^- > n \right). 
\end{align}
The function $V_{\alpha}$ is harmonic  for $(X_n, S_n)$ killed outside the set $\bb S_+^{d-1} \times [0, \infty)$, i.e. 
for any $x \in \bb S_+^{d-1}$ and $b \geq 0$, 
\begin{align*}
\bb E_{\bb Q_{x}^{\alpha}} \left[ V_{\alpha} (X_1, b + S_1) \mathds 1_{ \{ b + S_1 \geq 0 \} } \right] = V_{\alpha} (x, b).
\end{align*}
Moreover, assuming in addition that for some $\delta>0$, 
\begin{align}\label{precise-condi-Qx-alpha}
\inf_{x \in \bb S_+^{d-1}} \bb Q_x^\alpha \left( \left\{ g \, : \, - \sigma(g,x)>\delta \right\} \right)>0,
\end{align}
then there exist constants $c, c' >0$ such that for any $x \in \bb S_+^{d-1}$ and $b \geq 0$, 
\begin{align}\label{inequa-V-harmonic}
c' (1 + b)  \leq  V_{\alpha} (x, b)  \leq  c (1 + b), 
\end{align}
and  $\lim_{b \to \infty} \frac{ V_{\alpha} (x, b) }{b} = 1$ uniformly in $x \in \bb S_+^{d-1}$. 
\end{lemma}

The existence of the function $V_{\alpha}$ in Lemma \ref{Lem-positi-harmonic-func} follows by adapting the same arguments as in \cite[Propositions 5.11 and 5.12]{GLP17}, see also \cite{Pha18}. 
The key point in the proof of Lemma \ref{Lem-positi-harmonic-func} is that a martingale approximation holds under a change of measure $\bb Q_{x}^{\alpha}$ due to the spectral gap properties of the transfer operator $P_s$
shown in Section \ref{sect:spectralgap}. 
The details are left to the reader. 

Turning to the positivity of $V_{\alpha}$, 
we can follow the proof of \cite[Proposition 5.12]{GLP17}. 
It relies on the Markov property and a condition called P5 there, which corresponds to \eqref{precise-condi-Qx-alpha}. 
The following lemma shows that this condition is satisfied subject to \ref{Condi-Furstenberg-Kesten} and \ref{Condi_harmonic}.

\begin{lemma}\label{lem:condition_harmonic}
Conditions \ref{Condi-Furstenberg-Kesten} and \ref{Condi_harmonic} imply \eqref{precise-condi-Qx-alpha} 
with $\delta>0$ given by \ref{Condi_harmonic}. 
\end{lemma}

\begin{proof} 
By \eqref{Formula_many_to_one} with $b=0$, and Lemma \ref{lemma kappa 1}, it holds with $\overline{\varkappa}:=2 \log \varkappa$, 
\begin{align*}
& \bb Q_x^\alpha \left( \left\{ g \, : \,  - \sigma(g,x)>\delta \right\} \right)  \notag\\
& = \frac{1}{r_\alpha(x) \mathfrak{m}(\alpha)} \bb E \bigg(  \sum_{|u|=1} e^{-\alpha \log \| G_u x \| } r_\alpha(G_u x) 
  \mathds{1}_{(\delta,\infty)}( -\log \| G_u x \|)  \bigg)  \notag\\
&\geq  c \bb E \bigg(  \sum_{|u|=1} e^{-\alpha \log \norm{G_u}} \mathds{1}_{(\overline{\varkappa} +\delta,\infty)}(-\log \norm{G_u})  \bigg) \geq  c e^{\alpha(\overline{\varkappa} +\delta)} \bb E \bigg(  \sum_{|u|=1}  \mathds{1}_{(\overline{\varkappa} +\delta,\infty)}(-\log \norm{G_u})  \bigg),
\end{align*}
where the last term is positive by our assumption.
\end{proof}

The following lemma allows us to replace $\log \| G_n x \|$ by $\log \| G_n x' \|$ due to condition \ref{Condi-Furstenberg-Kesten}. 

\begin{lemma}[\cite{GX23}] \label{Lem-contractivity}
Assume condition \ref{Condi-Furstenberg-Kesten}.  
Then there exists a constant $c_0 >0$  such that for any $g \in \Gamma$ and $x, x' \in \bb S_+^{d-1}$, 
\begin{align*}
\left|  \log \| g x \| -   \log \| g x' \| \right|  \leq  c_0. 
\end{align*}
\end{lemma}

A core ingredient of the proofs will be 
a conditioned local limit theorem for products of positive random matrices.

\begin{lemma}[\cite{GX23}] \label{Lem-CLLT-bound}
Assume conditions \ref{Condi-Furstenberg-Kesten}, \ref{CondiNonarith} and \ref{Condi_ms}.  
Then there exists a constant $c>0$ such that for any $x \in \bb S_+^{d-1}$, $y \geq 0$, $z \geq 0$, $t \in \bb R$ and $n \geq 1$, 
\begin{align*}
 \bb Q_{x}^{\alpha} \left( y + S_n \in  t +  [0, z], \tau_y^- > n \right)
\leq  \frac{c}{n^{3/2}} (1 + y)  (1 + z) (1 + z + \max\{t, 0\})
\end{align*}
and 
\begin{align*}
 \bb Q_{x}^{\alpha} \left( y - S_n \in t + [0, z], \tau_y^+ > n \right)
\leq  \frac{c}{n^{3/2}} (1 + y)  (1 + z) (1 + z + \max\{t, 0\}). 
\end{align*}
\end{lemma}

Lemma \ref{Lem-CLLT-bound} for $\alpha = 0$ was established in \cite[Theorem 1.3]{GX23},
see also \cite[Theorem 2]{Peigne Pham 2023}. The case $\alpha > 0$ follows analogously under the change of measure $\bb Q_{x}^{\alpha}$.
Examining the proof of \cite[Theorem 1.3]{GX23}, we observe that the key ingredient for the result is the spectral gap properties of the transfer operator $P_s$ defined in \eqref{Def-Ps}. These properties have already been established in Section \ref{sect:spectralgap}.

Note that the moment condition A2 in \cite{GX23} is satisfied 
in our setting as a consequence of \ref{Condi-Furstenberg-Kesten} and \ref{Condi_ms},
see Remark \ref{Rem-moment-condi}.

\begin{lemma}\label{Lem-Renewal-thm}
Assume conditions \ref{Condi-Furstenberg-Kesten}, \ref{CondiNonarith} and \ref{Condi_ms}.  
Then there exists a constant $c>0$ such that for any $x \in \bb S_+^{d-1}$, $y \geq 0$ and $z \geq 0$, 
\begin{align*}
\sum_{n = 0}^{\infty} \bb Q_{x}^{\alpha} \left( y + S_n \in [0, z], \tau_y^- > n \right)
\leq  c  (1 + z) \left( 1 + \min\{y, z\} \right)
\end{align*}
and 
\begin{align*}
\sum_{n = 0}^{\infty} \bb Q_{x}^{\alpha} \left( y - S_n \in [0, z], \tau_y^+ > n \right)
\leq  c  (1 + z) \left( 1 + \min\{y, z\} \right). 
\end{align*}
\end{lemma}

\begin{proof}
We follow the proof of \cite[Lemma A.5]{Shi12}. 
We only prove the first inequality, the second one being similar. 

First case: $z \geq y$. By \eqref{bound-tau-b} and Lemma \ref{Lem-CLLT-bound}, we have 
\begin{align*}
\sum_{n = 0}^{\infty} \bb Q_{x}^{\alpha} \left( y + S_n \in [0, z], \tau_y^- > n \right)
& \leq  \sum_{n = 0}^{ \floor{z^2} } \bb Q_{x}^{\alpha} \left( \tau_y^- > n \right) 
 +   \sum_{n = \floor{z^2} + 1}^{\infty} \bb Q_{x}^{\alpha} \left( y + S_n \in [0, z], \tau_y^- > n \right)  \notag\\
 & \leq   \sum_{n = 0}^{ \floor{z^2} }  \frac{c(1 +y)}{\sqrt{n + 1}} +  \sum_{n = \floor{z^2} + 1}^{\infty}  \frac{c}{n^{3/2}} (1 + y)  (1 + z)^2 \notag\\
 & \leq   c(1 +y) (1 + z). 
\end{align*}

Second case: $z < y$.
Upon defining the stopping time $T_z = \inf \{n \geq 1: S_n \leq z\},$
 we have 
\begin{align*}
& \sum_{n = 0}^{\infty} \bb Q_{x, 0}^{\alpha} \left( y + S_n \in [0, z], \tau_y^- > n \right) \notag\\
& =  \bb E_{\bb Q_{x, y}^{\alpha}}  \left( \sum_{n = 0}^{\infty} \mathds 1_{ \{ S_n \leq z,  \ \min_{1 \leq j \leq n} S_j \geq 0  \} } \right)
\notag\\
& = \bb E_{\bb Q_{x, y}^{\alpha}}  \left( \sum_{n = T_z}^{\infty} \mathds 1_{ \{ S_n \leq z,  \ \min_{1 \leq j \leq n} S_j \geq 0  \} } \right) \notag\\
& \leq \bb E_{\bb Q_{x, y}^{\alpha}}  \left( \sum_{k = 0}^{\infty} 
 \mathds 1_{ \{ S_{T_z} + (S_{T_z + k} - S_{T_z} ) \leq z,  \ \min_{0 \leq j \leq k}  ( S_{T_z} +  S_{T_z + j} - S_{T_z})  \geq 0  \} } \right).
\end{align*}
We can without loss of generality assume that $0 \leq S_{T_z} \leq z$, otherwise the above sum is zero. 
Employing this, we estimate
\begin{align*}
& \bb E_{\bb Q_{x, y}^{\alpha}}  \left( \sum_{k = 0}^{\infty} 
 \mathds 1_{ \{ S_{T_z} + (S_{T_z + k} - S_{T_z} ) \leq z,  \ \min_{0 \leq j \leq k}  ( S_{T_z} +  S_{T_z + j} - S_{T_z})  \geq 0  \} } \right) \notag\\
& \leq    \bb E_{\bb Q_{x, y}^{\alpha}}  \left( \sum_{k = 0}^{\infty} 
 \mathds 1_{ \{ S_{T_z + k} - S_{T_z}  \leq z,  \ \min_{0 \leq j \leq k}   S_{T_z + j} - S_{T_z}  \geq -z  \} } \right)  \notag\\
 & =  \bb E_{\bb Q_{x, y}^{\alpha}}  \left(  \bb E_{\bb Q_{X_{T_z}, 0}^{\alpha}}
  \sum_{k = 0}^{\infty} 
 \mathds 1_{ \{ S_{k}   \leq z,  \ \min_{0 \leq j \leq k}   S_{j}  \geq -z  \} } \right)  \notag\\
& =  \int_{ \bb S_+^{d-1} } \sum_{k = 0}^{\infty}   \bb Q_{x', 0}^{\alpha}
  \left( S_{k}   \leq z,  \ \min_{0 \leq j \leq k}   S_{j}  \geq -z  \right)  \bb Q_{x, y}^{\alpha} (X_{T_z} \in dx')
\end{align*}
For $k >  \floor{z^2}$, applying Lemma \ref{Lem-CLLT-bound} gives for any $x' \in \bb S_+^{d-1}$, 
\begin{align*}
\bb Q_{x', 0}^{\alpha}
  \left( S_{k}   \leq z,  \ \min_{0 \leq j \leq k}   S_{j}  \geq -z  \right)
   = \bb Q_{x', 0}^{\alpha}
  \left( z + S_{k}   \leq 2 z,  \  z + \min_{0 \leq j \leq k}   S_{j}  \geq 0  \right)
  \leq \frac{c}{k^{3/2}} (1 + z)^3. 
\end{align*}
Therefore, we get
\begin{align*}
\sum_{k = 0}^{\infty}   \bb Q_{x', 0}^{\alpha}
  \left( S_{k}   \leq z,  \ \min_{0 \leq j \leq k}   S_{j}  \geq -z  \right)
  \leq  1 + \floor{z^2} +   \sum_{k = \floor{z^2} + 1}^{\infty}   \frac{c}{k^{3/2}} (1 + z)^3 
  \leq c (1 + z)^2,
\end{align*}
as required. 
\end{proof}

\begin{lemma}\label{Lem-Renewal-thm-bb}
Assume conditions \ref{Condi-Furstenberg-Kesten}, \ref{CondiNonarith} and \ref{Condi_ms}.  
Then there exists a constant $c>0$ such that for any $x \in \bb S_+^{d-1}$, $y \geq 0$, $z \geq 0$ and $t \in \bb R$, 
\begin{align*}
\sum_{n = 0}^{\infty} \bb Q_{x}^{\alpha} \left( y + S_n \in t + [0, z], \tau_y^- > n \right)
\leq  c(1 +z) (1 + \max\{z, t + z, y\}) 
\end{align*}
and 
\begin{align*}
\sum_{n = 0}^{\infty} \bb Q_{x}^{\alpha} \left( y - S_n \in t + [0, z], \tau_y^+ > n \right)
\leq  c  (1 +z) (1 + \max\{z, t + z, y\}). 
\end{align*}
\end{lemma}

\begin{proof}
It is enough to give a proof for $t \geq 0$. 
By the ordinary local limit theorem (cf.\ \cite[Theorem 2.6]{GX23}) and Lemma \ref{Lem-CLLT-bound}, 
we have that, for any $x \in \bb S_+^{d-1}$, $y \geq 0$, $z \geq 0$ and $t \geq 0$, 
\begin{align*}
& \sum_{n = 0}^{\infty} \bb Q_{x}^{\alpha} \left( y + S_n \in t + [0, z], \tau_y^- > n \right)  \notag\\
& \leq  \sum_{n = 0}^{ \floor{(t + z)^2} } \bb Q_{x}^{\alpha} \left(  y + S_n \in t + [0, z] \right) 
 +   \sum_{n = \floor{(t+z)^2} + 1}^{\infty} \bb Q_{x}^{\alpha} \left( y + S_n \in t + [0, z], \tau_y^- > n \right)  \notag\\
 & \leq   \sum_{n = 0}^{ \floor{(t+z)^2} }  \frac{c(1 +z)}{\sqrt{n + 1}} 
  +  \sum_{n = \floor{(t+z)^2} + 1}^{\infty}  \frac{c}{n^{3/2}} (1 + y)  (1 + z) (1 + z +t) \notag\\
 & \leq   c(1 +z) (1 + t + z) +  c (1 + y)  (1 + z)  \notag\\
 & \leq   c(1 +z) (1 + \max\{t + z, y\}), 
\end{align*}
which shows the first inequality. The proof of the second one is similar. 
\end{proof}

\section{Convergence of martingales}\label{Sec-conver-deriva-mart}

\subsection{Convergence of the additive martingale}
The goal of this section is to establish Theorem \ref{Thm:Biggins-bb} on the convergence of Biggins' martingale. 
Recall that $\mathfrak M$ is defined by \eqref{def-psi-s}.

\begin{proposition}\label{Thm:Biggins}
Assume condition \ref{CondiAP} and $\bb E N >1$. 
Assume also that there exists a constant $s \in I$ such that $\mathfrak M(s) \leq s \mathfrak M'(s)$
and 
$$\bb E \Big( \sum_{|u| = 1} \|g_u\|^{s} (\log\|g_u\| + \log \iota(g_u)) \Big) < \infty.$$ 
Then, for any $x \in \bb S_+^{d-1}$, we have $W_{\infty}(s) = 0$, $\bf P_{x}$-a.s. 
\end{proposition}

\begin{proof}
By \eqref{sized-BRW-spinal} and the fact that the pair $(X_n, S_n)$ is a Markov chain under  $\bb Q_{x,b}^{s}$, 
we have that $(X_{w|n}, S_{w|n})$ is also a Markov chain under $\wh{\bb P}_{x, b}$. 
Hence, by the strong law of large numbers (cf.\ \cite[Theorem 6.1]{BDGM14}), we get that, for any $x \in \bb S_+^{d-1}$, 
\begin{align}\label{SLLN-spine-aa}
\lim_{n \to \infty} \frac{ S_{w|n} }{n} 
=  \wh{\bb E}_{x} (S_{w|1}) 
= \bb{E}_{ \bb Q_{x}^{s} } S_1
= - \mathfrak M'(s),  \quad  \wh{\bb P}_{x}\mbox{-a.s.}
\end{align}
where in the second equality we used again \eqref{sized-BRW-spinal}. 
It follows that 
\begin{align*}
\lim_{n \to \infty} \frac{ s S_{w|n} - n \mathfrak M(s) }{n} = s \mathfrak M'(s) - \mathfrak M(s) \geq 0,  \quad  \wh{\bb P}_{x}\mbox{-a.s.}
\end{align*}
and hence $\limsup_{n \to \infty} [ s S_{w|n} - n \mathfrak M(s) ] =  \infty$, $\wh{\bb P}_{x}$-a.s.
Using the definition of $W_n$ (cf.\ \eqref{def-addi-martigale}) and the fact that $r_{s}$ is bounded from above and below by 
a strictly positive constant, we obtain that there exists a constant $c \in (0, \infty)$ such that
\begin{align*}
W_n(s) \geq  c   \sum_{|u| = n} \exp { [ - s S_u - n \mathfrak M(s) ] }  
\geq  c  \exp{ [ - s S_{w|n} - n \mathfrak M(s) ] }. 
\end{align*}
This yields $\limsup_{n \to \infty}  W_n(s) =  \infty$, $\wh{\bb P}_{x}$-a.s.,
so that $\bb E_{x} (W_{\infty}(s)) = 0$ for any $x \in \bb S_+^{d-1}$. 
Since $W_{\infty}(s) \geq 0$, this implies that, for any $x \in \bb S_+^{d-1}$, we have $W_{\infty}(s) = 0$, $\bf P_{x}$-a.s.
\end{proof}

\begin{proof}[Proof of Theorem \ref{Thm:Biggins-bb}]
(1) We prove that, if $\bb E_x ( W_1(s) \log^+ W_1(s)) = \infty$, then $W_{\infty}(s) = 0$, $\bf P_{x}$-a.s. 
Since $r_{s} > c$ for some constant $c>0$, by \eqref{def-addi-martigale}, we have that for any $n \geq 0$, 
\begin{align}\label{Lower-W-n-aaa}
W_{n+1}(s) 
& \geq  c  \sum_{|u| = n+1} e^{-s S_u -  (n+1) \mathfrak M(s) }   \notag\\
& =  c \sum_{|u| = n} e^{-s S_u -  (n+1) \mathfrak M(s) }  \sum_{|v| = n+1, \overset{\leftarrow}{v} = u} e^{-s (S_v - S_u)}   \notag\\
& \geq  c e^{-s S_{w|n} -  (n+1) \mathfrak M(s) }   U_n, 
\end{align}
where we recall that $S_{w|0} = S_{\o} = 0$ under $\bb P_x$, and 
\begin{align}\label{def-Un-thm}
U_n = \sum_{|v| = n+1, \overset{\leftarrow}{v} = w|n} e^{-s (S_v - S_{w|n})}. 
\end{align}
We shall apply the following extended Borel-Cantelli lemma (\cite[Theorem 4.3.4]{Dur19}): 
let $(Y_n)_{n \geq 1}$ be any process, and $A_n \in \mathscr{Y}_n: = \sigma(Y_1, \ldots, Y_n)$, then almost surely, 
\begin{align}\label{extended-Borel-Cantellli}
\{ \omega: \omega \in A_n  \  \mbox{i.o.} \} = \left\{ \omega: \sum_{n=1}^{\infty} \bb P (A_{n+1} | \mathscr{Y}_n) = \infty  \right\}. 
\end{align}
We fix a constant $a >0$ whose value will be taken sufficiently large, and for $n \geq 0$, denote
\begin{align*}
A_n = \{ \log U_n > a n \}. 
\end{align*}
Let $\mathscr{H}_n$ be the natural filtration of the spine particles and their children, i.e.
\begin{align}\label{filtration-spine-child}
\mathscr{H}_n = \sigma \left\{ w|k, X_{w|k}, S_{w|k}, (X_v, S_v)_{ |v| = k+1, \overset{\leftarrow}{v} = w|k };  0 \le k \le n \right\}. 
\end{align}
Using the branching property and the change of measure \eqref{def-hat-P-x-b-intro} with $b=0$, we get 
\begin{align}\label{identity-P-An-01}
\wh{\bb P}_{x} (A_{n+1} | \mathscr{H}_n)
& =  \wh{\bb P}_{x} \left( \log U_{n+1} > a (n+1) | \mathscr{H}_n \right)  \notag\\ 
& =   \wh{\bb P}_{X_{w|n}} \left( \log U_{0} > a (n+1) \right)  \notag\\
& = \frac{ 1 }{  r_{s}(X_{w|n})  }   \bb E_{X_{w|n}} \left( W_1(s) \mathds 1_{\{ \log U_{0} > a (n+1) \}}  \right). 
\end{align}
By Lemma \ref{lemma kappa 1} and the fact that $S_{\o} = 0$ under $\bb P_x$, under condition \ref{Condi-Furstenberg-Kesten}, 
there exist constants $c, c' >0$ such that for any $x, x' \in \bb S_+^{d-1}$,
\begin{align}\label{identity-P-An-02}
\bb E_{x'} \left( W_1(s) \mathds 1_{\{ \log U_{0} > a (n+1) \}}  \right)
& = \bb E_{x'} \bigg( \frac{1}{\mathfrak m (s)} 
  \sum_{|u| = 1} e^{-s S_u} r_{s}(X_u) \mathds 1_{\{ \sum_{|v| = 1} e^{-s S_v} > e^{a (n+1)} \}}  \bigg) \notag\\
& \geq  c' \bb E_{x} \bigg( \frac{1}{\mathfrak m (s)} \sum_{|u| = 1} e^{-s S_u} r_{s}(X_u) 
 \mathds 1_{\{ \frac{1}{\mathfrak m (s)}  \sum_{|v| = 1} e^{-s S_v} r_{s}(X_v) > c e^{a (n+1)} \}}  \bigg) \notag\\
& = c' \bb E_{x} \left( W_1(s) \mathds 1_{\{ W_1(s) > c e^{a (n+1)} \}}  \right)  \notag\\
& = c' \wh{\bb P}_{x} \left(  W_1(s) > c e^{a (n+1)}   \right).
\end{align}
Since $\bb E_x ( W_1(s) \log^+ W_1(s)) = \wh{\bb E}_{x} ( \log^+ W_1(s)) = \infty$, it follows that for any $a > 0$, 
\begin{align}\label{identity-P-An-03}
\sum_{n=1}^{\infty} \wh{\bb P}_{x} (A_{n+1} | \mathscr{H}_n)
\geq  c'  \sum_{n=1}^{\infty} \wh{\bb P}_{x} \left(  W_1(s) > c e^{a (n+1)}   \right)
& = c'  \sum_{n=1}^{\infty} \wh{\bb P}_{x} \left( \frac{1}{a} \log \frac{W_1(s)}{c} >  n+1   \right) \notag\\
& \geq c'' \wh{\bb E}_{x}  \left( \frac{1}{a} \log \frac{W_1(s)}{c}   \right) = \infty. 
\end{align}
Using \eqref{extended-Borel-Cantellli}, we get that $\limsup_{n \to \infty} \frac{\log U_n}{n} > a$ for any $a > 0$, $\wh{\bb P}_{x}$-a.s.
By \eqref{SLLN-spine-aa}, 
we have $\lim_{n \to \infty} \frac{S_{w|n}}{n} = -\mathfrak M'(s)$, $\wh{\bb P}_{x}$-a.s.
Therefore, in view of \eqref{Lower-W-n-aaa}, we get that 
$\limsup_{n \to \infty} W_n(s) = \infty$, $\wh{\bb P}_{x}$-a.s. 
By (2.3) of Section (2.3) in \cite{Shi12}, we conclude that $W_{\infty} = 0$,  $\bf P_{x}\mbox{-a.s.}$

(2) We prove that if $\bb E_x ( W_1(s) \log^+ W_1(s)) < \infty$ 
and there exists $s \in I$ such that $\mathfrak M(s) > s \mathfrak M'(s)$,
then $W_{\infty}(s) > 0$,  $\bf P_{x}\mbox{-a.s.}$
By rearranging summation over all particles in generation $n$ according to their last ancestor in the spine, we have
\begin{align*}
W_n(s) & = r_{s} (X_{w|n}) e^{-s S_{w|n} - n \mathfrak M(s)}   \\
	 & \quad + \sum_{k=0}^{n-1} r_s(X_{w|k}) e^{-s S_{w|k}-k \mathfrak{M}(s)} \sum_{|u|=1, (w|k)u\neq (w|k+1)} \frac{r_s(X_{(w|k) u})}{r_s(X_{w|k})} e^{-s (S_{(w|k)u}-S_{w|k})-\mathfrak{M}(s)}  \times \\ 
	 &\qquad\qquad \sum_{|v|=n-k-1} \frac{r_s(X_{(w|k) uv})}{r_s(X_{(w|k)u})} e^{-s (S_{(w|k) uv} - S_{(w|k)u}) - (n-k-1) \mathfrak{M}(s)}. 
\end{align*}
Let $\mathscr{H}_{\infty} = \sigma \{\mathscr{H}_n, n \geq 1\}$, where $\mathscr{H}_n$ is defined by \eqref{filtration-spine-child}, {\em i.e.}, spine particles and their direct offspring are measurable w.r.t. $\mathscr{H}_\infty$. Note that by the properties of the spinal decomposition (cf. Section \ref{sect:spinal-decomp}), non-spine particles produce their offspring according to the original measure. It follows that
\begin{align*}
& \wh{\bb E}_{x} (W_n(s) | \mathscr{H}_{\infty}) \\
&= r_{s} (X_{w|n}) e^{-s S_{w|n} - n \mathfrak M(s)}   \\
&~+ \sum_{k=0}^{n-1} r_s(X_{w|k}) e^{-s S_{w|k}-k \mathfrak{M}(s)} \sum_{|u|=1, (w|k)u\neq (w|k+1)} \frac{r_s(X_{(w|k) u})}{r_s(X_{w|k})} e^{-s (S_{(w|k)u}-S_{w|k})-\mathfrak{M}(s)}  \times \\ 
&~~\wh{\bb E}_{x} \left( \left. \sum_{|v|=n-k-1} \frac{r_s(X_{(w|k) uv})}{r_s(X_{(w|k)u})} e^{-s (S_{(w|k) uv} - S_{(w|k)u}) - (n-k-1) \mathfrak{M}(s)} \, \right| \, \mathcal{H}_\infty\right)\\
&= r_{s} (X_{w|n}) e^{-s S_{w|n} - n \mathfrak M(s)}   \\
&~+ \sum_{k=0}^{n-1} r_s(X_{w|k}) e^{-s S_{w|k}-k \mathfrak{M}(s)} \sum_{|u|=1, (w|k)u\neq (w|k+1)} \frac{r_s(X_{(w|k) u})}{r_s(X_{w|k})} e^{-s (S_{(w|k)u}-S_{w|k})-\mathfrak{M}(s)}  \cdot 1\\
&= r_{s} (X_{w|n}) e^{-s S_{w|n} - n \mathfrak M(s)}   \\
&~+ \sum_{k=0}^{n-1}  e^{-s S_{w|k}-k \mathfrak{M}(s)} \sum_{|u|=1, (w|k)u\neq (w|k+1)} r_s(X_{(w|k) u}) e^{-s (S_{(w|k)u}-S_{w|k})-\mathfrak{M}(s)}  \cdot 1\\
&= r_{s} (X_{w|n}) e^{-s S_{w|n} - n \mathfrak M(s)}  + \sum_{k=0}^{n-1} \ \sum_{|u|=1, (w|k)u\neq (w|k+1)} r_s(X_{(w|k) u}) e^{-s S_{(w|k)u}-(k+1)\mathfrak{M}(s)}. 
\end{align*}
Using the notation \eqref{def-Un-thm}, we get
\begin{align}\label{expression-E-Wn}
\wh{\bb E}_{x} (W_n(s) | \mathscr{H}_{\infty}) 
= \sum_{k=0}^{n-1}  r_{s} (X_{w|k}) e^{-s S_{w|k} - (k+1) \mathfrak M(s)} U_k
- \sum_{k=1}^{n-1}  r_{s} (X_{w|k}) e^{-s S_{w|k} - k \mathfrak M(s) }. 
\end{align}
Similarly to the proof of the first step, we fix a constant $\ee >0$ whose value will be taken sufficiently small, 
and for $n \geq 0$, denote
\begin{align*}
B_n = \{ \log U_n > \ee n \}. 
\end{align*}
Following the proof of \eqref{identity-P-An-01} and \eqref{identity-P-An-02}, using the fact that $r_{s}$ is bounded from above and below, 
we have that there exist constants $c, c'>0$ such that for any $x \in \bb S_+^{d-1}$, 
\begin{align*}
\wh{\bb P}_{x} (B_{n+1} | \mathscr{H}_n)
= \frac{ 1 }{  r_{s}(X_{w|n})  }   \bb E_{X_{w|n}} \left( W_1(s) \mathds 1_{\{ \log U_{0} > \ee (n+1) \}}  \right)
\leq c' \wh{\bb P}_{x} \left(  W_1(s) > c e^{\ee (n+1)}   \right).
\end{align*}
As in \eqref{identity-P-An-03}, since $\bb E_x ( W_1(s) \log^+ W_1(s)) = \wh{\bb E}_{x} ( \log^+ W_1(s)) < \infty$, it follows that for any $\ee > 0$, 
\begin{align*}
\sum_{n=1}^{\infty} \wh{\bb P}_{x} (B_{n+1} | \mathscr{H}_n)
\leq  c'  \sum_{n=1}^{\infty} \wh{\bb P}_{x} \left(  W_1(s) > c e^{\ee (n+1)}   \right)
& = c'  \sum_{n=1}^{\infty} \wh{\bb P}_{x} \left( \frac{1}{\ee} \log \frac{W_1(s)}{c} >  n+1   \right) \notag\\
& \geq c'' \wh{\bb E}_{x}  \left( \frac{1}{\ee} \log \frac{W_1(s)}{c}   \right) < \infty. 
\end{align*}
Using \eqref{extended-Borel-Cantellli},  for any $\ee > 0$, it holds that $\limsup_{n \to \infty} \frac{\log U_n}{n} < \ee$, $\wh{\bb P}_{x}$-a.s.
Since $\ee >0$ can be arbitrarily small, we get 
\begin{align*}
\lim_{n \to \infty} \frac{\log^+ U_n}{n} = 0, \quad \wh{\bb P}_{x}\mbox{-a.s.}
\end{align*}
Using again the strong law of large numbers \eqref{SLLN-spine-aa}, 
we have $\lim_{n \to \infty} \frac{S_{w|n}}{n} = -\mathfrak M'(s)$, $\wh{\bb P}_{x}$-a.s.
Since $\mathfrak M(s) > s \mathfrak M'(s)$, 
both $e^{-s S_{w|k} - (k+1) \mathfrak M(s)}$ and $e^{-s S_{w|k} - k \mathfrak M(s) }$ in \eqref{expression-E-Wn} 
decay exponentially fast as $k \to \infty$.
Using the fact that $r_{s} (X_{w|k})$ is positive and bounded, 
by the monotone convergence theorem, 
both sums in \eqref{expression-E-Wn} converge to a positive and finite limit, so that 
$\lim_{n \to \infty} \wh{\bb E}_{x} (W_n(s) | \mathcal{H}_{\infty})$ exists $\wh{\bb P}_{x}$-a.s. 
By Lemma \ref{lem-condi-limsup}, it follows that $\bb E_x (W_{\infty}) = 1$ for any $x \in \bb S_+^{d-1}$. 
Since $\bb E N >1$, we have $\bb P(\scr S) >0$. Therefore, 
 $W_{\infty}(s) > 0$, $\bf P_{x}\mbox{-a.s.}$ is equivalent to saying that $\bb P_x (W_{\infty}(s) > 0) >0$, 
and the assertion follows. 
\end{proof}

\subsection{Convergence of the derivative martingale}
The goal of this section is to establish Theorem \ref{Thm-conv-deriv-mart} on the convergence of the derivative martingale. 
We follow the proof strategy in \cite{Shi12}.

Let us fix $\beta \geq 0$. For any $x \in \bb S_+^{d-1}$ and $y \geq -\beta$, define
\begin{align}\label{def-V-alpha-beta-xy}
V_{\alpha}^{\beta} (x, y) : = V_{\alpha} (x, y + \beta),
\end{align}
where $V_{\alpha}$ is defined by \eqref{def-V-alpha-xy}. 
By Lemma \ref{Lem-positi-harmonic-func}, 
there exist constants $c, c' >0$ such that for any $x \in \bb S_+^{d-1}$ and $y \geq -\beta$, 
\begin{align}\label{inequa-V-alpha}
c' (1 + y + \beta)  \leq  V_{\alpha}^{\beta} (x, y)  \leq  c (1 + y + \beta),
\end{align}
and for any $x \in \bb S_+^{d-1}$, we have $\lim_{y \to \infty} \frac{ V_{\alpha}^{\beta} (x, y) }{y} = 1$. 
 Similarly to \eqref{def-martingale-Dn}, we define 
\begin{align}\label{def-martingale-Dn-V}
M_n^V(\beta) := M_n^{V_{\alpha}^{\beta}} = \sum_{ |u| = n }  r_{\alpha}(X_u)  V_{\alpha}^{\beta}(X_u, S_u)  e^{-\alpha S_u}
 \mathds{1}_{ \left\{  S_{u|k} \geq -\beta,  \, \forall  \,   0 \leq k \leq n  \right\} },  
\quad  n \geq 0. 
\end{align}
Note that this corresponds to \eqref{def-martingale-Dn} by choosing $h = V_{\alpha}^{\beta}$ and $B = [-\beta, \infty)$. 
Since $\alpha$ is fixed, we use the shorter notation $M_n^V(\beta)$. 
By Lemma \ref{Lem-martiangle-Dn}, 
$(M_n^V(\beta),  \mathscr F_n)_{n \geq 0}$ is a non-negative martingale, 
so there exists a random variable $M_{\infty}^V(\beta) \geq 0$ such that, as $n \to \infty$,
 $M_n^V(\beta) \to M_{\infty}^V(\beta)$, $\bb P_{x}$-a.s. 

Let $b \in (- \infty, \beta]$. 
As in \eqref{def-hat-P-x-b}, 
there exists a unique probability measure $\wh{\bb P}^{(\beta)}_{x, b}$ on $\mathscr F_{\infty}$ such that,
for any $A \in \mathscr F_n$ and $n \geq 0$, 
\begin{align}\label{def-P-beta}
\wh{\bb P}^{(\beta)}_{x, b} (A) 
=  \frac{ e^{\alpha b} }{  r_{\alpha}(x) V_{\alpha}^{\beta} (x, b)   }  \bb E_{x,b}  \left( M_n^V(\beta) \mathds 1_A \right). 
\end{align}
Denote by $\wh{\bb E}^{(\beta)}_{x, b}$ the corresponding expectation.

\begin{lemma}
Assume conditions \ref{Condi-Furstenberg-Kesten}, \ref{CondiNonarith}, \ref{Condi_ms} and \ref{Condi_harmonic}.  
Then, for any $\ee \in (0, 1/6)$, 
we have $S_{w|n} \geq  n^{\ee}$ for all sufficiently large $n$,  $\wh{\bb P}^{(b)}_{x}$-a.s.
\end{lemma}

\begin{proof}
Fix $\ee \in (0, 1/4)$. 
By \eqref{thm2-spinal-condi} of Theorem \ref{Thm-spinal-decom},  we have  
\begin{align*}
\wh{\bb P}^{(b)}_{x}  (S_{w|n} <  n^{\ee})
=  \bb{E}_{ \bb Q_{x,b}^{\alpha} } 
        \Bigg[ \frac{ V_{\alpha} \left( X_n, S_n \right)  }
                  { V_{\alpha}(x, b)  e^{\alpha b} }  
                   \mathds 1_{ \{ S_n <  n^{\ee} \} } 
            \mathds{1}_{ \left\{ S_{k} \geq 0,  \, \forall  \,   0 \leq k \leq n  \right\} }
           \Bigg]. 
\end{align*}
Using \eqref{inequa-V-alpha}, we have $V_{\alpha} \left( X_n, S_n \right) \leq c (1 + S_n)$. 
Therefore,
\begin{align*}
\wh{\bb P}^{(b)}_{x}  (S_{w|n} < n^{\ee})
& \leq  c  \bb{E}_{ \bb Q_{x,b}^{\alpha} } 
        \Big[ (1 + S_n)  \mathds 1_{ \{ S_n < n^{\ee} \} } 
            \mathds{1}_{ \left\{ S_{k} \geq 0,  \, \forall  \,   0 \leq k \leq n  \right\} }
           \Big]   \notag\\
& \leq  c  n^{\ee}   \bb Q_{x,b}^{\alpha}  \left( S_n < n^{\ee},  S_{k} \geq 0,  \, \forall  \,   0 \leq k \leq n   \right)  \notag\\
& = c  n^{ \ee}  \bb Q_{x,b}^{\alpha}  \left( S_n \in [0,  n^{ \ee}),   \tau_0 > n   \right). 
\end{align*}
By Lemma \ref{Lem-CLLT-bound}, it holds
\begin{align*}
\bb Q_{x,b}^{\alpha}  \left( S_n \in [0,  n^{ \ee}),   \tau_0 > n   \right)
\leq   \frac{c}{ n^{3/2 - 2\ee} }. 
\end{align*}
Thus, 
\begin{align*}
\wh{\bb P}^{(b)}_{x}  (S_{w|n} < n^{ \ee})
\leq   \frac{c}{ n^{3/2 - 3\ee} }. 
\end{align*}
Since $\ee \in (0, 1/6)$, 
by Borel-Cantelli lemma, we get that $S_{w|n} \geq  n^{\ee}$ for all sufficiently large $n$, $\wh{\bb P}^{(b)}_{x}$-a.s., 
completing the proof of the lemma. 
\end{proof}

Now we prove the $L^1$ convergence of $M_n^V(\beta)$. 

\begin{lemma}\label{Lem-DnV}
Assume conditions \ref{Condi-Furstenberg-Kesten}, \ref{CondiNonarith}, \ref{Condi_ms}, \ref{Condi_harmonic} and \ref{condi:momentsW1}.  
Then, for any $x \in \bb S_+^{d-1}$ and $\beta \geq 0$,  
we have 
\begin{align*}
\lim_{n \to \infty}  \bb E_{x}  M_n^V(\beta) = \bb E_{x} M_{\infty}^V(\beta). 
\end{align*}
\end{lemma}

\begin{proof}
Let $ \mathscr{G}_{\infty}$ be the natural filtration of the spine particles and their brothers, i.e.
\begin{align}\label{def-scr-G-infty}
\mathscr{G}_{\infty} = \sigma \left\{ w|k, X_{w|k}, S_{w|k}, (X_u, S_u)_{ u \in {\rm brot}(w|k) } ; k \geq 0 \right\}. 
\end{align}
By Lemma \ref{lem-condi-limsup}, it suffices to prove that for any $x \in \bb S_+^{d-1}$ and $\beta \geq 0$, 
\begin{align*}
\liminf_{n \to \infty} \wh{\bb P}^{(\beta)}_{x} \left( M_n^V(\beta) |  \mathscr{G}_{\infty}  \right)  < \infty,  \quad   \wh{\bb P}^{(\beta)}_{x}\mbox{-a.s.}
\end{align*}
Using the martingale property of $M_n^V(\beta)$ for the subtrees rooted at the brothers of the spine, 
we get that under $\wh{\bb P}^{(\beta)}_{x}$, 
\begin{align}\label{equa-martin-spine}
\wh{\bb E}^{(\beta)}_{x} \left( M_n^V(\beta) |  \mathscr{G}_{\infty}  \right)
& =  r_{\alpha}(X_{w|n}) V_{\alpha}^{\beta} (X_{w|n}, S_{w|n})  e^{-\alpha S_{w|n}}     \notag\\
& \quad   +  \sum_{k=1}^n  \sum_{ u \in {\rm brot}(w|k) }  r_{\alpha}(X_u) V_{\alpha}^{\beta} (X_u, S_u)  e^{-\alpha S_u} 
       \mathds 1_{ \{ S_{u|j} \geq -\beta,  \forall  0 \leq  j \leq k \} }.  
\end{align}
By Theorem \ref{Thm-spinal-decom}, 
the process $(X_{w|n}, S_{w|n})_{n \geq 0}$ under $\wh{\bb P}^{(\beta)}_{x}$ is distributed as the Markov random walk
$(X_n, S_n)_{n \geq 0}$, under $\bb Q_{x}^{\alpha}$, 
conditioned to stay non-negative (meaning that $\min_{1 \leq k \leq n} S_k \geq 0$),
so that $S_{w|n} \to \infty$ as $n \to \infty$, under $\wh{\bb P}^{(\beta)}_{x}$.  
Hence, for the first term on the right hand side of \eqref{equa-martin-spine},
 using \eqref{inequa-V-alpha} and the fact that $r_{\alpha}$ is bounded on $\bb S_+^{d-1}$,
we get that $r_{\alpha}(X_{w|n}) V_{\alpha}^{\beta} (X_{w|n}, S_{w|n})  e^{-\alpha S_{w|n}} \to 0$ 
as $n \to \infty$, under the measure $\wh{\bb P}^{(\beta)}_{x}$. 
For the second term on the right hand side of \eqref{equa-martin-spine},
we use \eqref{inequa-V-alpha} to get
\begin{align}\label{ine-V-alpha-beta-3}
V_{\alpha}^{\beta} (X_u, S_u)  \mathds 1_{ \{ S_{u|j} \geq -\beta,  \forall  0 \leq  j \leq k \} }
\leq  c \left( 1 + (\beta + S_u)^+ \right), 
\end{align}
where we write $t^+ = \max \{ t, 0 \}$ for $t \in \bb R$. 
It follows that 
\begin{align*}
r_{\alpha}(X_u) V_{\alpha}^{\beta} (X_u, S_u)  e^{-\alpha S_u} 
       \mathds 1_{ \{ S_{u|j} \geq -\beta,  \forall  0 \leq  j \leq k \} }
   \leq  c \left( 1 + (\beta + S_u)^+ \right)  e^{- \alpha S_u }.  
\end{align*}
Hence, 
\begin{align}\label{inequa-DnV-aa}
\liminf_{n \to \infty} \wh{\bb E}^{(\beta)}_{x} \left( M_n^V(\beta) |  \mathscr{G}_{\infty}  \right)
\leq  c  \sum_{k=1}^{\infty}  \sum_{ u \in {\rm brot}(w|k) }  \left( 1 + (\beta + S_u)^+ \right)  e^{ -\alpha S_u }. 
\end{align}
Since $(\beta + S_u)^+ \leq   \beta + S_{w|k-1} +  ( S_u - S_{w|k-1} )^+$, 
in order to prove that the right hand side of \eqref{inequa-DnV-aa} is $\wh{\bb P}^{(\beta)}_{x}$-a.s. finite,
it suffices to show that $\wh{\bb P}^{(\beta)}_{x}$-a.s.,
\begin{align}
& \sum_{k=1}^{\infty}  \left( 1 + \beta + S_{w|k-1} \right)  e^{ -\alpha S_{w|k-1} }  
    \sum_{ u \in {\rm brot}(w|k) }    e^{ - \alpha (S_u - S_{w|k-1} ) }  < \infty,   \label{ine-suffice-01} \\
& \sum_{k=1}^{\infty}  e^{ - \alpha S_{w|k-1} } 
  \sum_{ u \in {\rm brot}(w|k) }  ( S_u - S_{w|k-1} )^+   e^{ - \alpha (S_u - S_{w|k-1} ) }  < \infty.   \label{ine-suffice-02}
\end{align}
We first prove \eqref{ine-suffice-01}. 
Let $b \geq -\beta$. 
By \eqref{def-martingale-Dn-V} and \eqref{ine-V-alpha-beta-3}, we have 
\begin{align}\label{bound-D1-V-beta}
M_1^V(\beta)  
 & \leq  c   \sum_{ |u| = 1 }   e^{- \alpha S_u}  \left( 1 +  (\beta + S_u)_+  \right)    \notag\\
 & \leq  c   \sum_{ |u| = 1 }   e^{ - \alpha S_u}  \left[ 1 +  \beta + b +  (S_u - b)^+  \right], 
\end{align}
where in the last inequality we used $(\beta + S_u)^+ \leq   \beta + b +  (S_u - b)^+$.  
By \eqref{inequa-V-alpha} and the fact that $\inf_{x \in \bb S_+^{d-1} } r_{\alpha}(x) >0$, 
we have $\frac{1}{ r_{\alpha}(x)  V_{\alpha}^{\beta} (x, b) } \leq  \frac{c}{ 1 +  \beta + b }$. 
Therefore, from \eqref{def-P-beta} and \eqref{bound-D1-V-beta}, it follows that for any $z \in \bb R$ and $b \geq -\beta$, 
\begin{align}\label{Pxb-exp-Su}
& \wh{\bb P}^{(\beta)}_{x,b}  \bigg( \sum_{ |u| = 1 }  e^{- \alpha (S_u - b)}  > z  \bigg)   \notag\\
& =  \frac{ e^{\alpha b} }{  r_{\alpha}(x) V_{\alpha}^{\beta} (x, b) }  
  \bb E_{x,b}  \bigg( M_1^V(\beta) \mathds 1_{ \big\{ \sum_{ |u| = 1 }  e^{-\alpha (S_u - b)}  > z \big\} } \bigg)  \notag\\
& \leq  \frac{ c e^{\alpha b} }{ 1 +  \beta + b }
    \bb E_{x,b}  \bigg( \sum_{ |u| = 1 }   e^{-\alpha S_u}  \left[ 1 +  \beta + b +  (S_u - b)^+  \right]
       \mathds 1_{ \big\{  \sum_{ |u| = 1 }  e^{-\alpha (S_u - b)}  > z  \big\} } \bigg)  \notag\\
& =  c   \bb E_{x,b}  \bigg( \sum_{ |u| = 1 }   e^{-\alpha (S_u - b) }  \left[ 1 + \frac{ (S_u - b)^+ }{ 1 +  \beta + b  } \right]
       \mathds 1_{ \big\{ \sum_{ |u| = 1 }  e^{-\alpha (S_u - b)}  > z \big\} } \bigg)   \notag\\
& = c  \bb E_{x}  \bigg( \sum_{ |u| = 1 }   e^{-\alpha S_u }  \left[ 1 + \frac{  S_u^+ }{ 1 +  \beta + b  } \right]
       \mathds 1_{ \big\{ \sum_{ |u| = 1 }  e^{-\alpha  S_u }  > z \big\} } \bigg). 
\end{align}
Denote $Z = \sum_{ |u| = 1 }   e^{-\alpha S_u }$ and $\widetilde{Z} = \sum_{ |u| = 1 } S_u^+  e^{- \alpha S_u } $.
Then \eqref{Pxb-exp-Su} can be rewritten as follows: for any $z \in \bb R$ and $b \geq -\beta$, 
\begin{align}\label{rewrite-Z-tilde-Z}
\wh{\bb P}^{(\beta)}_{x,b}  \bigg( \sum_{ |u| = 1 }  e^{- \alpha (S_u - b)}  > z  \bigg)
\leq  c \bb E_{x} \left( Z  \mathds 1_{ \left\{ Z  > z \right\} }  \right)
   +  \frac{c}{ 1 +  \beta + b }  \bb E_{x} \left( \widetilde{Z}  \mathds 1_{ \left\{ Z  > z \right\} }  \right).  
\end{align}
In the same way as in the proof of \eqref{Pxb-exp-Su} and \eqref{rewrite-Z-tilde-Z}, one has 
\begin{align}\label{Z-tilde-Z-second}
\wh{\bb P}^{(\beta)}_{x,b}  \bigg( \sum_{ |u| = 1 }  (S_u - b)^+  e^{\alpha (S_u - b)}  > z  \bigg)
\leq  c \bb E_{x} \Big( Z  \mathds 1_{ \{ \widetilde{Z}  > z \} }  \Big)
   +  \frac{c}{ 1 +  \beta + b }  \bb E_{x} \Big( \widetilde{Z}  \mathds 1_{ \{ \widetilde{Z}  > z \} }  \Big).  
\end{align}
Now let us fix $x' \in \bb S_+^{d-1}$. 
Let $c_0 > 0$ be as in Lemma \ref{Lem-contractivity}. 
Let $Z^*$ and $\widetilde{Z}^*$ be random variables that are independent of the rest of the world
and that have -- under any probability measures introduced -- the same law as 
$\sum_{ |u| = 1 }   e^{- \alpha (\log |g_u x'| - c_0) }$
and $\sum_{ |u| = 1 } ( \log |g_u x'| - c_0)_+  e^{- \alpha (\log |g_u x'| - c_0) } $ under $\bb P$, respectively.  
In particular, the law of $Z^*$ and $\widetilde{Z}^*$ is not changed if we go to $\bb Q_x^{\alpha}$.
Let $f_{1}(z) = \bb E ( Z^*  \mathds 1_{ \left\{ Z^*  > z \right\} }  )$
and $f_{2}(z) = \bb E ( \widetilde{Z}^*  \mathds 1_{ \left\{ Z^*  > z \right\} } )$, $z \in \bb R$. 
Under condition \ref{Condi-Furstenberg-Kesten},
we have that for any $x \in \bb S_+^{d-1}$, 
\begin{align*}
\bb E_{x} \Big( Z  \mathds 1_{ \left\{ Z  > z \right\} }  \Big)
 & = \bb E_{x} \bigg( \sum_{ |u| = 1 }   e^{- \alpha S_u }  \mathds 1_{ \big\{ \sum_{ |u| = 1 }   e^{- \alpha S_u }  > z \big\} }  \bigg)   \notag\\
& \leq  \bb E_{x'} \bigg( \sum_{ |u| = 1 }   e^{- \alpha (S_u - c_0) }  \mathds 1_{ \big\{ \sum_{ |u| = 1 }   e^{- \alpha (S_u - c_0) }  > z \big\} }  \bigg)
 = f_{1}(z), 
\end{align*}
and similarly, 
\begin{align*}
\bb E_{x} \left( \widetilde{Z}  \mathds 1_{ \left\{ Z  > z \right\} }  \right)
\leq  f_{2}(z). 
\end{align*}
Letting $\lambda \in (0, \alpha)$ be a fixed number, taking $z = e^{ \lambda b}$ in \eqref{rewrite-Z-tilde-Z} 
and using the Markov property, 
we get 
\begin{align*}
&  \wh{\bb P}^{(\beta)}_{x}   \bigg(  \sum_{ u \in {\rm brot}(w|k) }    e^{ -\alpha (S_u - S_{w|k-1} ) }  >  e^{ \lambda S_{w|k-1} }  \bigg)  \notag\\
& \leq  c  \bb E_{ \wh{\bb P}^{(\beta)}_{x} }
   \bigg[  f_{1} \left( e^{ \lambda S_{w|k-1} } \right)  
      +   \frac{ f_{2} \left( e^{ \lambda S_{w|k-1} } \right) }{ 1 + \beta + S_{w|k-1} }   \bigg]
  = : c I_k. 
\end{align*}
By Theorem \ref{Thm-spinal-decom} and \eqref{def-P-beta}, 
the process $(X_{w|n}, S_{w|n})_{n \geq 0}$ under $\wh{\bb P}^{(\beta)}_{x}$ is distributed as the Markov random walk
$(X_n, S_n)_{n \geq 0}$ under the measure $\bb Q_{x}^{\alpha}$, 
conditioned to stay in $[ -\beta, \infty)$ (meaning that $\min_{1 \leq k \leq n} S_k \geq -\beta$).
Hence, in view of \eqref{thm2-spinal-condi}, we get 
\begin{align*}
I_k &  =  \frac{1}{ V_{\alpha}^{\beta} (x, 0) }
  \bb{E}_{ \bb Q_{x}^{\alpha} } 
        \bigg[  \bigg(  f_{1} \big( e^{ \lambda S_{k-1} } \big)  
         +   \frac{ f_{2} \big( e^{ \lambda S_{k-1} } \big) }{ 1 + \beta + S_{k-1} }   \bigg)
            V_{\alpha}^{\beta} \left( X_{k-1}, S_{k-1} \right)   
            \mathds{1}_{ \{ \min_{0 \leq i \leq k-1} S_i \geq -\beta \} }
           \bigg].      \notag\\
& = \frac{1}{ V_{\alpha}^{\beta} (x, 0) }
  \bb{E}_{ \bb Q_{x}^{\alpha} } 
        \bigg[  \bigg(  Z^*  \mathds 1_{ \{ S_{k-1}  <  \frac{1}{\lambda} \log Z^* \} }  
           +   \frac{ \widetilde{Z}^*  \mathds 1_{ \{ S_{k-1}  <  \frac{1}{\lambda} \log Z^* \} }  }{ 1 + \beta + S_{k-1} }   \bigg)  \notag\\
  & \qquad\qquad\qquad\qquad\qquad\qquad  \times        V_{\alpha}^{\beta} \left( X_{k-1}, S_{k-1} \right)   
            \mathds{1}_{ \{  \min_{0 \leq i \leq k-1} S_i \geq -\beta  \} }
           \bigg], 
\end{align*}
where $(S_i)_{i \geq 0}$ is independent of the pair $(Z^*, \widetilde{Z}^*)$.  
By \eqref{inequa-V-alpha}, it follows that 
\begin{align*}
I_k  & \leq  \frac{c}{ V_{\alpha}^{\beta} (x, 0) }
  \bb{E}_{ \bb Q_{x}^{\alpha} } 
        \Bigg[  \left(  \left( 1 + \beta +  \frac{1}{\lambda} \log Z^* \right) Z^*  \mathds 1_{ \left\{ S_{k-1}  <  \frac{1}{\lambda} \log Z^* \right\} }  
           +    \widetilde{Z}^*  \mathds 1_{ \left\{ S_{k-1}  <  \frac{1}{\lambda} \log Z^* \right\} }    \right)    \notag\\
& \qquad\qquad\qquad\qquad\qquad\qquad  \times 
            \mathds{1}_{ \{ \min_{0 \leq i \leq k-1} S_i \geq -\beta \} }
           \Bigg]. 
\end{align*}
By Lemma \ref{Lem-Renewal-thm}, 
there exists $c>0$ such that for any $x \in \bb S_+^{d-1}$, $\beta \geq 0$, $\lambda \in (0, \alpha)$ and $z^* > 0$, 
\begin{align*}
& \sum_{k=1}^{\infty}  \bb Q_{x}^{\alpha} \Big( S_{k-1}  <  \frac{1}{\lambda} \log z^*,  \min_{ 0 \leq i \leq k-1 } S_i \geq -\beta   \Big)
   \notag\\
& =  \sum_{k=1}^{\infty}  \bb Q_{x}^{\alpha} \Big( \beta + S_{k-1}  \in \Big[ 0, \beta + \frac{1}{\lambda} \log z^*  \Big),  
    \min_{ 0 \leq i \leq k-1 } (\beta + S_i) \geq  0   \Big)
  \notag\\
 & \leq  \sum_{k=1}^{\infty}  \bb Q_{x}^{\alpha} \Big( \beta + S_{k-1}  \in \Big[ 0, \beta + \frac{1}{\lambda} \log^+ z^*  \Big),  
    \min_{ 0 \leq i \leq k-1 } (\beta + S_i) \geq  0   \Big)
  \notag\\
& \leq  c (1 + \beta)  \Big( 1 + \beta + \frac{1}{\lambda} \log^+ z^* \Big). 
\end{align*}
Therefore, using the fact that $Z^*$ is independent of $(X_k, S_k)_{k \geq 1}$, 
 we obtain
\begin{align*}
\sum_{k = 1}^{\infty} I_k 
 \leq  \frac{c (1 + \beta)}{ V_{\alpha}^{\beta} (x, 0) }
   \bb{E}_{ \bb Q_{x}^{\alpha} } 
        \Big[  \Big( 1 + \beta +  \frac{1}{\lambda} \log^+ Z^* \Big)^2  Z^*  
            +    \Big( 1 + \beta +  \frac{1}{\lambda} \log^+ Z^* \Big)    \widetilde{Z}^*  
           \Big] < \infty,
\end{align*}
where the last inequality holds due to the moment condition 
\begin{align}\label{condi-moment-mart}
\bb{E}_{ \bb Q_{x}^{\alpha} } 
        \left[   \left( 1 + \log^+ Z^* \right)^2  Z^*  
            +    \left( 1 + \log^+ Z^* \right)    \widetilde{Z}^*  
           \right] < \infty. 
\end{align}
Using the Borel-Cantelli lemma, for any $\lambda \in (0, \alpha)$ and sufficiently large $k \geq 1$, 
$\wh{\bb P}^{(\beta)}_{x}$-a.s., 
\begin{align*}
\sum_{ u \in {\rm brot}(w|k) }    e^{ -\alpha (S_u - S_{w|k-1} ) }  \leq  e^{ \lambda S_{w|k-1} }.   
\end{align*}
Therefore, using the fact that for any $\ee >0$ and sufficiently large $k \geq 1$, 
$S_{w|k-1} \geq (k-1)^{1/2 - \ee}$, $\wh{\bb P}^{(\beta)}_{x}$-a.s., 
we get \eqref{ine-suffice-01}. 

The proof of \eqref{ine-suffice-02} is similar 
by using \eqref{Z-tilde-Z-second} instead of \eqref{rewrite-Z-tilde-Z},
and $\bb{E}_{ \bb Q_{x}^{\alpha} } 
        [   ( 1 + \log^+ \widetilde{Z} )^2  Z  
            +   ( 1 + \log^+ \widetilde{Z} )    \widetilde{Z}  
           ] < \infty$ instead of \eqref{condi-moment-mart}. 
Thus we finish the proof of the lemma. 
\end{proof}

The following lemma shows that $(D_n, \scr F_n)$ is a martingale, where $D_n$ is defined by \eqref{def-derivative-martingale}
 and $\scr F_n$ is given by \eqref{def-filtration-Fn}.

\begin{lemma}\label{Lem-derivative-mart}
Assume condition \ref{Condi_ms}. 
Then, for any $x \in \bb S_+^{d-1}$, 
we have that $(D_n, \scr F_n)_{n\geq 1}$ is a $\bb P_x$-martingale with mean $ r_\alpha(x)\ell_{\alpha}(x)$. 
\end{lemma}

\begin{proof}
For any $x \in \bb S_+^{d-1}$ and $n \geq 0$, we have
\begin{align*}
 \bb E_x (D_{n+1} | \scr F_n)  
& = \bb E_x  \Big( \sum_{|v|=n} \sum_{|u| = 1}  \big( S_{vu} + \ell_{\alpha} (X_{vu}) \big) e^{- \alpha S_{vu}}  r_{\alpha}(X_{vu})  \Big| \scr F_n \Big) \notag\\
 & =  \sum_{|v|=n} S_{v}  e^{- \alpha S_{v}}   \bb E_x  \Big(  \sum_{|u| = 1}   
 e^{- \alpha (S_{vu} - S_{v}) } r_{\alpha}(X_{vu})  \Big| \scr F_n \Big) \notag\\
& \quad +  \sum_{|v|=n} e^{- \alpha S_{v}}   \bb E_x  \Big(  \sum_{|u| = 1}  \big( (S_{vu} - S_{v})  + \ell_{\alpha} (X_{vu}) \big) 
 e^{- \alpha (S_{vu} - S_{v}) } r_{\alpha}(X_{vu})  \Big| \scr F_n \Big).
\end{align*}
Since $P_{\alpha} r_{\alpha} = \mathfrak{m}(\alpha) r_{\alpha}$ and $\mathfrak{m}(\alpha) = 1$, we get
\begin{align*}
 \bb E_x  \Big(  \sum_{|u| = 1}   
 e^{- \alpha (S_{vu} - S_{v}) } r_{\alpha}(X_{vu})  \Big| \scr F_n \Big)
 = \bb E_{X_v}  \Big(  \sum_{|u| = 1}   
 e^{- \alpha S_{u}  } r_{\alpha}(X_{u})  \Big)
 = r_{\alpha}(X_v). 
\end{align*}
Using \eqref{expect-Su-ell-Xu}, we see that 
\begin{align*}
& \bb E_x  \Big(  \sum_{|u| = 1}  \big( (S_{vu} - S_{v})  + \ell_{\alpha} (X_{vu}) \big) 
 e^{- \alpha (S_{vu} - S_{v}) } r_{\alpha}(X_{vu})  \Big| \scr F_n \Big) \notag\\
& = \bb E_{X_v}  \Big(  \sum_{|u| = 1}  \big( S_{u}   + \ell_{\alpha} (X_{u}) \big) 
 e^{- \alpha S_{u} } r_{\alpha}(X_{u})  \Big)
 = r_{\alpha}(X_v)  \ell_{\alpha} (X_v). 
\end{align*}
This completes the proof. 
\end{proof}

\begin{lemma}\label{Lem-infinimum-Su}
Assume \ref{Condi-Furstenberg-Kesten} and \ref{Condi_ms} (in fact, $\bb E N>1$ and  $\mathfrak{m}(\alpha)\le1$ suffice). Then for all $x \in \bb S^{d-1}_+$,
	\begin{align*}
		\lim_{n \to \infty} \inf_{|u| = n} S_u =\infty ,  \quad  \inf_{u \in \bb T} S_u > - \infty,  \quad   \bf P_x\mbox{-a.s.}
	\end{align*}
\end{lemma}

\begin{proof}
	We follow the proof in \cite{Big98}. Consider the random variable
	$$ Y := r_\alpha(X_{\o})^{-1} \limsup_{n \to \infty} \exp(- \alpha\inf_{|u|=n} S_u).$$
	Let $c:= \sup_{x,y \in \bb S^{d-1}_+} r_\alpha(x)/r_\alpha(y)$. Then it holds
	$$ Y \le r_\alpha(X_{\o})^{-1} \limsup_{n \to \infty} \sum_{|u|=n} e^{-\alpha S_u} \le c \limsup_{n \to \infty}  W_n = c W_\infty.$$
	In particular, $\bb E_x Y \le c \bb E_x W_\infty < \infty$ for all $x \in \bb S^{d-1}_+$.  
	Observe that the assertion of the lemma follows once we have proved that $F(x):=\bb E_x Y =0$ for all $x \in \bb S^{d-1}_+$. 

	By Lemma \ref{Lem-contractivity}, we have the following dichotomy: 
\begin{align} \label{KS-alternative001}
	\mbox{either} \ \bb E_x Y =0  \  \mbox{for all} \  x \in \bb S^{d-1}_+  \quad 
	\mbox{or}\quad \bb E_x Y >0  \  \mbox{for all}  \  x \in \bb S^{d-1}_+.
\end{align}
	Hence we may study the average $\bb E_{\pi_\alpha} Y := \int  F(x) \, \pi_\alpha(dx)$, with $\pi_{\alpha}$ 
	defined in \eqref{invar mes for Q_s-001}.
	If we suppose $\bb E_{\pi_\alpha}(Y)>0$, then, by the above considerations, 
	$\bb P_x(Y>0)>0$ for all $x \in \bb S^{d-1}_+$. 

	Decomposing at the first generation, we have
	$$
	Y =  r_\alpha(X_{\o})^{-1} \sup_{|u|=1} e^{-\alpha S_u} r_\alpha(X_u) [Y]_u,
	$$ 
	where $[Y]_u$ are independent conditioned on $\scr F_1$.
	Using that, due to the assumption $\bb E N>1$, $\sup$ is strictly smaller than the sum, we then have
	\begin{align*}
		\bb E_{\pi_\alpha} \big(Y\big) &= \bb E_{\pi_\alpha} \big(  r_\alpha(X_{\o})^{-1} \sup_{|u|=1} e^{-\alpha S_u} r_\alpha(X_u) [Y]_u \big) \\
		&<  \bb E_{\pi_\alpha} \big(  r_\alpha(X_{\o})^{-1} \sum_{|u|=1} e^{-\alpha S_u} r_\alpha(X_u) [Y]_u\big) \\
		&= \bb E_{\pi_\alpha} \bigg(  r_\alpha(X_{\o})^{-1} \sum_{|u|=1} e^{-\alpha S_u} r_\alpha(X_u) \bb E  [Y]_u \, \big| \scr{F}_1 \big) \bigg) \\
		&=\bb E_{\pi_\alpha} \bigg(  r_\alpha(X_{\o})^{-1} \sum_{|u|=1} e^{-\alpha S_u} r_\alpha(X_u) F(X_u) \bigg)  \\
		&= \int_{\bb S^{d-1}_+} Q_\alpha F(x) \pi_\alpha(dx) 
		 = \int_{\bb S^{d-1}_+} F(x) \pi_\alpha(dx) = \bb E_{\pi_\alpha} (Y).
	\end{align*}
	Here we have used the fact  that $\pi_\alpha$ is an invariant measure for the Markov operator $Q_\alpha$. We have obtained a contradiction. Hence, it must hold $\bb E_{\pi_\alpha}(Y)=0$.  Then, by \eqref{KS-alternative001}, 
	we have $[Y]_u=0$ a.s. for all $u$ with $|u|=1$ and all terms would be equal to 0.
\end{proof}

Now we are equipped to establish Theorem \ref{Thm-conv-deriv-mart}. 

\begin{proof}[Proof of Theorem \ref{Thm-conv-deriv-mart}]
By Lemma \ref{Lem-infinimum-Su}, 
we have that for any $x \in \bb S_+^{d-1}$, 
\begin{align}\label{min-Su-infty}
\min_{|u| = n} S_u \to  \infty,  \quad  \inf_{u \in \bb T} S_u > - \infty,  \quad   \bf P_x\mbox{-a.s.}
\end{align}
It follows that, for any $\ee >0$, there exists $b = b(\ee) >0$ such that 
$\bb P_x ( \inf_{u \in \bb T} S_u > - b ) \geq 1 - \ee$.

Recall that in \eqref{def-martingale-Dn-V}, 
we have defined a non-negative $\bb P_{x}$-martingale $(M_n^V(\beta),  \mathscr F_n)_{n \geq 0}$,
so there exists a random variable $M_{\infty}^V(\beta)  \geq 0$ such that $M_n^V(\beta)  \to M_{\infty}^V(\beta)$, $\bb P_{x}$-a.s. 
On the other hand, under the measure $\bb P_x$, on the set $\{ \inf_{u \in \bb T} S_u > - b \}$, 
 $M_n^V(\beta)$ coincides with 
\begin{align}\label{def-tilde-D-n-V-beta}
\widetilde{D}_n^V(\beta) := \sum_{ |u| = n }   r_{\alpha}(X_u)  V_{\alpha}^{\beta}\left( X_u, S_u \right)  e^{- \alpha S_u}. 
\end{align}
Since $\lim_{y \to \infty} \frac{ V_{\alpha}^{\beta} (x, y) }{y} = 1$ uniformly in $x \in \bb S_+^{d-1}$ (cf.\ Lemma \ref{Lem-positi-harmonic-func}), 
using \eqref{min-Su-infty} and the boundedness of $\ell_{\alpha}$, we get that as $|u| = n \to \infty$, 
\begin{align}\label{ratio-Su-lu}
\frac{ V_{\alpha}^{\beta}\left( X_u, S_u \right) }{ S_u  +\ell_{\alpha}(X_u) } 
=  \frac{ V_{\alpha}^{\beta}\left( X_u, S_u \right) }{ S_u }  \frac{ S_u   }{ S_u +\ell_{\alpha}(X_u) }
\to 1, 
\end{align}
which implies that $\widetilde{D}_n^V(\beta)$ is equivalent to 
\begin{align*}
D_n = \sum_{ |u| = n }  r_{\alpha}(X_u) \left(S_u  +\ell_{\alpha}(X_u)\right)  e^{-\alpha S_u}. 
\end{align*}
Therefore, with probability $\bb P_{x}$ at least $1 - \ee$, 
$D_n$ converges to a finite limit. 
Since $\ee > 0$ is arbitrary, this proves the almost sure convergence of $D_n$. 

Now we show that $\bf P_{x} (D_{\infty} > 0) > 0$. 
Let $\beta = 0$ and consider 
\begin{align*}
M_n^V(0) := \sum_{ |u| = n }  r_{\alpha}(X_u)  V_{\alpha}^0 (X_u, S_u)  e^{- \alpha S_u}
 \mathds{1}_{ \left\{  S_{u|k} \geq 0,  \, \forall  \,   0 \leq k \leq n  \right\} },  
\quad  n \geq 0. 
\end{align*}
By Lemma \ref{Lem-DnV}
and the fact that $\bb E_{x}  M_n^V(0) = 1$, we have $\lim_{n \to \infty}  \bb E_{x}  M_n^V(0) = \bb E_{x} M_{\infty}^V(0) = 1$. 
This implies 
\begin{align}\label{prob-D-infty-V}
\bf P_{x} \left( M_{\infty}^V(0) > 0 \right) > 0. 
\end{align}
By \eqref{inequa-V-alpha}, there exists $c >0$ such that for any $x \in \bb S_+^{d-1}$ and $y \geq 0$, 
$V_{\alpha}^{0} (x, y)  \leq  c (1 + y)$. 
Hence, in view of \eqref{def-addi-martigale}, 
\begin{align*}
M_n^V(0) \leq  c \sum_{ |u| = n }  r_{\alpha}(X_u)   (1 + S_u)  e^{-\alpha S_u}
 \leq  c \bigg( W_n  +  \sum_{ |u| = n } (S_u)_{+}  e^{-\alpha S_u}  \bigg). 
\end{align*}
Since $\lim_{n \to \infty} W_n = 0$, $\bb P_{x}$-a.s., 
letting $n \to \infty$, we get $M_{\infty}^V(0) \leq  c D_{\infty}.$
This, together with \eqref{prob-D-infty-V}, implies that $\bf P_{x} (D_{\infty} > 0) > 0$.  
\end{proof}


\section{Proof of the Seneta-Heyde scaling}\label{Sec-Seneta-Heyde}

The main goal of this section is to establish Theorem \ref{Thm-Seneta-Heyde}. 
Following the approach introduced in \cite{BM19}, 
we rely on a key asymptotic result for the event $\{\max_{1 \leq i \leq n} S_i \leq 0\}$ 
under the tilted measure $\bb Q_{x, b}^{\alpha}$. Recall that $V_{\alpha}$ is defined in \eqref{def-V-alpha-xy}.

\begin{lemma}\label{Lem-exit-time-alpha}
Assume conditions \ref{Condi-Furstenberg-Kesten} and \ref{Condi_ms}. 
For any $x \in \bb S_+^{d-1}$ and $b \in \bb R$, we have 
\begin{align*}
\lim_{n \to \infty} \sqrt{n} \,  \bb Q_{x, b}^{\alpha}  \Big(  \min_{1 \leq i \leq n} S_i \geq 0  \Big) 
= \frac{ 2 V_{\alpha}(x,b)}{ \sigma_{\alpha} \sqrt{2 \pi} }. 
\end{align*}
In addition, there exist $c, c' >0$ such that for any $n \geq 1$, $x \in \bb S_+^{d-1}$ and $b \in \bb R$,
\begin{align}\label{bound-tau-b}
 \sqrt{n} \,  \bb Q_{x, b}^{\alpha}  \Big(  \min_{1 \leq i \leq n} S_i \geq 0  \Big) 
\leq  c V_{\alpha}(x,b) \leq c' (1 + \max \{b, 0\} ). 
\end{align}
\end{lemma}

For any $n, k \geq 1$, define 
\begin{align}\label{def-Wn-tilde}
\widetilde{W}_n =  \sum_{|u| = n}  r_{\alpha}(X_u)  e^{- \alpha S_u}  \mathds 1_{ \{ \min_{v \leq u} S_v \geq 0 \} }
\end{align}
and 
\begin{align}\label{def-Wn-tilde-nk}
\widetilde{\widetilde{W}}_{n, k} =  \sum_{|u| = n}  r_{\alpha}(X_u)  e^{-\alpha S_u}  \mathds 1_{ \{ \min_{v \leq u, |v| \geq k} S_v \geq 0 \} }.  
\end{align}
By Lemma \ref{Lem-infinimum-Su}, 
for any $x \in \bb S_+^{d-1}$ and $b \in \bb R$, 
we have that, as $n \to \infty$,  $\min_{|u| = n} S_u \to  \infty$ $\bb P_{x, b}$-a.s. on the set of survival $\mathscr S$,
therefore, for any $\ee > 0$, there exists $k \in \bb N$ such that 
\begin{align}\label{proba-Wn-12}
\bb P_{x, b} \Big( \widetilde{\widetilde{W}}_{n, k} = W_n  \  \forall n \geq 1 \Big) > 1 - \ee. 
\end{align}

\begin{lemma}\label{Lem-expect-mart-tilde}
Assume conditions \ref{Condi-Furstenberg-Kesten} and \ref{Condi_ms}. 
Then, for any $x \in \bb S_+^{d-1}$ and $b \in \bb R$, 
\begin{align*}
\lim_{n \to \infty} \sqrt{n}  \, \bb E_{x, b} \widetilde{W}_n 
= e^{-\alpha b} r_{s}(x)  \frac{ 2 V_{\alpha}(x,b)}{ \sigma_{\alpha} \sqrt{2 \pi} },
\end{align*}
and there exists $c>0$ such that, for all $n \geq 1$, 
\begin{align*}
\sqrt{n}  \, \bb E_{x, b} \widetilde{W}_n 
\leq  c e^{-\alpha b}  V_{\alpha}(x,b). 
\end{align*}
\end{lemma}

\begin{proof}
By the many-to-one formula (cf.\ \eqref{Formula_many_to_one}), we have 
\begin{align*}
\bb E_{x, b} \widetilde{W}_n
 = \bb E_{x, b}  \Big(  \sum_{|u| = n}  r_{\alpha}(X_u)  e^{-\alpha S_u}  \mathds 1_{ \{ \min_{v \leq u} S_v \geq 0 \} } \Big)
 = e^{-\alpha b} r_{s}(x)   
  \bb Q_{x, b}^{\alpha}  \left(  \min_{1 \leq i \leq n} S_i \geq 0  \right).  
\end{align*}
By Lemma \ref{Lem-exit-time-alpha}, the result follows. 
\end{proof}

Let $\mathscr H_{\infty}$ be the natural filtration of the spine particles and their children, i.e.
\begin{align}\label{def-scr-H-infty}
\mathscr H_{\infty} = \sigma \left\{ w|k, (X_{w|k}, S_{w|k}), (X_v, S_v);  \overset{\leftarrow}{v} = w|k,  k \geq 0 \right\}. 
\end{align}

\begin{lemma}\label{Lem-minimum-condi}
Assume conditions \ref{Condi-Furstenberg-Kesten} and \ref{Condi_ms}. 
Then, for any $x \in \bb S_+^{d-1}$ and $b \geq 0$, we have $\wh{\bb P}_{x, b}$-a.s., 
\begin{align*}
\bb E_{x, b} \left(  \sqrt{n} \widetilde{W}_n  \mathds 1_{ \{ \sqrt{n} \widetilde{W}_n \geq \ee \} }  \right)
 \leq   T_1(x, b, \ee, n) + T_2(x, b, \ee, n),  
\end{align*}
where 
\begin{align*}
T_1(x, b, \ee, n)  & =  
e^{-\alpha b} r_{\alpha}(x)  \sqrt{n}   \,  \wh{\bb P}_{x, b}   
\left(   \sqrt{n}  r_{\alpha} (X_{w|n}) e^{-\alpha S_{w|n}}  \geq \frac{\ee}{2},   \min_{1 \leq k \leq n} S_{w|k} \geq 0 \right),  
\notag\\
T_2(x, b, \ee, n)  & =  c e^{-\alpha b}  V_{\alpha}(x,b)  
\\
&\quad \times \wh{\bb E}_{x, b}   \Bigg.
\Bigg[ 
\Bigg(  \frac{2 \sqrt{n}}{\ee}  \sum_{k=0}^{n-1}  \sum_{\substack{  i \in \bb N \\  (w|k)i \neq w|k+1}} 
 \bb E_{X_{(w|k) i}, S_{(w|k) i} }  \widetilde{W}_{n-k-1}  \Bigg)  \wedge 1  \Bigg| \min_{1 \leq k \leq n} S_{w|k} \geq 0 \Bigg]. 
\end{align*}
\end{lemma}

\begin{proof}
By \eqref{Formula_many_to_one} and \eqref{sized-BRW-spinal}, 
we get 
\begin{align*}
\bb E_{x, b} \left(  \sqrt{n} \widetilde{W}_n  \mathds 1_{ \{ \sqrt{n} \widetilde{W}_n \geq \ee \} }  \right)
 = e^{-\alpha b} r_{\alpha}(x)  \sqrt{n} 
\,   \wh{\bb P}_{x, b}  
\Big(  \sqrt{n} \widetilde{W}_n \geq \ee,  \min_{1 \leq k \leq n} S_{w|k} \geq 0  \Big). 
\end{align*}
We use the branching property and the Markov property to decompose $\widetilde{W}_{n}$ along the spine and its children
to get the following identity:
\begin{align*}
\widetilde{W}_{n} = r_{\alpha} (X_{w|n}) e^{-\alpha S_{w|n}} \mathds 1_{\{ \min_{1 \leq k \leq n} S_{w|k} \geq 0 \} } 
+ \sum_{k=0}^{n-1}  \sum_{\substack{  i \in \bb N \\  (w|k)i \neq w|k+1}} 
[\widetilde{W}_{n-k-1}]_{(w|k) i} \mathds 1_{\{ S_v \geq 0, \forall v \leq (w|k)i \} },
\end{align*}
where $[\widetilde{W}_{n-k-1}]_{(w|k) i}$ has the same law as $\widetilde{W}_{n-k-1}$ under 
$\bb P_{X_{(w|k) i}, S_{(w|k) i} }$ conditioned on $\mathscr H_{\infty}$. 
Using the inequality $\mathds 1_{\{ Z_1 + Z_2 \geq \ee  \} } \leq \mathds 1_{\{ Z_1  \geq \frac{\ee}{2}  \} } + \mathds 1_{\{ Z_2  \geq \frac{\ee}{2}  \} }$,
we get
\begin{align*}
\wh{\bb P}_{x, b}   
\Big(    \sqrt{n} \widetilde{W}_n \geq \ee,   \min_{1 \leq k \leq n} S_{w|k} \geq 0  \Big)
\leq   \wh{\bb P}_{x, b}   
\Big(   \sqrt{n}  r_{\alpha} (X_{w|n}) e^{-\alpha S_{w|n}}  \geq \frac{\ee}{2},   \min_{1 \leq k \leq n} S_{w|k} \geq 0  \Big) + I_n, 
\end{align*}
where
\begin{align*}
I_n = \wh{\bb P}_{x, b}   
\Bigg(  \sqrt{n}  \sum_{k=0}^{n-1}  \sum_{\substack{  i \in \bb N \\  (w|k)i \neq w|k+1}} 
[\widetilde{W}_{n-k-1}]_{(w|k) i}   \geq \frac{\ee}{2},   \min_{1 \leq k \leq n} S_{w|k} \geq 0 \Bigg).  
\end{align*}
Denote
\begin{align*}
J_n = \wh{\bb P}_{x, b}   
	\Bigg(  \sqrt{n} \Bigg.    \sum_{k=0}^{n-1}  \sum_{\substack{  i \in \bb N \\  (w|k)i \neq w|k+1}} 
[\widetilde{W}_{n-k-1}]_{(w|k) i}   \geq \frac{\ee}{2}   \Bigg| \min_{1 \leq k \leq n} S_{w|k} \geq 0 \Bigg). 
\end{align*}
Taking conditional probability, by \eqref{sized-BRW-spinal} and Lemma \ref{Lem-exit-time-alpha}, it follows that
\begin{align}\label{Inequa-Wn-ee-001}
I_n  =  \wh{\bb P}_{x, b} \left( \min_{1 \leq k \leq n} S_{w|k} \geq 0 \right)
	\,   J_n 
 =  \bb Q_{x, b}^{\alpha} \left( \min_{1 \leq k \leq n} S_{k} \geq 0 \right)
	\,   J_n  
 \leq  \frac{c V_{\alpha}(x,b) }{\sqrt{n}}    
	\,  J_n. 
\end{align}
For $J_n$, taking conditional expectation w.r.t.  $\mathscr H_{\infty}$ and using Markov's inequality, we obtain 
\begin{align*}
J_n  & = \wh{\bb E}_{x, b}   
	\Bigg[  \wh{\bb E}_{x, b}  \Bigg(  \Bigg.  \Bigg.   \mathds 1_{ \{  \sqrt{n} \sum_{k=0}^{n-1}  \sum_{\substack{  i \in \bb N \\  (w|k)i \neq w|k+1}} 
[\widetilde{W}_{n-k-1}]_{(w|k) i}   \geq \frac{\ee}{2}  \} }   \Bigg|  \mathscr H_{\infty}   \Bigg)  \Bigg| \min_{1 \leq k \leq n} S_{w|k} \geq 0 \Bigg]  \notag\\
& \leq  \wh{\bb E}_{x, b}   \Bigg.
\Bigg[ 
\Bigg(  \frac{2 \sqrt{n}}{\ee}  \sum_{k=0}^{n-1}  \sum_{\substack{  i \in \bb N \\  (w|k)i \neq w|k+1}} 
\bb E_{X_{(w|k) i}, S_{(w|k) i} }  \widetilde{W}_{n-k-1}  \Bigg)  \wedge 1
\Bigg| \min_{1 \leq k \leq n} S_{w|k} \geq 0 \Bigg],  
\end{align*}
completing the proof of the lemma. 
\end{proof}

We start by handling the first term $T_1(x, b, \ee, n)$ in Lemma \ref{Lem-minimum-condi}. 

\begin{lemma}\label{Lem-T1-bound}
Assume conditions \ref{Condi-Furstenberg-Kesten} and \ref{Condi_ms}. 
For any fixed $x \in \bb S_+^{d-1}$, $b \geq 0$ and $\ee > 0$, we have $\lim_{n \to \infty} T_1(x, b, \ee, n) = 0$. 
\end{lemma}

\begin{proof}
This is a consequence of the conditioned local limit theorem for products of positive random matrices (cf.\ Lemma \ref{Lem-CLLT-bound}).
By \eqref{sized-BRW-spinal}, we have 
\begin{align*}
T_1(x, b, \ee, n) =  e^{-\alpha b} r_{\alpha}(x)  \sqrt{n}   \,  \bb Q_{x,b}^{\alpha} 
\Big(   \sqrt{n}  r_{\alpha} (X_{n}) e^{-\alpha S_{n}}  \geq \frac{\ee}{2},   \min_{1 \leq k \leq n} S_{k} \geq 0  \Big). 
\end{align*}
It follows that 
\begin{align*}
T_1(x, b, \ee, n) 
& \leq  c e^{-\alpha b}  \sqrt{n}   \,  \bb Q_{x,b}^{\alpha} 
\Big(  S_n \leq  n^{\delta},  \min_{1 \leq k \leq n} S_{k} \geq 0  \Big)  \notag\\
& \quad +  c e^{-\alpha b}  \sqrt{n}   \,  \bb Q_{x,b}^{\alpha}  
\Big(  S_n > n^{\delta},  \sqrt{n}  r_{\alpha} (X_{n}) e^{-\alpha S_{n}}  \geq \frac{\ee}{2}  \Big) \notag\\
& =: T_{11}(x, b, \ee, n) + T_{12}(x, b, \ee, n). 
\end{align*}

For the first term $T_{11}(x, b, \ee, n)$, by Lemma \ref{Lem-CLLT-bound}, 
there exists a constant $c>0$ such that for any $x \in \bb S_+^{d-1}$, $b \geq 0$, $\ee >0$ and $n \geq 1$, 
\begin{align*}
\bb Q_{x,b}^{\alpha} 
\Big(  S_n \leq  n^{\delta},  \min_{1 \leq k \leq n} S_{k} \geq 0  \Big)
\leq  \frac{c (1+b)}{n^{3/2 - 2 \delta}}, 
\end{align*}
so that, for any $x \in \bb S_+^{d-1}$, $b \geq 0$, $\ee > 0$ and $\delta \in (0, \frac{1}{2})$, we have $\lim_{n \to \infty} T_{11}(x, b, \ee, n) = 0$. 

For the second term $T_{12}(x, b, \ee, n)$,
using Markov's inequality and the fact that $r_{\alpha}$ is bounded, 
there exists $c>0$ such that for any $x \in \bb S_+^{d-1}$, $b \geq 0$, $\ee >0$ and $n \geq 1$, 
\begin{align*}
\bb Q_{x,b}^{\alpha}  
\Big(  S_n > n^{\delta},  \sqrt{n}  r_{\alpha} (X_{n}) e^{-\alpha S_{n}}  \geq \frac{\ee}{2}  \Big)
\leq  \frac{c}{\ee} \sqrt{n} e^{- \alpha n^{\delta}}, 
\end{align*}
so that $\lim_{n \to \infty} T_{12}(x, b, \ee, n) = 0$,
completing the proof of the lemma. 
\end{proof}

To deal with the second term $T_2(x, b, \ee, n)$ in Lemma \ref{Lem-minimum-condi}, 
we need to introduce a new probability measure as follows:
for any $x \in \bb S_+^{d-1}$, $b \in \bb R$, $n \geq 0$ and $A \in \mathscr F_n$, 
\begin{align}\label{change-measure-PW}
	\bb P^+_{x, b} (A) 
	=  \frac{1}{ V_{\alpha}(x,b) } 
	\wh{\bb E}_{x,b}  \left( V_{\alpha}(X_{w|n}, S_{w|n}) \mathds 1_{\{ \min_{1 \leq k \leq n} S_{w|k} \geq 0 \}}  \mathds 1_A \right). 
\end{align}
Denote by $\bb E^+_{x, b}$ the corresponding expectation. 
In other words, under the measure $\bb P^+_{x, b}$, 
the spine is conditioned to stay nonnegative at all times. 

\begin{lemma}\label{Lem-conver-Yn-Yinfty}
	Let $(Y_n)_{n \geq 0}$ be a uniformly bounded sequence of random variables,
	adapted to $(\mathscr F_n)_{n \geq 0}$. 
	Let $x \in \bb S_+^{d-1}$ and $b \in \bb R$. 
	If $Y_{\infty}$ is a random variable such that $Y_n \to Y_{\infty}$ in probability under $\bb P^+_{x, b}$,
	then
	\begin{align*}
		\lim_{n \to \infty}  \wh{\bb E}_{x, b}   \left(  Y_n  \, \Big| \,  \min_{1 \leq k \leq n} S_{w|k} \geq 0  \right)
		= \bb E^+_{x, b} (Y_{\infty}). 
	\end{align*}
\end{lemma}

\begin{proof}
We fix $x \in \bb S_+^{d-1}$ and $b \in \bb R$. Denote $m_n(x, b) = \wh{\bb P}_{x, b} (  \min_{1 \leq k \leq n} S_{w|k} \geq 0 )$. 
For any $1 \leq l \leq n$, taking conditional expectation with respect to $\mathscr F_l$, we have
\begin{align*}
\wh{\bb E}_{x, b}   \left(  Y_l  \, \Big| \,  \min_{1 \leq k \leq n} S_{w|k} \geq 0  \right)
& = \frac{1}{m_n(x, b)}  \wh{\bb E}_{x, b}   \left(  Y_l  \mathds 1_{\{ \min_{1 \leq k \leq n} S_{w|k} \geq 0 \}}  \right) \notag\\
& =   \wh{\bb E}_{x, b}   
\bigg(  Y_l    \mathds 1_{\{ \min_{1 \leq k \leq l} S_{w|k} \geq 0 \}}   \frac{m_{n-l}(X_{w|l}, S_{w|l})}{m_n(x, b)}  \bigg). 
\end{align*}
For any fixed $l \in [1, n]$, we use Lemma \ref{Lem-exit-time-alpha} to get that 
$\frac{m_{n-l}(X_{w|l}, S_{w|l})}{m_n(x, b)}$ converges to $\frac{V_{\alpha}(X_{w|l}, S_{w|l})}{V_{\alpha}(x, b)}$
as $n \to \infty$,  and is bounded by a constant multiple of this quantity. 
By the Lebesgue dominated converge theorem and the definition of $\bb E^+_{x, b}$, we obtain 
\begin{align}\label{pf-lim-Yl-01}
\lim_{n \to \infty} \wh{\bb E}_{x, b}   \left(  Y_l  \, \Big| \,  \min_{1 \leq k \leq n} S_{w|k} \geq 0  \right)
= \wh{\bb E}_{x, b}   
\bigg(  Y_l    \mathds 1_{\{ \min_{1 \leq k \leq l} S_{w|k} \geq 0 \}}   \frac{V_{\alpha}(X_{w|l}, S_{w|l})}{V_{\alpha}(x, b)}  \bigg)
= \bb E^+_{x, b} (Y_l). 
\end{align}
	
Now let $a>1$ and we use again Lemma \ref{Lem-exit-time-alpha} to get that there exists a constant $c>0$ such that for any $1 \leq l \leq n$, 
\begin{align}\label{pf-lim-Yl-02}
&\left| \wh{\bb E}_{x, b}   \left( Y_n - Y_l  \, \Big| \,  \min_{1 \leq k \leq [an]} S_{w|k} \geq 0  \right)  \right|  \notag\\
& \leq  \wh{\bb E}_{x, b}   
\left( | Y_n - Y_l |   \mathds 1_{\{ \min_{1 \leq k \leq n} S_{w|k} \geq 0 \}}   \frac{m_{ [(a-1)n] }(X_{w|n}, S_{w|n})}{m_{ [an] }(x, b)}  \right)  \notag\\
& \leq c \sqrt{\frac{a}{a-1}}  \wh{\bb E}_{x, b}   
\left( | Y_n - Y_l |   \mathds 1_{\{ \min_{1 \leq k \leq n} S_{w|k} \geq 0 \}}   \frac{V(X_{w|n}, S_{w|n})}{V(x, b)}  \right)  \notag\\
& =  c \sqrt{\frac{a}{a-1}}  \bb E^+_{x, b} \left( | Y_n - Y_l |  \right). 
\end{align}
Since $(Y_n)_{n \geq 0}$ is uniformly bounded and $Y_n \to Y_{\infty}$ in probability under $\bb P^+_{x, b}$, we have $\lim_{l \to \infty} \bb E^+_{x, b} (Y_l) = \bb E^+_{x, b} (Y_{\infty})$ and $\lim_{l \to \infty} \lim_{n \to \infty} \bb E^+_{x, b} (| Y_n - Y_l |) = 0$. This, together with \eqref{pf-lim-Yl-01} and \eqref{pf-lim-Yl-02}, implies that 
\begin{align}\label{pf-lim-Yl-03}
\lim_{n \to \infty} \wh{\bb E}_{x, b}   \left(  Y_n  \, \Big| \,  \min_{1 \leq k \leq [an]} S_{w|k} \geq 0  \right)
= \bb E^+_{x, b} (Y_{\infty}). 
\end{align}	
Since $(Y_n)_{n \geq 0}$ is uniformly bounded by some constant, say $c_Y >0$, we have 
\begin{align*}
& \left| \wh{\bb E}_{x, b}   \left(  Y_n  \mathds 1_{ \{ \min_{1 \leq k \leq n} S_{w|k} \geq 0 \} }  \right) - \bb E^+_{x, b} (Y_{\infty}) m_n(x, b) \right|  \notag\\
 & \leq \left| \wh{\bb E}_{x, b}   \left(  Y_n  \mathds 1_{ \{ \min_{1 \leq k \leq [an] } S_{w|k} \geq 0 \} }  \right) - \bb E^+_{x, b} (Y_{\infty}) m_{[an]}(x, b) \right| 
         + 2 c_Y \left( m_{n}(x, b) - m_{[an]}(x, b) \right),
 \end{align*}
 so that 
 \begin{align*}
 & \left|  \wh{\bb E}_{x, b}   \left(  Y_n  \, \Big| \,  \min_{1 \leq k \leq n} S_{w|k} \geq 0  \right) - \bb E^+_{x, b} (Y_{\infty}) \right|  \notag\\
 & \leq \frac{ m_{[an]}(x, b) }{ m_{n}(x, b) }  
  \left| \wh{\bb E}_{x, b}   \left(  Y_n  \mathds 1_{ \{ \min_{1 \leq k \leq [an] } S_{w|k} \geq 0 \} }  \right)  -  \bb E^+_{x, b} (Y_{\infty}) \right| + 2 c_Y \left( 1 - \frac{ m_{[an]}(x, b) }{ m_{n}(x, b) } \right). 
  \end{align*}
By Lemma \ref{Lem-exit-time-alpha}, we have $\lim_{n \to \infty} \frac{ m_{[an]}(x, b) }{ m_{n}(x, b) } = \frac{1}{\sqrt{a}}$. Therefore, using \eqref{pf-lim-Yl-03}, we obtain 
 \begin{align*}
  \limsup_{n \to \infty} \left|  \wh{\bb E}_{x, b}   \left(  Y_n  \, \Big| \,  \min_{1 \leq k \leq n} S_{w|k} \geq 0  \right) - \bb E^+_{x, b} (Y_{\infty}) \right|
  \leq 2 c_Y \left( 1 - \frac{1}{\sqrt{a}} \right).  
 \end{align*}
  Taking the limit $a \to 1$ concludes the proof of the lemma. 
\end{proof}

For $k \geq 0$, we define 
\begin{align}\label{def-Lk-001}
L_k = \sum_{\substack{  i \in \bb N \\  (w|k)i \neq w|k+1}}  e^{-\alpha (S_{(w|k) i} - S_{w|k} ) }  
  ( 1 + \max\{ S_{(w|k) i} -  S_{w|k},  0 \} ). 
\end{align}

\begin{lemma}\label{Lem-L0-001}
There exists $c>0$ such that for any $x \in \bb S_+^{d-1}$, $b \in \bb R$ and $\ee >0$, 
\begin{align*}
\bb E^+_{x, b}  \left(  \frac{ e^{- \frac{\alpha}{4} S_{\o}}  L_0 }{\ee}    \wedge 1  \right)
\leq  c \bb E_{x}  \left(  \left( W_1 + \frac{Z_1}{V_{\alpha}(x,b)} \right)  \left(  \frac{ e^{- \frac{\alpha}{4} b}  
 (W_1 + Z_1) }{\ee}    \wedge 1  \right) \right). 
\end{align*}
\end{lemma}

\begin{proof}
By \eqref{def-Lk-001}, we have
\begin{align*}
J(x, b):&= \bb E^+_{x, b}  \left(  \frac{ e^{- \frac{\alpha}{4} S_{\o}}  L_0 }{\ee}    \wedge 1  \right) \\
 & \leq \bb E^+_{x, b}  \left(   \frac{ e^{- \frac{\alpha}{4} S_{\o}}  \sum_{|u|=1}  e^{-\alpha (S_u - S_{\o} ) }  
  ( 1 + \max\{ S_u -  S_{\o},  0 \} ) }{\ee}    \wedge 1  \right) \\
  & \leq \bb E^+_{x, b}  \left(  \frac{ e^{- \frac{\alpha}{4} S_{\o}}  
 (W_1 + Z_1) }{\ee}    \wedge 1   \right).
 \end{align*}
By a change of measure \eqref{change-measure-PW}, we get
 \begin{align*}
J(x, b) & \leq \frac{1}{ V_{\alpha}(x,b) }  \wh{\bb E}_{x,b}  \left(
 V_{\alpha}(X_{w|1}, S_{w|1}) \mathds 1_{\{ \min_{1 \leq j \leq k} S_{w|j} \geq 0 \}}
  \left(  \frac{ e^{- \frac{\alpha}{4} S_{\o}}  
 (W_1 + Z_1) }{\ee}    \wedge 1   \right) \right) \\
 & \leq  \frac{1}{ V_{\alpha}(x,b) }  \wh{\bb E}_{x,b}  \left( \sum_{|z| = 1} \mathds 1_{\{w|1 = z\}} 
 V_{\alpha}(X_z, S_z) 
  \left(  \frac{ e^{- \frac{\alpha}{4} S_{\o}}  
 (W_1 + Z_1) }{\ee}    \wedge 1   \right) \right).
  \end{align*}
Taking conditional expectation with respect to $\scr F_1$, using Corollary \ref{Cor-spine-001} and \eqref{def-hat-P-x-b-intro}, we get
 \begin{align*}
J(x, b) & \leq \wh{\bb E}_{x,b}  \left( \sum_{|z| = 1} \frac{ r_{\alpha} (X_z)  e^{- \alpha S_z} }{ W_1 } 
\frac{V_{\alpha}(X_{z}, S_{z})}{V_{\alpha}(x,b)}    \left(  \frac{ e^{- \frac{\alpha}{4} S_{\o}}  
 (W_1 + Z_1) }{\ee}    \wedge 1   \right) \right) \\
& = \bb E_{x,b}  \left( \frac{e^{\alpha b}}{r_{\alpha}(x)} \sum_{|z| = 1}  r_{\alpha} (X_z)  e^{- \alpha S_z} 
\frac{V_{\alpha}(X_{z}, S_{z})}{V_{\alpha}(x,b)}   \left(  \frac{ e^{- \frac{\alpha}{4} b}  
 (W_1 + Z_1) }{\ee}    \wedge 1  \right) \right).
   \end{align*}
Taking into account that 
 \begin{align*}
\frac{e^{\alpha b}}{r_{\alpha}(x)} \sum_{|z| = 1}  r_{\alpha} (X_z)  e^{- \alpha S_z} 
\frac{V_{\alpha}(X_{z}, S_{z})}{V_{\alpha}(x,b)} 
\leq  c \left( W_1 + \frac{Z_1}{V_{\alpha}(x,b)} \right), 
\end{align*}
we obtain
    \begin{align*}
J(x, b) & \leq  c \bb E_{x,b}  \left(  \left( W_1 + \frac{Z_1}{V_{\alpha}(x,b)} \right)  \left(  \frac{ e^{- \frac{\alpha}{4} b}  
 (W_1 + Z_1) }{\ee}    \wedge 1  \right) \right), 
\end{align*}
completing the proof. 
\end{proof}

Lemmas \ref{Lem-conver-Yn-Yinfty} and \ref{Lem-L0-001} are useful to deal with the second term $T_2(x, b, \ee, n)$ in Lemma \ref{Lem-minimum-condi}.  

\begin{lemma}\label{Lem-bound-T2}
Assume conditions \ref{Condi-Furstenberg-Kesten}, \ref{Condi_ms} and \ref{condi:momentsW1}. 
For every $\ee >0$, there exists a positive function $\tilde h$ satisfying $\lim_{b \to \infty} \tilde h(b) = 0$
such that the following holds: 
for every $x \in \bb S_+^{d-1}$ and $b \geq 0$, we have $\limsup_{n \to \infty} T_2(x, b, \ee, n) \leq \tilde h(b)$. 
\end{lemma}

\begin{proof}
In view of \eqref{def-Wn-tilde}, using the many-to-one formula \eqref{Formula_many_to_one}
and the fact that $\mathfrak m(\alpha) =1$, we get that, for any $0 \leq k \leq n-1$ and $i \in \bb N$ with $(w|k)i \neq w|k+1$, 
\begin{align*}
\bb E_{X_{(w|k) i}, S_{(w|k) i} }  \widetilde{W}_{n-k-1} 
& = \bb E_{X_{(w|k) i}, S_{(w|k) i} } 
\bigg[  \sum_{|u| = n-k-1}  r_{\alpha}(X_u)  e^{- \alpha S_u}  \mathds 1_{ \{ \min_{v \leq u} S_v \geq 0 \} }  \bigg]  \notag\\
& = r_{\alpha}(X_{(w|k) i})   e^{-\alpha S_{(w|k) i}}
\bb Q_{ X_{(w|k) i}, S_{(w|k) i} }^{\alpha} 
 \left(   \min_{1 \leq i \leq n-k-1} S_i \geq 0   \right). 
\end{align*}
Since $r_{\alpha}(X_{(w|k) i}) \leq c'$, by Lemma \ref{Lem-exit-time-alpha}, it follows that, on the set $\{ S_{w|k} \geq 0 \}$, 
\begin{align*}
& \bb E_{X_{(w|k) i}, S_{(w|k) i} }  \widetilde{W}_{n-k-1}   \notag\\
& \leq  \frac{c}{\sqrt{n-k}}  e^{-\alpha S_{(w|k) i}}  ( 1 + \max\{ S_{(w|k) i}, 0 \} )  \notag\\ 
& \leq  \frac{c}{\sqrt{n-k}}  e^{-\alpha S_{w|k}}   ( 1 + \max\{ S_{w|k}, 0 \} )   e^{-\alpha (S_{(w|k) i} - S_{w|k} ) }  
  ( 1 + \max\{ S_{(w|k) i} -  S_{w|k},  0 \} ) \notag\\
& \leq  \frac{c'}{\sqrt{n-k}}  e^{- \frac{\alpha}{2} S_{w|k}}    e^{-\alpha (S_{(w|k) i} - S_{w|k} ) }  
  ( 1 + \max\{ S_{(w|k) i} -  S_{w|k},  0 \} ), 
\end{align*}
where in the last inequality we used $(1 + \max \{ x, 0 \}) e^{-\alpha x} \leq c e^{- \frac{\alpha}{2} x}$ for $x \geq 0$. 
Hence, 
\begin{align*}
\sum_{\substack{  i \in \bb N \\  (w|k)i \neq w|k+1}} 
	\bb E_{X_{(w|k) i}, S_{(w|k) i} }  \widetilde{W}_{n-k-1}  
\leq  \frac{c'}{\sqrt{n-k}}  e^{- \frac{\alpha}{2} S_{w|k}}  L_k, 
\end{align*}
where $L_k$ is defined by \eqref{def-Lk-001}. 
Therefore, 
\begin{align}\label{bound-T2-xbn}
T_2(x, b, \ee, n)  
\leq c  \wh{\bb E}_{x, b}  
	\left( 
	Y_n  \Big| \min_{1 \leq k \leq n} S_{w|k} \geq 0 \right), 
\end{align}
where 
\begin{align*}
Y_n = \left(  \sum_{k=0}^{n-1}  \left( \sqrt{\frac{n}{n-k}}  \frac{ e^{- \frac{\alpha}{2} S_{w|k}}  L_k }{\ee}    \wedge 1  \right)  \right)  \wedge 1. 
\end{align*}
It holds that 
\begin{align}\label{inequa-Yn-12}
Y_n \leq  \sqrt{2} Y_n' + Y_n''
\end{align}
where 
\begin{align*}
Y_n' =  \left(  \sum_{k=0}^{[n/2]}  \left(  \frac{ e^{- \frac{\alpha}{4} S_{w|k}}  L_k }{\ee}    \wedge 1  \right)  \right)  \wedge 1,
\quad 
Y_n'' = \left(  \sum_{k=[n/2] + 1}^{n-1}  \left( \sqrt{n}  \frac{ e^{- \frac{\alpha}{2} S_{w|k}}  L_k }{\ee}    \wedge 1  \right)  \right)  \wedge 1. 
\end{align*}
Note that both $Y_n'$ and $Y_n''$ are measurable with respect to the natural filtration $\mathscr F_n$. 
Since the sequence $(Y_n')_{n \geq 1}$ is bounded by $1$, by monotonicity, we get that $\bb P^+_{x, b}$-almost surely, 
\begin{align*}
\lim_{n \to \infty} Y_n' = Y_{\infty}' : = \left(  \sum_{k=0}^{\infty}  \left(  \frac{ e^{- \frac{\alpha}{4} S_{w|k}}  L_k }{\ee}    \wedge 1  \right)  \right)  \wedge 1. 
\end{align*}
We claim that 
\begin{enumerate}
\item
$\lim_{b \to \infty} \bb E^+_{x, b} (Y_{\infty}') = 0$. 

\item
$\lim_{n \to \infty} Y_n'' = 0$ in probability $\bb P^+_{x, b}$. 
\end{enumerate}
Now we show how to obtain Lemma \ref{Lem-bound-T2} by using these two claims.  
Applying Lemma \ref{Lem-conver-Yn-Yinfty} to $(Y_n')_{n \geq 0}$ and $(Y_n'')_{n \geq 0}$ gives 
\begin{align*}
\lim_{n \to \infty}  \wh{\bb E}_{x, b}   \left(  Y_n'  \, \Big| \,  \min_{1 \leq k \leq n} S_{w|k} \geq 0  \right)
= \bb E^+_{x, b} (Y_{\infty}') 
\end{align*}
and 
\begin{align*}
\lim_{n \to \infty}  \wh{\bb E}_{x, b}   \left(  Y_n''  \, \Big| \,  \min_{1 \leq k \leq n} S_{w|k} \geq 0  \right)
= 0. 
\end{align*}
Combining these with \eqref{bound-T2-xbn} and \eqref{inequa-Yn-12} gives 
\begin{align*}
\limsup_{n \to \infty} T_2(x, b, \ee, n) \leq c  \bb E^+_{x, b} (Y_{\infty}'), 
\end{align*}
which shows Lemma \ref{Lem-bound-T2} by using claim (1). 

Now we prove claim (1). 
Note that 
\begin{align}\label{indenti-Y-infty}
\bb E^+_{x, b} (Y_{\infty}')
\leq  \sum_{k=0}^{\infty}  \bb E^+_{x, b}  \left(  \frac{ e^{- \frac{\alpha}{4} S_{w|k}}  L_k }{\ee}    \wedge 1  \right) 
=   \sum_{k=0}^{\infty}  \bb E^+_{x, b}  \bb E^+_{X_{w|k}, S_{w|k}}  \left(  \frac{ e^{- \frac{\alpha}{4} S_{\o}}  L_0 }{\ee}    \wedge 1  \right).  
\end{align}
By Lemma \ref{Lem-L0-001} and \eqref{inequa-V-harmonic}, we get
\begin{align}\label{indenti-Y-infty-b}
\bb E^+_{X_{w|k}, S_{w|k}}  \left(  \frac{ e^{- \frac{\alpha}{4} S_{\o}}  L_0 }{\ee}    \wedge 1  \right)
\leq  c f(X_{w|k}, S_{w|k}),
\end{align}
where $f(x, y) = f_1(x, y) + f_2(x, y)$ for $x \in \bb S_+^{d-1}$ and $y \geq 0$, and 
\begin{align*}
& f_1(x, y) =  \bb E_{x}  \left( W_1  \left(  \frac{ e^{- \frac{\alpha}{4} y}  
 (W_1 + Z_1) }{\ee}    \wedge 1  \right) \right),  \notag\\
& f_2(x, y) = \frac{1}{1 + y} \bb E_{x}  \left( Z_1  \left(  \frac{ e^{- \frac{\alpha}{4} y}  
 (W_1 + Z_1) }{\ee}    \wedge 1  \right) \right). 
\end{align*}
As in the proof of Lemma B.1 (i) of \cite{Aid13}, one can check that 
condition \ref{condi:momentsW1} implies that 
\begin{align*}
\bb E_x \left( W_1 \Big( \log^+ (W_1 + Z_1) \Big)^2 \right) < \infty,  
\quad 
\bb E_x \left( Z_1 \log^+ (W_1 + Z_1) \right) < \infty. 
\end{align*}
Then, applying Lemma C.1 of \cite{BM19} twice to $W_1 + Z_1$, once under the law $\bb E_x ( W_1 \cdot )$ and with $\rho(y) = y$,
and once under the law $\frac{1}{\bb E_x (Z_1)} \bb E_x (Z_1 \cdot )$ and with $\rho =1$, we get
\begin{align*}
\int_{\bb R_+} f(x, y) y dy 
= \int_{\bb R_+} f_1(x, y) y dy +  \int_{\bb R_+} f_2(x, y) y dy < \infty. 
\end{align*}
By \eqref{indenti-Y-infty} and \eqref{indenti-Y-infty-b}, we have 
\begin{align*}
\bb E^+_{x, b} (Y_{\infty}')
\leq  c \sum_{k=0}^{\infty}  \bb E^+_{x, b}  f(X_{w|k}, S_{w|k}).  
\end{align*}
Using \eqref{change-measure-PW} and \eqref{sized-BRW-spinal}, we get that, for any $x \in \bb S_+^{d-1}$, $b \geq 0$ and $k \geq 1$, 
\begin{align*}
\bb E^+_{x, b}  f(X_{w|k}, S_{w|k}) 
& =  \frac{1}{ V_{\alpha}(x,b) }  \bb E_{\bb Q_{x}^\alpha}  \left[ V_{\alpha}(X_k, b + S_k) f(b + S_k) \mathds{1}_{\{b+\min_{1 \leq j \leq k} S_j \geq 0\}}  \right]  \notag\\
& \leq  \frac{c}{1 + b} \bb E_{\bb Q_{x}^\alpha}  \left[ (1 + b + S_k) f(b + S_k) \mathds{1}_{\{b+\min_{1 \leq j \leq k} S_j \geq 0\}}  \right]. 
\end{align*}
Applying Theorem \ref{Thm-Green-function} yields 
\begin{align*}
\lim_{b \to \infty} \frac{1}{1+b} \sum_{k=0}^\infty \bb E_{\bb Q_{x}^\alpha} 
 \Big[ (1 + b+S_k) f(b+S_k) \mathds{1}_{\{b+\min_{1 \leq j \leq k} S_j \geq 0\}}  \Big] =0.
\end{align*}
This proves claim (1). 

Next we prove claim (2). 
We start by showing that, for any $\ee \in (0, 1/6)$, 
we have $S_{w|n} \geq  n^{\ee}$ for  all sufficiently large $n$,  $\bb P^+_{x, b}$-a.s.
Using \eqref{change-measure-PW} and \eqref{sized-BRW-spinal}, we get that 
there exists a constant $c>0$ such that for any $x \in \bb S_+^{d-1}$, $b \geq 0$ and $n \geq 1$, 
\begin{align*}
\bb P^+_{x, b}  \left( S_{w|n} <  n^{\ee}  \right)
& =  \frac{1}{ V_{\alpha}(x,b) }  \bb E_{\bb Q_{x}^\alpha}  
\left[ V_{\alpha}(X_n, b + S_n) \mathds{1}_{\{b + S_n < n^{\ee} \}} \mathds{1}_{\{b + \min_{1 \leq j \leq n} S_j \geq 0\}}  \right]  \notag\\
& \leq  c n^{\ee}  \bb Q_{x}^\alpha  \left( b + S_n \in [0, n^{\ee}),   b + \min_{1 \leq j \leq k} S_j \geq 0  \right). 
\end{align*}
By Lemma \ref{Lem-CLLT-bound}, there exists $c>0$ such that for any $x \in \bb S_+^{d-1}$, $b \geq 0$ and $n \geq 1$, 
\begin{align*}
\bb Q_{x}^\alpha  \left( b + S_n \in [0, n^{\ee}),   b + \min_{1 \leq j \leq k} S_j \geq 0  \right)
\leq  c \frac{1 +b}{ n^{3/2 - 2\ee} },
\end{align*}
so that 
\begin{align*}
\bb P^+_{x, b}  \left( S_{w|n} <  n^{\ee}  \right)
\leq c \frac{1 +b}{ n^{3/2 - 3\ee} }. 
\end{align*}
Since $\ee \in (0, 1/6)$, 
by Borel-Cantelli's  lemma, $S_{w|n} \geq  n^{\ee}$ for all sufficiently large $n$, $\wh{\bb P}^{(b)}_{x}$-a.s.
This implies that, for any $\ee \in (0, 1/6)$, 
we have $S_{w|k} \geq  2^{- \ee} n^{\ee}$ for all $k \in [[n/2] + 1, n]$ and all sufficiently large $n$,  $\bb P^+_{x, b}$-a.s.
Therefore, for all sufficiently large $n$, $\bb P^+_{x, b}$-a.s.
\begin{align*}
Y_n'' \leq  \sum_{k=[n/2] + 1}^{n-1}  \left( \frac{ e^{- \frac{\alpha}{4} S_{w|k}}  L_k }{\ee}    \wedge 1  \right), 
\end{align*}
which converges to $0$ in $\bb P^+_{x, b}$-probability as $n \to \infty$ since in the proof of claim (1) we have shown that 
for any $x \in \bb S_+^{d-1}$ and $b \geq 0$, 
\begin{align*}
  \bb E^+_{x, b}  \sum_{k=0}^{\infty} \left(  \frac{ e^{- \frac{\alpha}{4} S_{w|k}}  L_k }{\ee}    \wedge 1  \right) < \infty
\end{align*}
Hence claim (2) holds and the proof of the lemma is complete. 
\end{proof}

\begin{proposition}\label{Prop-Wn-ee}
Assume conditions \ref{Condi-Furstenberg-Kesten}, \ref{Condi_ms} and \ref{condi:momentsW1}. 
Then, for every $\ee > 0$, there exists a positive measurable function $U$ on $\bb S_+^{d-1} \times \bb R_+$ such that for any $x \in \bb S_+^{d-1}$, 
$\frac{U(x, b)}{V_{\alpha}(x, b)} \to 0$ as $b \to \infty$
and such that the following holds:
for every $b \geq 0$, we have 
\begin{align*}
\limsup_{n \to \infty}  \bb E_{x, b} \left(  \sqrt{n} \widetilde{W}_n  \mathds 1_{ \{ \sqrt{n} \widetilde{W}_n \geq \ee \} }  \right)
\leq  U(x, b) e^{-\alpha b}. 
\end{align*}
\end{proposition}

\begin{proof}
This follows from Lemmas \ref{Lem-minimum-condi}, \ref{Lem-T1-bound} and \ref{Lem-bound-T2}. 
\end{proof}

\begin{proof}[Proof of Theorem \ref{Thm-Seneta-Heyde}]
We first give a lower bound of the conditional Laplace transform of $\widetilde{\widetilde{W}}_{n, k}$ for fixed $k \in \bb N$. 
By \eqref{def-Wn-tilde}, \eqref{def-Wn-tilde-nk} and Jensen's inequality, we have that, for any $\lambda > 0$ and $x \in \bb S_+^{d-1}$, 
\begin{align}\label{inequ-Wnk-lower-1}
\bb E_{x} \Big(   e^{-\lambda \sqrt{n} \widetilde{\widetilde{W}}_{n, k} }  \Big| \scr F_{k}  \Big)
 = \prod_{|u| = k} \bb E_{X_u^x, S_u^x}  \Big(  e^{-\lambda \sqrt{n} \widetilde{W}_{n - k} }   \Big)  
  \geq  \exp  \bigg(  -\lambda \sqrt{n}  \sum_{|u| = k}  \bb E_{X_u^x, S_u^x}   \widetilde{W}_{n - k}   \bigg). 
\end{align}
By Lemma \ref{Lem-expect-mart-tilde}, we have, for any $x \in \bb S_+^{d-1}$ and $b \in \bb R$, 
\begin{align*}
\lim_{n \to \infty} \sqrt{n}  \, \bb E_{x, b}  \widetilde{W}_{n - k}  
= e^{-\alpha b} r_{s}(x)  \frac{ 2 V_{\alpha}(x,b)}{ \sigma_{\alpha} \sqrt{2 \pi} },
\end{align*}
and for all $n \geq 1$, 
\begin{align*}
\sqrt{n}  \, \bb E_{x, b}  \widetilde{W}_{n - k} 
\leq  c_0 e^{-\alpha b} r_{s}(x)  V_{\alpha}(x,b) 
\leq c e^{-\alpha b} (1 + \max\{0, b\}). 
\end{align*}
By the many-to-one formula \eqref{Formula_many_to_one} and the moment assumption $\bb E_{\bb Q_{x}^{\alpha}}  (\max\{0, S_1\}) < \infty$, 
we see that $\sum_{|u| = k}  e^{-\alpha S_u} r_{\alpha} (X_u) V_{\alpha}(X_u, S_u)$ is finite $\bb P_x$-almost surely. 
Hence, by dominated convergence, we get that, $\bb P_x$-almost surely,
\begin{align}\label{inequ-Wnk-lower-2}
\lim_{n \to \infty}  \sqrt{n}  \sum_{|u| = k}  \bb E_{X_u^x, S_u^x}   \widetilde{W}_{n - k} 
=  \sum_{|u| = k}  e^{-\alpha S_u} r_{\alpha} (X_u)  \frac{2 V_{\alpha}(X_u, S_u)}{ \sigma_{\alpha} \sqrt{2 \pi} }. 
\end{align}
Therefore, from \eqref{inequ-Wnk-lower-1} and \eqref{inequ-Wnk-lower-2}, we get that, $\bb P_x$-almost surely, 
\begin{align*}
\liminf_{n \to \infty}
\bb E_{x}  \Big(  e^{-\lambda \sqrt{n} \widetilde{\widetilde{W}}_{n, k} }  \Big| \scr F_{k}  \Big)
\geq  \exp  \bigg(  -\lambda  \sum_{|u| = k}  e^{-\alpha S_u} r_{\alpha} (X_u)  \frac{2 V_{\alpha}(X_u, S_u)}{ \sigma_{\alpha} \sqrt{2 \pi} }   \bigg). 
\end{align*}
Taking into account Lemma \ref{Lem-infinimum-Su} and \eqref{ratio-Su-lu}, we have, 
for every $x \in \bb S_+^{d-1}$, $\bf P_{x}$-a.s., 
\begin{align}\label{limit-derivative-mart-001}
\lim_{k \to \infty} \sum_{|u| = k}  
 e^{-\alpha S_u^x}  r_{s}(X_u^x)  \frac{ 2 V_{\alpha}(X_u^x, S_u^x)}{ \sigma_{\alpha} \sqrt{2 \pi} }
 = \frac{2}{ \sigma_{\alpha} \sqrt{2 \pi} } D_{\infty}, 
\end{align}
so that, $\bf P_{x}$-a.s., 
\begin{align}\label{lower-bound-thm001}
\liminf_{k \to \infty} \liminf_{n \to \infty}
\bb E_{x}  \Big(  e^{-\lambda \sqrt{n} \widetilde{\widetilde{W}}_{n, k} }  \Big| \scr F_{k}  \Big)
\geq  \exp  \bigg(  -\lambda \frac{2}{ \sigma_{\alpha} \sqrt{2 \pi} } D_{\infty}   \bigg). 
\end{align}

We now give an upper bound of the conditional Laplace transform of $\widetilde{\widetilde{W}}_{n, k_0}$. 
For any $\lambda > 0$ and any $\lambda' \in (0, \lambda)$, there exists $\ee > 0$ such that 
$e^{-\lambda t} \leq 1 - \lambda' t$ for any $t \in [0, \ee]$. 
Then we get 
\begin{align*}
\bb E_{x} \Big(  e^{-\lambda \sqrt{n} \widetilde{\widetilde{W}}_{n, k} }  \Big| \scr F_{k}  \Big)
& = \prod_{|u| = k} \bb E_{X_u^x, S_u^x}  \Big(  e^{-\lambda \sqrt{n} \widetilde{W}_{n - k} }  \Big)  \notag\\
& \leq   \prod_{|u| = k} \bb E_{X_u^x, S_u^x} 
\Big[  \exp \Big( -\lambda \sqrt{n} \widetilde{W}_{n - k}  \mathds 1_{ \{ \widetilde{W}_{n - k} < \ee \} }  \Big)  \Big]  \notag\\
& \leq  \prod_{|u| = k}  
\Big[  1 - \lambda'  \bb E_{X_u^x, S_u^x} \Big(  \sqrt{n} \widetilde{W}_{n - k}  \mathds 1_{ \{ \widetilde{W}_{n - k} < \ee \} }  \Big)  \Big]  \notag\\
& \leq  \exp \bigg[ -  \sum_{|u| = k}  \lambda'  \bb E_{X_u^x, S_u^x} \Big(  \sqrt{n} \widetilde{W}_{n - k}  \mathds 1_{ \{ \widetilde{W}_{n - k} < \ee \} }  \Big)  \bigg], 
\end{align*}
where in the last inequality we used $1 - t \leq e^{-t}$ for $t \geq 0$. 
By Fatou's lemma, we have 
\begin{align*}
& \limsup_{n \to \infty} \bb E_{x} \Big( e^{-\lambda \sqrt{n} \widetilde{\widetilde{W}}_{n, k} }  \Big| \scr F_{k}  \Big)  \notag\\
& \leq  \exp \bigg[   \sum_{|u| = k}  - \lambda'  
  \liminf_{n \to \infty} \bb E_{X_u^x, S_u^x} \Big(  \sqrt{n} \widetilde{W}_{n - k}  \mathds 1_{ \{ \widetilde{W}_{n - k} < \ee \} }  \Big)  \bigg]  \notag\\
& =  \exp \bigg( \lambda'    \sum_{|u| = k}  
 \bigg[  - \liminf_{n \to \infty} \bb E_{X_u^x, S_u^x} \Big(  \sqrt{n} \widetilde{W}_{n - k}  \Big)  
   +  \limsup_{n \to \infty} \bb E_{X_u^x, S_u^x} \Big(  \sqrt{n} \widetilde{W}_{n - k}  \mathds 1_{ \{ \widetilde{W}_{n - k} \geq \ee \} }  \Big)   \bigg]  \bigg). 
\end{align*}
By Lemma \ref{Lem-expect-mart-tilde} and Proposition \ref{Prop-Wn-ee}, 
we obtain that, for every $k \in \bb N$ and $x \in \bb S_+^{d-1}$,  $\bb P_x$-a.s.
\begin{align*}
\lim_{n \to \infty} \bb E_{X_u^x, S_u^x} \Big(  \sqrt{n} \widetilde{W}_{n - k}  \Big)  
=   r_{s}(X_u^x)  \frac{ 2 V_{\alpha}(X_u^x, S_u^x)}{ \sigma_{\alpha} \sqrt{2 \pi} }  e^{-\alpha S_u^x}
\end{align*}
and 
\begin{align*}
\limsup_{n \to \infty} \bb E_{X_u^x, S_u^x} \Big(  \sqrt{n} \widetilde{W}_{n - k}  \mathds 1_{ \{ \widetilde{W}_{n - k} \geq \ee \} }  \Big)
\leq  U(X_u^x, S_u^x) e^{-\alpha S_u^x}, 
\end{align*}
where $U$ is a positive measurable function on $\bb S_+^{d-1} \times \bb R_+$ such that for any $x \in \bb S_+^{d-1}$, 
$\frac{U(x, b)}{V_{\alpha}(x, b)} \to 0$ as $b \to \infty$. 
Therefore, for every $k \in \bb N$ and $x \in \bb S_+^{d-1}$,  $\bb P_x$-a.s.,
\begin{align*}
\limsup_{n \to \infty} \bb E_{x} \Big( e^{-\lambda \sqrt{n} \widetilde{\widetilde{W}}_{n, k} }  \Big| \scr F_{k}  \Big)  
\leq  \exp \bigg( \lambda'    \sum_{|u| = k}  
 e^{-\alpha S_u^x} \bigg[  U(X_u^x, S_u^x)  -  r_{s}(X_u^x)  \frac{ 2 V_{\alpha}(X_u^x, S_u^x)}{ \sigma_{\alpha} \sqrt{2 \pi} }
   \bigg]  \bigg).  
\end{align*}
Taking into account Lemma \ref{Lem-infinimum-Su}, \eqref{ratio-Su-lu} and \eqref{limit-derivative-mart-001}, we get, 
for every $x \in \bb S_+^{d-1}$, $\bf P_{x}$-a.s., 
\begin{align*}
\lim_{k \to \infty} \sum_{|u| = k}  
 e^{-\alpha S_u^x}  U(X_u^x, S_u^x) = 0,
\end{align*}
This, together with \eqref{limit-derivative-mart-001}, implies that,  $\bf P_{x}$-a.s.,  
\begin{align*}
\limsup_{k \to \infty}  \limsup_{n \to \infty} 
\bb E_{x} \Big( e^{-\lambda \sqrt{n} \widetilde{\widetilde{W}}_{n, k} }  \Big| \scr F_{k}  \Big) 
\leq \exp \bigg( - \lambda'  \frac{2}{ \sigma_{\alpha} \sqrt{2 \pi} } D_{\infty} \bigg). 
\end{align*}
Combining this with \eqref{lower-bound-thm001} and letting $\lambda'  \to \lambda$, 
we get that, $\bf P_{x}$-a.s.,  
\begin{align*}
\lim_{k \to \infty} \lim_{n \to \infty}
\bb E_{x}  \Big(  e^{-\lambda \sqrt{n} \widetilde{\widetilde{W}}_{n, k} }  \Big| \scr F_{k}  \Big)
=  \exp  \bigg(  -\lambda \frac{2}{ \sigma_{\alpha} \sqrt{2 \pi} } D_{\infty}   \bigg). 
\end{align*}
Using Cantor diagonal extraction, there exists a sequence $(k_n)_{n \geq 1}$ satisfying $k_n \to \infty$ as $n \to \infty$,
such that for any $\lambda \in \bb Q \cap (0, \infty)$, $\bf P_{x}$-a.s.,  
\begin{align*}
\lim_{n \to \infty}
\bb E_{x}  \Big(  e^{-\lambda \sqrt{n} \widetilde{\widetilde{W}}_{n, k_n} }  \Big| \scr F_{k_n}  \Big)
=  \exp  \bigg(  -\lambda \frac{2}{ \sigma_{\alpha} \sqrt{2 \pi} } D_{\infty}   \bigg). 
\end{align*}
Applying \cite[Lemma B.1]{BM19}, we get 
\begin{align*}
\lim_{n \to \infty} \sqrt{n} \widetilde{\widetilde{W}}_{n, k_n} = \left( \frac{2}{\pi \sigma_\alpha^2} \right)^{1/2} D_{\infty}  
\qquad  \mbox{in probability $\bf P_{x}$.}
\end{align*}
By \eqref{proba-Wn-12}, this concludes the proof of Theorem \ref{Thm-Seneta-Heyde}. 
\end{proof}

\section{Duality results for matrix random walks}\label{Sec-Duality}

In this section, we introduce the analogue of the dual (or reversed) random walk \cite[XII.2]{Feller71} in the setting of Markov random walks driven by products of random matrices. In contrast to the classical one-dimensional result, we will not obtain equality in distribution between the original and the dual process, but lower and upper bounds which will be sufficient for many applications. 

Recall that, given a sequence $(g_i)_{i \ge 1}$ of nonnegative random matrices, we have defined the processes 
$S_n = S_0 - \log \| g_n \cdots g_1 X_0 \|$ and $X_n=(g_n \cdots g_1) \cdot X_0$, see \eqref{def-Sn-Xn}. 
In that context, the law of the sequence of the matrices 
$(g_i)_{i \geq 1}$ was given as the projective limit of $\bb Q_{x, b, n}^\alpha$, cf.\ \eqref{def:Qsxn}, where the case $N\equiv 1$ is included by simply omitting summation over $|u|=n$.
The following construction is valid regardless of the specific law of $(g_i)_{i \ge 1}$, and forms the basis for our approximate duality result.

For $1 \leq k \leq n$, we denote $h_{n, k} = g_{n-k+1}^{ * }$, where $g^{ * }$ is the transpose of the matrix $g$. 
Now we introduce the reversed processes $(X_{n, k}^{*}, S_{n, k}^{*})_{0 \leq k \leq n}$: 
let $X_{n, 0}^{*} \in \bb S_+^{d-1}$ and $S_{n, 0}^{*} = 0$, and for any $k \in [1, n]$, 
\begin{align}\label{def-Xni-Sni}
X_{n, k}^{*} =  (h_{n, k} \cdots h_{n, 1}) \cdot X_{n, 0}^*,  \quad   S_{n, k}^{*} =  -\log \| h_{n, k} \cdots h_{n, 1} X_{n, 0}^* \|.  
\end{align}
The following lemma compares the reversed process to the original process. In particular, their values are not equal (as it would be in the one-dimensional case), but still comparable. 

\begin{lemma}\label{Lem-Sn-Snstar}
Assume condition \ref{Condi-Furstenberg-Kesten}. Set $\overline{\varkappa} := 2 \log \varkappa$, where $\varkappa > 1$ is given by condition \ref{Condi-Furstenberg-Kesten}. 
Then, for any $0 \leq k \leq n$ and $X_0, X_{n, 0}^* \in \bb S_+^{d-1}$, 
\begin{align*}
 S_{n, n-k}^* - \overline{\varkappa} - \log d \leq S_n - S_k  
\leq  S_{n, n-k}^* + \overline{\varkappa} + \log d.  
\end{align*}
\end{lemma}

\begin{proof}
We only prove the second inequality, the proof of the first one being exactly along the same lines. The inequality  is obvious when $k = n$, so we only consider $0 \leq k \leq n - 1$. 
By Lemma \ref{lemma kappa 1}, it holds that for any $g \in \Gamma$ and $x, x' \in \bb S_+^{d-1}$, 
\begin{align}\label{inequa-log-norm-vec}
\log \| gx \| \geq \log \|g\| - \overline{\varkappa} \geq  \log \| gx' \| - \overline{\varkappa}, 
\end{align} 
where $\overline{\varkappa} = 2 \log \varkappa$ and $\varkappa > 1$ is given by condition \ref{Condi-Furstenberg-Kesten}. 
Using the first inequality in \eqref{inequa-log-norm-vec}, 
we have that for any $X_0 \in \bb S_+^{d-1}$ and $0 \leq k \leq n-1$, 
\begin{align}\label{inequ-BJ01}
S_n - S_k 
 = - \log \| g_n \cdots g_{k+1} (G_k \cdot X_0) \|  
 \leq  - \log \| g_n \cdots g_{k+1} \|  +  \overline{\varkappa}. 
\end{align}
Since $\|g\| = \sup_{x \in \bb S_+^{d-1}} \|gx\| = \max_{1 \leq j \leq d} \sum_{i=1}^d g^{i,j}$, 
it holds that for any $g \in \Gamma$, 
\begin{align}\label{inequa-g-transpose}
\| g^{ * } \| 
= \max_{1 \leq i \leq d} \sum_{j=1}^d g^{i,j} 
\leq  d \max_{1 \leq j \leq d} \sum_{i=1}^d g^{i,j}
=  d \|g\|. 
\end{align}
By \eqref{inequ-BJ01}, \eqref{inequa-g-transpose}
 and Lemma \ref{lemma kappa 1},  
 we get that for any $X_{n, 0}^* \in \bb S_+^{d-1}$ and $0 \leq k \leq n-1$, 
\begin{align*}
S_n - S_k  \leq  - \log \| g_{k+1}^{ * }  \ldots g_n^{ * }  \|  +  \overline{\varkappa} + \log d
& \leq  - \log \| g_{k+1}^{ * }  \ldots g_n^{ * }  X_0^* \|  +  \overline{\varkappa} + \log d  \notag\\
& = - \log \| h_{n, n-k} \cdots h_{n, 1} X_{n, 0}^* \| +  \overline{\varkappa} + \log d  \notag\\
& =  S_{n, n-k}^* +  \overline{\varkappa} + \log d, 
\end{align*}
completing the proof of the lemma. 
\end{proof}

Note that the definition \eqref{def-Xni-Sni} of $(X_{n, k}^{*}, S_{n, k}^{*})_{0 \leq k \leq n}$ depends on $n$. As the next step, we define a dual 
(branching) random walk with the help of the random matrices $(g_u^*)_{u \in \bb T}$, that allows to approximate (in distribution) the reversed process by an infinite sequence, see Lemma \ref{Lem-seq-hn} below. 

Assume condition \ref{Condi_ms2}, which fixes a constant $\alpha>0$. For each $x \in \bb S_+^{d-1}$, define a probability measure $\bb Q^{\alpha,*}_{x}$ as the projective limit of the sequence of measures $(\bb Q_{x,n}^{\alpha,*})_{n \geq 1}$ defined by 
\begin{align}\label{def:Qsstarxn}
& \int f(h_1, \dots, h_n) \, \bb Q_{x,n}^{\alpha,*}( dh_1, \dots, dh_n)    \notag\\
& := \frac{e^{\alpha b}}{r_\alpha^*(x) }
	\bb E \bigg[ 
	\sum_{|u| = n} r_\alpha^*(g_u^* \cdots g_{u|1}^*x) e^{ \alpha \log \|{g_{u|n}^* \cdots g_{u|1}^* x\|}} 
	f \left(g_{u|1}^*, \dots, g_{u|n}^* \right) \bigg], 
\end{align}
where $r_\alpha^*$ is the eigenfunction of $P_\alpha^*$ (cf. \eqref{eq:defnPsstar}).
Again, this definition includes the case $N\equiv 1$ by omitting summation over $|u|=n$.
Writing $(h_n)_{n\geq 1}$ for the coordinate process with respect to the measure $\bb Q_{x}^{\alpha,*}$,
for $n \in \bb N$, we set 
\begin{align}\label{def:dualprocess}
S_n^*:= -\log \norm{h_n \cdots h_1 X_0^*} 
\quad \mbox{and} \quad 
X_n^*:= (h_n \cdots h_1) \cdot X_0^* 
\end{align}
with the convention that $\bb Q^{\alpha,*}_{x}(X^*_0=x,S^*_0=0)=1.$

 The following lemma compares, for every $n \in \bb N$, the law of $(X_{n, k}^{*}, S_{n, k}^{*})_{0 \leq k \leq n}$ with that of the dual (branching) random walk. It relies on the exchangeability of factors in matrix branching random walks, see Lemma \ref{lem:exchangeability} for details.

\begin{lemma}\label{Lem-seq-hn}
Assume conditions \ref{Condi-Furstenberg-Kesten} and \ref{Condi_ms2}. There are constants $c, C >0$ such that  for every $x \in \bb S^{d-1}_+$, every $n \ge 0$ and all  measurable functions $f$, 
\begin{align*}
	 c \bb E_{\bb Q_{x}^{\alpha,*}} \big[ f(S^*_{0}, \dots, S^*_{n},X^*_{0}, \dots, X^*_{n}) \big] 
	 & \leq  \bb E_{\bb Q_{x}^\alpha} \big[ f(S^*_{n,0}, \dots, S^*_{n,n},X^*_{n,0}, \dots, X^*_{n,n}) \big]\\ 
	  & \leq  C \bb E_{\bb Q_{x}^{\alpha,*}} \big[ f(S^*_{0}, \dots, S^*_{n},X^*_{0}, \dots, X^*_{n}) \big].
\end{align*}
\end{lemma}

\begin{proof} We start by proving the upper bound, {\em i.e.}, the second inequality.
Recall that $G_u=g_u \cdots g_{u|1}$ is the product of the matrices along the branch leading to node $u$. 
Similarly, we denote $G_u^*:=g_u^* \cdots g_{u|1}^*$. 
In the following calculation, we trace the expressions for generic variables $S^*_{n,j}$ and $X^*_{n,j}$. 
By \eqref{def:Qsxn} and \eqref{def-Xni-Sni},
\begin{align*}
& \bb E_{\bb Q_{x}^\alpha} \big[ f(S^*_{n,0}, \dots, S^*_{n,j}, \dots, S^*_{n,n}, \dots, X^*_{n,j}, \dots, X^*_{n,n}) \\ 
& \leq  C \bb E \left[ \sum_{|u| = n} e^{\alpha \log \| g_{u|n} \cdots g_{u|1}\|} 
		f \left(\dots, -\log \|g_{u|n-j+1}^* \cdots g_{u|n}^* x\|, \dots, (g_{u|n-j+1}^* \cdots g_{u|n}^*) \cdot x,\dots \right)   \right] \\
 & =  C \bb E \left[ \sum_{|u| = n} e^{ \alpha \log \| g_{u|1}^* \cdots g_{u|n}^*\| } 
		f \left(\dots, -\log \|g_{u|n-j+1}^* \cdots g_{u|n}^* x\|, \dots, (g_{u|n-j+1}^* \cdots g_{u|n}^*) \cdot x,\dots \right) \right]. 
\end{align*}
By the exchangeability lemma \ref{lem:exchangeability}, we may now reverse the order of the matrices, replacing $g_{u|n-j+1}$ by $g_{u|{j}}$, without changing the value of the expectation. 
We continue
\begin{align*}
& =  C\bb E \left[ \sum_{|u| = n} e^{ \alpha \log \| g_{u|n}^* \cdots g_{u|1}^*\| } 
	f \left(\dots, -\log \|g_{u|j}^* \cdots g_{u|1}^* x\|, \dots, (g_{u|j}^* \cdots g_{u|1}^*) \cdot x,\dots \right) \right] \\
	& =  C \bb E \left[ \sum_{|u| = n} e^{ \alpha (\log \| g_{u|n}^* \cdots g_{u|1}^*x\| +\bar{\varkappa}) } 
	f \left(\dots, -\log \|g_{u|j}^* \cdots g_{u|1}^* x\|, \dots, (g_{u|j}^* \cdots g_{u|1}^*) \cdot x,\dots \right) \right] \\
& \leq  C'  \frac{e^{\alpha \bar{\varkappa}}}{r_\alpha^*(x) }\bb E \left[ \sum_{|u| = n} r_\alpha^*(G_u^*\cdot x) e^{ \alpha \log \| G_u^*x\|} 
	f \left(\dots, -\log \|G_{u|j}^* x\|, \dots, G_{u|j}^* \cdot x,\dots \right) \right] \\
& = C' \bb E_{\bb Q_{x}^{\alpha,*}} \big[ f(S^*_{0}, \dots, S^*_{j}, \dots,S^*_{n},\dots, X^*_{j}, \dots,X^*_{n}) \big]. 
\end{align*} 
The lower bound is proved along the same lines.
\end{proof}

\begin{remark}\label{rem:DualofDual}
	Note that since taking the transpose of a matrix is an involution, the dual process of $(X_n^*,S_n^*)$ w.r.t.~the measure $\bb Q_x^{\alpha,*}$ is again the original process $(X_n,S_n)$ under $\bb Q_x^{\alpha}$.
\end{remark}

Now we are ready to prove an approximate duality result for the laws of the processes $(S_n)_{n \geq 0}$ and $(S_n^*)_{n \geq 0}$; which will serve as a replacement of the distributional identity valid in the classical one-dimensional case, see e.g \cite[Lemma VII.1.2]{Asm03}. 
Let $c_0$ be as in Lemma \ref{Lem-contractivity} and $\bar{\varkappa}:=2 \log \varkappa$, 
where $\varkappa$ is given by \ref{Condi-Furstenberg-Kesten}. 
Introduce
\begin{equation}\label{eq:c1}
	 c_1:=c_0+ \bar{\varkappa} + \log d,
\end{equation}
and for $n \in \bb N$, any nonnegative measurable  function $f:\bb R^n \to \bb R_+$, 
define
\begin{equation}\label{eq:ftilde}
	 \widetilde{f}(x_1, \dots, x_n):=\sup \{f(y_1, \dots y_n) \, : \, |y_i-x_i| \le c_1 \text{ for all } 1 \le i \le n\}.
\end{equation}

\begin{lemma}[Approximate duality]\label{Lem-appro-duality}	
	Assume \ref{Condi-Furstenberg-Kesten} and \ref{Condi_ms2}.
	There exists  $C_1\in (0,\infty)$ such that for all $x,y \in \bb S_+^{d-1}$, $n \geq 1$ and any nonnegative measurable function $f:\bb R^{n} \to \bb R_+$, 
\begin{align*}
 \bb E_{\bb Q_{x}^\alpha} \left[f(S_n-S_{n-1}, \dots,  S_n-S_0)\right] 
	  \leq  C_1\bb E_{\bb Q_{y}^{\alpha,*}}  \left[\widetilde{f}(S_1^*, \dots,  S_n^*)\right]	
\end{align*}
	as well as
		\begin{align*}
\bb E_{\bb Q_{x}^{\alpha,*}} \left[f(S_1^*, \dots,  S_n^*)\right] 
 \leq  C_1\bb E_{\bb Q_{y}^{\alpha}}  \left[\widetilde{f}(S_n-S_{n-1}, \dots,  S_n-S_0)\right].
		\end{align*}
\end{lemma} 

\begin{proof} We only prove the first inequality, the proof of the second one is along the same lines. 
	By  Lemma \ref{Lem-Sn-Snstar}, for every $n \geq 1$,  $0 \le k \le n$ and  $X_0, X_{n, 0}^* \in \bb S_+^{d-1}$, 
	\begin{equation}\label{eq:compSnSnstar}
		 \abs{(S_n-S_k)-S_{n,n-k}^*}\le \bar{\varkappa}+d .
	\end{equation}
	Setting $\varkappa_1 :=\bar{\varkappa}+\log d$ and 
	\begin{equation*} 
		\widetilde{f}_{\varkappa_1}(x_1, \dots, x_n):=\sup \{f(y_1, \dots y_n) \, : \, |y_i-x_i| \le \varkappa_1 \text{ for all } 1 \le i \le n\},
	\end{equation*}
	we obtain by applying first \eqref{eq:compSnSnstar} and then Lemma \ref{Lem-seq-hn}
	\begin{align*}
		  \bb E_{\bb Q_{x}^\alpha} \left[f(S_n-S_{n-1}, \dots,  S_n-S_0)\right]  
		 \leq   \bb E_{\bb Q_{x}^\alpha} \left[ \widetilde{f}_{\varkappa_1}(S_{n,1}^*, \dots  S_{n,n}^*)\right] 
		  \leq   c \bb E_{\bb Q_{x}^{\alpha,*}} \left[ \widetilde{f}_{\varkappa_1}(S_{1}^*,  \dots, S_{n}^*)\right]. 
	\end{align*}
	In the following calculation, we trace the expression for a generic variable $S^*_{n-k}$. 
	From the definition of $\bb Q_{x}^{\alpha,*}$, the boundedness of $r_\alpha^*$ and  Lemma \ref{Lem-contractivity},
	 there exist $c', C_1 >0$ such that for any $x,y \in \bb S_+^{d-1}$, 
	\begin{align*}
	 & \bb E_{\bb Q_{x}^{\alpha,*}} \left[ \widetilde{f}_{\varkappa_1}(S_{1}^*,  \dots, S^*_{n-k}, \dots, S_{n}^*)\right]  \\
		& =	  \frac{1}{r_\alpha^*(x) }\bb E \left[ \sum_{|u| = n} r_\alpha^*(G_u^*\cdot x) e^{ \alpha \log \| G_u^*x\|} 
		\widetilde{f}_{\varkappa_1}( \dots, -\log \|G_{u|(n-k)}^* x\|, \dots  )  \right] \\
			& \le  	C_1  \frac{1}{r_\alpha^*(y) }\bb E \left[ \sum_{|u| = n} r_\alpha^*(G_u^*\cdot y) e^{ \alpha \log \| G_u^*y\|} 
		\widetilde{f} \left( \dots, -\log \|G_{u|(n-k)}^* y\|, \dots \right) \right] \\
		& \leq   c' \bb E_{\bb Q_{y}^{\alpha,*}}  
		\left[\widetilde{f}( \dots, -\log \|G_{u|(n-k)}^* y\|, \dots )\right]. 	
	\end{align*}
	Note that we have replaced $\widetilde{f}_{\varkappa_1}$ 
	by the larger function $\widetilde{f}$ in the penultimate line, employing Lemma \ref{Lem-contractivity}.
\end{proof}

\subsection{Proofs concerning the renewal theory for products of random matrices}
\label{subsect:renewaltheory}

We have the following approximative version of the duality lemma \cite[XII.2, p. 395]{Feller71} for the renewal measure of oscillating Markov random walks driven by products of positive matrices. 	Let $T_n^+$ (or $T_n^-$, respectively) be the weakly ascending (descending) ladder epochs for $S_n^*$, {\em i.e.}, we set 
$T_0^+ = T_0^- =0$ and define recursively
\begin{align}\label{def-Tn-plus-minus}
T_n^+ = \inf\{ j > T_{n-1}^+ \, :\, S_j^* \geq S_{T_{n-1}^+}^* \}, \qquad T_n^- = \inf\{ j > T_{n-1}^- \, :\, S_j^* \leq S_{T_{n-1}^-}^* \}.
\end{align}

\begin{lemma}\label{lem:dualityrenewal} 
Assume \ref{Condi-Furstenberg-Kesten} and \ref{Condi_ms2}.
 Then, there exist constants $c,C>0$ such that, for all $x,y \in \bb S^{d-1}_+$, $t \in \bb R$  and $a\ge0$, 
	\begin{align}
		& c \bb E_{\bb Q_{y}^{\alpha,*}} \left[ \sum_{j=0}^\infty \mathds{1}_{[t+c_1,t+a-c_1]} (S_{T_j^+}^*)    \right]
		\le \bb E_{\bb Q_{x}^\alpha} \left[ \sum_{n=0}^{\tau_{c_1}^--1} \mathds{1}_{[t,t+a]} (S_{n})    \right], 
		 \label{eq:lowerboundAscending} \\ 
		& \bb E_{\bb Q_{x}^\alpha} \left[ \sum_{n=0}^{\tau_{-c_1}^--1} \mathds{1}_{[t,t+a]} (S_{n})    \right]
		\le C \bb E_{\bb Q_{y}^{\alpha,*}} \left[ \sum_{j=0}^\infty \mathds{1}_{[t-c_1,t+a+c_1]} (S_{T_j^+}^*)    \right], \label{eq:upperboundAscending}
	\end{align}
	as well as
		\begin{align*}
		& c \bb E_{\bb Q_{y}^{\alpha,*}} \left[ \sum_{j=0}^\infty \mathds{1}_{[t+c_1,t+a-c_1]} (-S_{T_j^-}^*)    \right] 
		\le  \bb E_{\bb Q_{x}^\alpha} \left[ \sum_{n=0}^{\tau_{c_1}^+-1} \mathds{1}_{[t,t+a]} (-S_{n})    \right],  \\ 
		&  \bb E_{\bb Q_{x}^\alpha} \left[ \sum_{n=0}^{\tau_{-c_1}^+-1} \mathds{1}_{[t,t+a]} (-S_{n})    \right]
		\le C \bb E_{\bb Q_{y}^{\alpha,*}} \left[ \sum_{j=0}^\infty \mathds{1}_{[t-c_1,t+a+c_1]} (-S_{T_j^-}^*)    \right]. 
	\end{align*}
\end{lemma}

\begin{proof} We focus on the upper bound \eqref{eq:upperboundAscending}. 
	For any Borel set $A \subset \bb R$, it holds by an appeal  to Lemma \ref{Lem-appro-duality}, upon defining $\widetilde{A}=\{y \in \bb R \, : \, |y-x| \le c_1 \text{ for some } x \in A\}$,
	\begin{align*}
		& \bb Q_x^\alpha \Big( S_n \in A, \tau_{-c_1}^- >n \Big) 
		= \bb Q_x^\alpha \Big( S_n \in A, - c_1 + S_k \geq 0, \ \forall \ 1\le k \le n \Big)\\
		& \leq C_1 \bb Q_y^{\alpha,*} \Big( S_{n}^* \in \widetilde{A}, \  S_{n}^*-S_{n-k}^* \geq 0, \ \forall \ 1\le k \le n \Big)
		= C_1  \bb Q_y^{\alpha,*} \Big( S_{n}^* \in \widetilde{A}, \  S_{n}^* \geq  \max_{0\le k < n} S_{k}^*   \Big) \\
		&= C  \bb Q_y^{\alpha,*} \Big( S_{n}^* \in \widetilde{A}, \  n=T_j^+ \text{ for some  } j \ge 1   \Big)  
		= C \sum_{j=0}^\infty  \bb Q_y^{\alpha,*} \Big( S_{n}^* \in \widetilde{A}, \  n=T_j^+    \Big).
	\end{align*}
Summing over $n$, we obtain the assertion. 
The lower bound \eqref{eq:lowerboundAscending} follows along the same lines; 
as does the corrresponding result for the descending ladder process.
\end{proof}

The next lemma is a standard result in renewal theory (see e.g.  
\cite[Theorem V.2.4(iii)]{Asm03}) and provides a uniform bound for the renewal measure of intervals of a given length.
Note that the assertions of the lemma would be valid for any underlying probability measure,
and \ref{Condi-Furstenberg-Kesten} and \ref{Condi_ms2} appear here only to ensure the existence of $\bb Q_{x}^{\alpha,*}$. 

\begin{lemma}\label{lem:uniformboundRenewalMeasure} 
Assume \ref{Condi-Furstenberg-Kesten} and \ref{Condi_ms2}.
For any  $x \in \bb S^{d-1}_+$, $a>0$ and $t\in \bb R$, it holds
	$$ \bb E_{\bb Q_{x}^{\alpha,*}} \left[ \sum_{n=0}^\infty \mathds{1}_{[t,t+a]} (S_{T_n^+}^*)   \right]  \leq \sup_{y \in \bb S^{d-1}_+} \bb E_{\bb Q_{y}^{\alpha,*}} \left[ \sum_{n=0}^\infty \mathds{1}_{[0,a]} (S_{T_n^+}^*)   \right]$$
	as well as
		$$ \bb E_{\bb Q_{x}^{\alpha,*}} \left[ \sum_{n=0}^\infty \mathds{1}_{[t,t+a]} (-S_{T_n^-}^*)   \right]  \leq \sup_{y \in \bb S^{d-1}_+} \bb E_{\bb Q_{y}^{\alpha,*}} \left[ \sum_{n=0}^\infty \mathds{1}_{[0,a]} (-S_{T_n^-}^*)   \right].$$
\end{lemma}

\begin{proof} For brevity, write $Z_n^*:=S_{T_n^+}^*$ or $Z_n^*:=-S_{T_n^-}^*$ and note that the process $Z_n^*$ is hence ascending.
	Let $\sigma:= \inf \{ n \, :\, Z_n^* \in [t,t+a] \}$ be the first time $Z_n^*$ enters $[t,t+a]$. Then 
\begin{align*}
 & \bb E_{\bb Q_{x}^{\alpha,*}} \left[ \sum_{n=0}^\infty \mathds{1}_{[t,t+a]} (Z_n^*)  \right]   
	 =  \bb E_{\bb Q_{x}^{\alpha,*}}  \left[ \mathds{1}_{\{\sigma<\infty\}} \sum_{n=\sigma}^\infty \mathds{1}_{[t,t+a]} (Z_{n}^*)    \right] \\
& = \bb E_{\bb Q_{x}^{\alpha,*}}  \left[ \mathds{1}_{\{\sigma<\infty\}} \sum_{j=0}^\infty \mathds{1}_{[t,t+a]} (Z_{\sigma+j}^*)    \right] 
		\le \bb E_{\bb Q_{x}^{\alpha,*}}  \left[ \mathds{1}_{\{\sigma<\infty\}} \sum_{j=0}^\infty \mathds{1}_{[0,a]} (Z_{\sigma+j}^*-Z_{\sigma}^*)    \right] \\
& \le \bb E_{\bb Q_{x}^{\alpha,*}} \left(  \mathds{1}_{\{\sigma<\infty\}}\bb  E_{\bb Q_{X_{\sigma}^*}^{\alpha,*}}  \left[ \sum_{j=0}^\infty \mathds{1}_{[0,a]} (Z_j^*)   \right] \right) 
		 \le \sup_{y \in \bb S_+^{d-1}} \bb  E_{\bb Q_{y}^{\alpha,*}}  \left[ \sum_{j=0}^\infty \mathds{1}_{[0,a]} (Z_{j}^*)   \right] .
	\end{align*}
\end{proof}

The following proposition provides upper bounds for the renewal function of the ladder process in the oscillating case. 

\begin{proposition}\label{thm:boundrenewalfct}
	Assume \ref{Condi-Furstenberg-Kesten}, \ref{CondiNonarith} and \ref{Condi_ms2}. 
	Then, there exists a constant $C$ such that for all $x \in \bb S_+^{d-1}$, $t \in \bb R$ and $a >0$, 
\begin{equation}\label{eq:bound renewal ascending}
		\bb E_{\bb Q_{x}^\alpha} \left[ \sum_{n=0}^{\tau_{-c_1}^--1} \mathds{1}_{[t,t+a]} (S_{n})    \right] 
		\leq C \bb E_{\bb Q_{x}^{\alpha,*}} \left[ \sum_{j=0}^\infty \mathds{1}_{[t-c_1,t+a+c_1]} (S_{T_j^+}^*)    \right] 
		\leq C  \max\{a, 1\}
\end{equation}
	as well as
\begin{equation}\label{eq:bound renewal descending}
		\bb E_{\bb Q_{x}^\alpha} \left[ \sum_{n=0}^{\tau_{-c_1}^+-1} \mathds{1}_{[t,t+a]} (-S_{n})    \right] 
		\leq C \bb E_{\bb Q_{x}^{\alpha,*}} \left[ \sum_{j=0}^\infty \mathds{1}_{[t-c_1,t+a+c_1]} (-S_{T_j^-}^*)    \right] 
		\leq C \max\{a, 1\}. 
\end{equation}
\end{proposition}

\begin{proof} We first prove \eqref{eq:bound renewal ascending}.
	Lemma \ref{lem:dualityrenewal} compares the direct process, killed when becoming negative, with the ascending ladder heights for the reversed process and provides us with the first inequality,
	$$\bb E_{\bb Q_{x}^\alpha} \left[ \sum_{n=0}^{\tau_{-c_1}^- -1} \mathds{1}_{[t,t+a]} (S_{n})    \right] \le C \bb E_{\bb Q_{x}^{\alpha,*}} \left[ \sum_{j=0}^\infty \mathds{1}_{[t-c_1,t+a+c_1]} (S_{T_j^+}^*)    \right].$$
	Next,  applying Lemma \ref{lem:uniformboundRenewalMeasure} twice allows to bound this quantity uniformly in $t \in \bb R$,
	\begin{align*}
\bb E_{\bb Q_{x}^{\alpha,*}} \left[ \sum_{j=0}^\infty \mathds{1}_{[t-c_1,t+a+c_1]} (S_{T_j^+}^*)   \right] 
& ~\le~ \sup_{y \in \bb S^{d-1}_+} \bb E_{\bb Q_{y}^{\alpha,*}} \left[ \sum_{j=0}^\infty \mathds{1}_{[0,a+2c_1]} (S_{T_j^+}^*)   \right] \\
& \leq  ( \floor{a + 2 c_1} + 1)
 \sup_{y \in \bb S^{d-1}_+} \bb E_{\bb Q_{y}^{\alpha,*}} \left[ \sum_{j=0}^\infty \mathds{1}_{[0, 1]} (S_{T_j^+}^*)   \right], 
\end{align*}
where $\floor{a'}$ denotes the integer part of $a'$. 
	We then again apply Lemma \ref{lem:dualityrenewal} to return to the direct process: for any $x,y \in \bb S^{d-1}_+$,	
\begin{align*}
 \bb  E_{\bb Q_{y}^{\alpha,*}}  \left[ \sum_{j=0}^\infty \mathds{1}_{[0, 1]} (S_{T_j^+}^*)   \right] 
 & \le C \bb  E_{\bb Q_{x}^{\alpha}}  \left[ \sum_{k=0}^{\tau_{c_1}^-  -1} \mathds{1}_{[-c_1, 1 + c_1]} (S_k)   \right] \\
& =  C \bb  E_{\bb Q_{x}^{\alpha}}  \left[ \sum_{k=0}^\infty \mathds{1}_{[0, 1 + 2c_1]} (S_k+c_1) \mathds{1}_{\{\tau_{c_1}^- >k\}}   \right]. 
\end{align*}
	Using Lemma \ref{Lem-CLLT-bound} with $y=c_1$, $z= 1 + 2c_1$ and $t = 0$, 
	\begin{align*}
		\bb Q_{x}^{\alpha} \left( c_1 + S_k \in [0, 1 + 2c_1], \tau_{c_1}^- > k \right)
		\leq  \frac{c}{k^{3/2}} \left( 1 +  c_1 \right) (1 + 2c_1)^2  \le  \frac{C}{k^{3/2}}.
	\end{align*}
	It holds $\sum_{k=1}^\infty \frac{1}{k^{3/2}} \le 3$; so we finally obtain that for all $t \in \bb R$, there exsits a constant $C>0$ such that
	$$ \bb E_{\bb Q_{x}^\alpha} \left[ \sum_{n=0}^{\tau_{-c_1}^- -1} \mathds{1}_{[t,t+a]} (S_{n})    \right]
	\leq C ( \floor{a + 2 c_1} + 1) \le C'  \max\{a, 1\},$$
	which completes the proof of \eqref{eq:bound renewal ascending}. 
	
	The proof of \eqref{eq:bound renewal descending} goes along similar lines by using the estimate
	\begin{align*}
		\bb Q_{x}^{\alpha} \left( c_1 - S_k \in [0, 1 + 2c_1], \tau_{c_1}^+ > k \right)
		\leq    \frac{c}{k^{3/2}} \left( 1 +  c_1 \right) (1 + 2c_1)^2  \le  \frac{C}{k^{3/2}} 
	\end{align*}
	provided by Lemma \ref{Lem-CLLT-bound}. 
\end{proof}

\begin{lemma}\label{Lem-second-sum-Renewal}
For all $y \geq 0$, there exists a constant $C = C(y) >0$ such that for any $a >0$ and $s \in \bb R$, 
\begin{align*}
\bb E_{\bb Q_{x}^\alpha} \left[ \sum_{n=\tau_{-c_1}^-}^{\tau_{y}^--1} \mathds{1}_{[s, s+a]} (S_{n})    \right] \le C \max\{a, 1\}. 
\end{align*}
\end{lemma}

\begin{proof}
By additivity, it suffices to prove the assertion for $a = 1$. 
By \cite[Lemma 4.3]{GLP17} or \cite[Proposition 2.6]{Pha18}, 
 there is a martingale approximation  for $S_n$, i.e. there is a mean zero martingale $(M_n, \scr F_n)$ 
 such that $|S_n - M_n| \leq b$, $\bb Q_{x}^\alpha$-a.s.
 The burden of the proof will be $s>b+1$, we first rule out the simpler cases where $s \leq b+1$. 

When $s < -y -1$, the sum is empty. 

When $s \in [-y -1, b+1]$, we use Lemma \ref{Lem-Renewal-thm} to get that 
\begin{align*}
\bb E_{\bb Q_{x}^\alpha} \left[ \sum_{n=\tau_{-c_1}^-}^{\tau_{y}^--1} \mathds{1}_{[s, s+1]} (S_{n})    \right]
\leq \bb E_{\bb Q_{x}^\alpha} \left[ \sum_{n=0}^{\tau_{y}^--1} \mathds{1}_{[-y -1, \, b+2]} (S_{n})    \right]
\leq  C(y). 
\end{align*}

When  $s > b+1$, 
we introduce $T_s = \inf\{n \geq 1: S_n > s \ \mbox{or} \ S_n < -y\}$.
First we prove that there exists $c>0$ such that for any $s >b+1$, 
\begin{align}\label{bound-proba-STs}
\bb Q_{x}^\alpha (S_{T_s} > s) \leq  \frac{c}{s}. 
\end{align}
 Consider $\widetilde{T}_s = \inf\{n \geq 1: M_n > s - b  \ \mbox{or} \ M_n < -y - b  \}$. 
 Then we have 
 \begin{align*}
\bb Q_{x}^\alpha (S_{T_s} > s) \leq \bb Q_{x}^\alpha (M_{\widetilde{T}_s} > s - b). 
\end{align*}
 Using the optional stopping theorem, we get 
 \begin{align*}
0 = \bb E_{\bb Q_{x}^\alpha} M_{\widetilde{T}_s}  
& =   \bb E_{\bb Q_{x}^\alpha} M_{\widetilde{T}_s}  \mathds 1_{ \{ M_{\widetilde{T}_s} > s - b \} }
  + \bb E_{\bb Q_{x}^\alpha} M_{\widetilde{T}_s}  \mathds 1_{ \{ M_{\widetilde{T}_s} <  -y - b \} }  \notag\\
  & \geq   (s - b) \bb Q_{x}^\alpha (M_{\widetilde{T}_s} > s - b) 
  + \bb E_{\bb Q_{x}^\alpha} M_{\widetilde{T}_s}  \mathds 1_{ \{ M_{\widetilde{T}_s} <  -y -b \} }. 
\end{align*}
Hence, for any $s>b+1$, 
\begin{align*}
\bb Q_{x}^\alpha (M_{\widetilde{T}_s} > s - b)  
\leq \frac{1}{s-b}  \bb E_{\bb Q_{x}^\alpha} (- M_{\widetilde{T}_s})  \mathds 1_{ \{ M_{\widetilde{T}_s} <  -y -b \} }. 
\end{align*}
Now let $\widehat T_{c} =  \inf\{n \geq 1:  M_n < -c  \}$. Then 
\begin{align*}
\bb E_{\bb Q_{x}^\alpha} (- M_{\widetilde{T}_s})  \mathds 1_{ \{ M_{\widetilde{T}_s} <  -y -b \} }
=  \bb E_{\bb Q_{x}^\alpha} (- M_{ \widehat T_{y + b} })  \mathds 1_{ \{  \widetilde{T}_s = \widehat T_{y + b} \} } 
\leq   \bb E_{\bb Q_{x}^\alpha} (- M_{ \widehat T_{y + b} })
\leq  c (1 + y + b),
\end{align*}
by using \cite[Lemma 5.8]{GLP17} or  \cite[Proposition 3.6]{Pha18}. 
Combining the above estimates, we have 
 \begin{align*}
\bb Q_{x}^\alpha (S_{T_s} > s) \leq \bb Q_{x}^\alpha (M_{\widetilde{T}_s} > s - b)
\leq  \frac{c (1 + y + b)}{s-b},
\end{align*}
which proves \eqref{bound-proba-STs} for $s > b+1$. 

Then, using the strong Markov property, Lemma \ref{Lem-Renewal-thm-bb} (with $y'=y + s$, $t = s + y - 1$ and $z = 2$)
and \eqref{bound-proba-STs}, we get that for $s > b+1$, 
\begin{align*}
\bb E_{\bb Q_{x}^\alpha} \left[ \sum_{n=\tau_{-c_1}^-}^{\tau_{y}^--1} \mathds{1}_{[s, s+1]} (S_{n})    \right]
& \leq  \bb E_{\bb Q_{x}^\alpha} \left[ \mathds{1}_{ \{ S_{T_s} > s \} }
 \bb E_{\bb Q_{X_{T_s}}^\alpha} \left( \sum_{n=0}^{\infty} \mathds{1}_{[-1, 1]} (S_{n})  \mathds{1}_{  \{\tau_{y+s}^- > n\} }     \right) \right]
 \notag\\
 & \leq  C (s + y + 2)  \bb Q_{x}^\alpha (S_{T_s} > s)  \notag\\
 & \leq  C(y),
\end{align*}
which shows the assertion of the lemma. 
\end{proof}

\begin{proof}[Proof of Theorem \ref{Thm-ladder-height}]
The first inequality follows from \eqref{eq:bound renewal ascending} together with Lemma \ref{Lem-second-sum-Renewal}. 
The second one is obtained from \eqref{eq:bound renewal ascending} by exchanging the roles of the original and the dual processes. 
\end{proof}

\subsection{Spitzer identity and proof of Theorem \ref{Thm-Green-function}}

Let $c_1$ be as in \eqref{eq:c1}, and recall the definition \eqref{eq:ftilde} of $\widetilde{f}$ for a nonnegative function $f$ and the definition \eqref{def-tau-y-plus-minus} of $\tau_y^-$.
The following proposition is a replacement of the well-known Spitzer identity for the joint distribution of partial sums and their maxima.

\begin{proposition} \label{thm:SparreAndersen} 
Assume \ref{Condi-Furstenberg-Kesten} and \ref{Condi_ms2}.
	For any $n \geq 0$, set
	\begin{align}\label{def-Ln}
		L_n = \min_{0 \leq k \leq n} S_k. 
	\end{align}
	There is a constant $C \in (0,\infty)$ such that for any $x \in \bb S_+^{d-1}$ and all non-negative measurable functions $\phi, h : \bb R \to [0,\infty) $ it holds 	
	\begin{align}\label{green-function-product}
		\sum_{n = 0}^{\infty} \bb E_{\bb Q_{x}^\alpha} \left[  \phi(L_n) h(S_n - L_n) \right]
		\leq C  \left( \sum_{n = 0}^{\infty}   \bb E_{\bb Q_{x}^{\alpha,*}} \left[  \widetilde{\phi}(S_n^*)  \mathds 1_{ \{ \tau^*_{c_1}  > n  \} }   \right]  \right)
		\left( \sum_{n = 0}^{\infty}   \bb E_{\bb Q_{x}^\alpha} \left[ \widetilde{h} (S_{n})  \mathds 1_{ \{ \tau_{c_1}  > n \} }  \right] \right), 
	\end{align}
	where  $\tau_y^* := \min \{ j \geq 1: S_j^*-y > 0 \}$. 
\end{proposition}

\begin{proof}
We proceed as in the proof of Theorem 4.4 in \cite{KV17} by decomposing according to the value of
 \begin{align*}
 \tau(n) = \min \{k \in \{0, \dots, n\}: S_k = L_n \}. 
 \end{align*}
 We have that for any $x \in \bb S_+^{d-1}$ and $n \geq 0$, 
 \begin{align*}
 I_n(x): & = \bb E_{\bb Q_{x}^\alpha} \left[  \phi(L_n) h(S_n - L_n) \right]  \notag\\
  & = \sum_{k=0}^n  \bb E_{\bb Q_{x}^\alpha} \left[  \phi(S_k) h(S_n - S_k) \mathds 1_{ \{ \tau(n) = k \} } \right]  \notag\\
  & = \sum_{k=0}^n \bb E_{\bb Q_{x}^\alpha} \left[  \phi(S_k)  \mathds 1_{ \{ S_k < \min_{0 \leq j \leq k-1} S_j \} }  h(S_n - S_k) \mathds 1_{ \{ S_k \leq \min_{ k < i \leq n } S_i  \} } \right], 
 \end{align*}
 with the convention $\min \emptyset = \infty$. 
  By the Markov property of $(X_n, S_n)$, we get
   \begin{align*}
  I_n(x)  = \sum_{k=0}^n \bb E_{\bb Q_{x}^\alpha} \left[  \phi(S_k)  \mathds 1_{ \{ S_k < \min_{0 \leq j \leq k-1} S_j \} }  
    \bb E_{\bb Q_{X_k}^\alpha} \left( h (S_{n-k})  \mathds 1_{ \{ S_j \geq 0, 1 \leq j \leq n - k \} }  \right) \right]. 
 \end{align*}
 Using Lemma \ref{Lem-contractivity}, we obtain that for any $x \in \bb S_+^{d-1}$,
         \begin{align*}
          \bb E_{\bb Q_{X_k}^\alpha} \left( h (S_{n-k})  \mathds 1_{ \{ S_j \geq 0, 1 \leq j \leq n - k \} }  \right) 
         & \leq  \bb E_{\bb Q_{x}^\alpha} \left( \widetilde{h} (S_{n-k})  \mathds 1_{ \{ c_1 + S_j \geq 0, 1 \leq j \leq n - k \} }  \right) \notag\\
         & =  \bb E_{\bb Q_{x}^\alpha} \left( \widetilde{h} (S_{n-k})  \mathds 1_{ \{ \tau_{c_1}  > n - k \} }  \right). 
       \end{align*}
Using Lemma \ref{Lem-appro-duality}, we get that for any $k \geq 0$,  
 \begin{align*}
       & \bb E_{\bb Q_{x}^\alpha} \left[  \phi(S_k)  \mathds 1_{ \{ S_k - \min_{0 \leq j \leq k-1} S_j <0 \} }   \right]
   = \bb E_{\bb Q_{x}^\alpha} \left[  \phi(S_k)  \mathds 1_{ \{ \max_{0 \leq j \leq k-1} (S_k-S_j)<0 \} }   \right]   \\ 
   \leq &  C_1\bb E_{\bb Q_{x}^{\alpha,*}} 
  \left[  \widetilde{\phi}(S_{k}^*)  \mathds 1_{ \{ \max_{0 \leq j  \leq k-1}  S_{k-j}^*-c_1  < 0  \} }   \right] = C_1\bb E_{\bb Q_{x}^{\alpha,*}} 
  \left[  \widetilde{\phi}(S_{k}^*)  \mathds 1_{ \{ \max_{1 \leq j \leq k}  S_j^*-c_1  < 0  \} }   \right]. 
       \end{align*}
Therefore, for any $x \in \bb S_+^{d-1}$, 
\begin{align*}
  I_n(x) \leq C \sum_{k=0}^n  
    \bb E_{\bb Q_{x}^{\alpha,*}} \left[  \widetilde{\phi}(S_k^*)  \mathds 1_{ \{ \tau^*_{c_1}  > k  \} }   \right]
    \bb E_{\bb Q_{x}^\alpha} \left[ \widetilde{h} (S_{n-k})  \mathds 1_{ \{ \tau_{c_1}  > n - k \} }  \right].
  \end{align*}
Summing up over $n \in \bb N$ and using the Cauchy product formula, we obtain the assertion. 
\end{proof}
 
 
 \begin{proof}[Proof of Theorem \ref{Thm-Green-function}]
 Let $f: \bb R_+ \to \bb R_+$ be a bounded and non-increasing function satisfying $\int_0^{\infty} y f(y) dy < \infty$. 
 Let $L_n$ be as in \eqref{def-Ln}. Then, for any $x \in \bb S^{d-1}_+$, 
we decompose according to the values of $(S_n-L_n)$ and use the monotonicity properties of $f$ to 
 rewrite the quantity of interest as follows:
 \begin{align*}
 &	 \sum_{n=0}^\infty \bb E_{\bb Q_{x}^\alpha}  \big[ (b+S_n) f(b+S_n) \mathds{1}_{\{b+\min_{0 \leq k \leq n} S_k \geq 0\}}  \big] \\
&  =  \sum_{n=0}^\infty \bb E_{\bb Q_{x}^\alpha}  \Big[ \big(b+L_n +(S_n-L_n)\big) f\big(b+L_n + (S_n-L_n)\big) \mathds{1}_{\{b+L_n \geq 0\}}  \Big] \notag\\
& \leq \sum_{n=0}^\infty \sum_{k=0}^\infty \bb E_{\bb Q_{x}^\alpha}  \Big[  \mathds{1}_{[k,k+1)}(S_n-L_n)  \big(b+L_n +k+1\big) f\big(b+L_n + k\big) \mathds{1}_{\{b+L_n \geq 0\}} \Big]. 
 \end{align*}
We may now apply Eq. \eqref{green-function-product} of Proposition \ref{thm:SparreAndersen} with functions 
\begin{align*}
	\phi_k(y)&=(y+k+1)f(y+k)\mathds{1}_{\{y\ge 0\}},  \notag \\
	 h_k(z)&=\mathds{1}_{[k,k+1)}(z). 
\end{align*}
Using monotonicity, we get explicit upper bounds for $\widetilde{h}_k$ and $\widetilde{\phi}_k$ (cf.\ \eqref{eq:ftilde}):
\begin{align*}
	\widetilde{\phi}_k(y)&\le(y+k+1+c_1)f(y+k-c_1)\mathds{1}_{\{y+c_1\ge 0\}}, \\
	\widetilde{h}_k(z)&\le\mathds{1}_{[k-c_1,k+c_1+1)}(z). 
\end{align*}
Thus, using \eqref{green-function-product} and then Fubini's theorem, 
 \begin{align}
& \sum_{n=0}^\infty \bb E_{\bb Q_{x}^\alpha}  \big[ (b+S_n) f(b+S_n) \mathds{1}_{\{b+\min_{1 \leq k \leq n} S_k \geq 0\}}  \big] 
\notag \\
& \leq C \sum_{k=0}^\infty 
\left( \sum_{n = 0}^{\infty}   \bb E_{\bb Q_{x}^{\alpha,*}} \left[  \widetilde{\phi}_k(b+S_n^*)  \mathds 1_{ \{ \tau^*_{c_1}  > n  \} }   \right]  \right)
   \left( \sum_{n = 0}^{\infty}   \bb E_{\bb Q_{x}^\alpha} \left[ \widetilde{h}_k (S_{n})  \mathds 1_{ \{ \tau_{c_1}  > n \} }  \right] \right),\label{eq:decompRenewal}
\end{align}
with $\tau_y^* := \min \{ j \geq 1: S_j^*-y > 0 \}$. 
By Theorem \ref{Thm-ladder-height}, we have that the second factor is uniformly bounded for all $k$:
$$\sum_{n = 0}^{\infty}   \bb E_{\bb Q_{x}^\alpha} \left[ \widetilde{h}_k (S_{n})  \mathds 1_{ \{ \tau_{c_1}  > n \} }  \right] =    \bb E_{\bb Q_{x}^\alpha} \left[ \sum_{n = 0}^{\infty} \mathds{1}_{[k-c_1,k+c_1+1)}  (S_n)  \mathds 1_{ \{ \tau_{c_1}  > n \} }  \right] \le C.$$ 
It is therefore sufficient to bound the sum over the first factor in \eqref{eq:decompRenewal}, {\em i.e.}, we consider
\begin{align*}
	& \sum_{k=0}^\infty 
	 \sum_{n = 0}^{\infty}   \bb E_{\bb Q_{x}^{\alpha,*}} \left[  \widetilde{\phi}_k(b+S_n^*)  \mathds 1_{ \{ \tau^*_{c_1}  > n  \} }   \right]  \\
	& \leq \sum_{k=0}^\infty 
  \sum_{n = 0}^{\infty}   \bb E_{\bb Q_{x}^{\alpha,*}} \left[  (b+S_n^*+k+1+c_1)f(b+S_n^*+k-c_1)
   \mathds{1}_{\{b+S_n^*+c_1\ge 0\}} \mathds 1_{ \{ \tau^*_{c_1}  > n  \} }   \right]  	\notag\\
	& = 
   \bb E_{\bb Q_{x}^{\alpha,*}} \left[ \sum_{n = 0}^{\infty}\mathds{1}_{\{-b-c_1\le S_n^*\le c_1\}} \mathds 1_{ \{ \tau^*_{c_1}  > n  \} }  \sum_{k=0}^\infty  (b+S_n^*+k+1+c_1)f(b+S_n^*+k-c_1) \right], 
\end{align*}
where we used Fubini's theorem and that $\{\tau^*_{c_1}>n\}  \subset \{S_n^*\le c_1\}$. 
Considering the inner sum 
$$g(s):=\sum_{k=0}^\infty  (s+k+1+c_1)f(s+k-c_1),$$
we obtain by comparison with an integral, using also the monotonicity of $f$, that 
$$ g(s) \le \int_{s-1}^\infty (r +1+2c_1)f(r) dr.$$
Here we used the substitute  $r=s+k-c_1$.
By the integrability condition $\int_0^\infty y f(y) dy<\infty$, it follows that $g(s)$ is uniformly bounded and decreasing with
$$ \lim_{s \to \infty} g(s)=0.$$
Summarizing, we have that
 \begin{align}
	&	\frac1b \sum_{n=0}^\infty \bb E_{\bb Q_{x}^\alpha}  \big[ (b+S_n) f(b+S_n) \mathds{1}_{\{b+\min_{1 \leq k \leq n} S_k \geq 0\}}  \big] \notag \\
	& \leq \frac{C}{b} \bb E_{\bb Q_{x}^{\alpha,*}} \left[ \sum_{n = 0}^{\infty}\mathds{1}_{[-b-2c_1, 0]}(S_n^*-c_1) \mathds 1_{ \{ \tau^*_{c_1}  > n  \} }  g(b+S_n^*) \right].   \label{eq:renewalSumg}
\end{align}
To finish the proof we have to show that the above quantity goes to $0$ for $b \to \infty$. In order to do so, we use that $g$ is decreasing with $\lim_{s \to \infty} g(s)=0$. Fixing any $\epsilon>0$, there is hence $s^*\in \bb R_+$ such that $0 \le g(s) < \epsilon$ for all $s\ge s^*$. On $[0,s^*]$, we further have that $g$ is uniformly bounded by a constant $c_g$.
We decompose at $s^*$ to obtain that the summand in \eqref{eq:renewalSumg} is bounded by
\begin{align*}
& \mathds{1}_{[-b-2c_1, 0]}(S_n^*-c_1) \mathds 1_{ \{ \tau^*_{c_1}  > n  \} }  g(b+S_n^*) \notag\\
& \leq  \mathds{1}_{[-b-2c_1,0]}(S_n^*-c_1)\mathds{1}_{ \{ \tau^*_{c_1}  > n  \} } \Big(c_g\mathds{1}_{[0, s^*]}(b+S_n^*)    + \epsilon \mathds{1}_{[s^*,\infty)}(b+S_n^*) \Big)  \\
& =  \mathds{1}_{[0,b+2c_1]}(c_1-S_n^*)\mathds{1}_{ \{ \tau^*_{c_1}  > n  \} } \Big(c_g\mathds{1}_{[b+c_1-s^*,b+c_1]}(c_1-S_n^*)    + \epsilon \mathds{1}_{(-\infty,b+c_1-s^*]}(c_1-S_n^*) \Big)  \\
& \leq  c_g  \mathds{1}_{[b-s^*,b]}(-S_n^*) \mathds 1_{ \{ \tau^*_{c_1}  > n  \} }   
  +  \epsilon  \mathds{1}_{[-c_1,b-s^*]}(-S_n^*) \mathds 1_{ \{ \tau^*_{c_1}  > n  \} }. 
\end{align*}
Hence
\begin{align*}  
& \frac{C}{b} \bb E_{\bb Q_{x}^{\alpha,*}} \left[ \sum_{n = 0}^{\infty}\mathds{1}_{[-b-2c_1, 0]}(S_n^*-c_1) \mathds 1_{ \{ \tau^*_{c_1}  > n  \} }  g(b+S_n^*) \right] \\
& \le  \frac{C c_g}{b} \bb E_{\bb Q_{x}^{\alpha,*}} \left[ \sum_{n = 0}^{\infty}\mathds{1}_{[b-s^*,b]}(-S_n^*) \mathds 1_{ \{ \tau^*_{c_1}  > n  \} }  \right] + \frac{\epsilon C}{b} \bb E_{\bb Q_{x}^{\alpha,*}} \left[ \sum_{n = 0}^{\infty}\mathds{1}_{[-c_1,b-s^*]}(-S_n^*) \mathds 1_{ \{ \tau^*_{c_1}  > n  \} } \right]. 
\end{align*}
To bound the renewal measure, we apply the first part of Theorem \ref{Thm-ladder-height} together with Remark \ref{rem:DualofDual} to exchange the roles of $S_n$ and $S_n^*$. We observe that the first interval has a fixed length $s^*$, while the second one grows linearly in $b$. Thus, for any $\epsilon >0$,
\begin{align*}
0  & \le \limsup_{b \to \infty} \frac1b \sum_{n=0}^\infty \bb E_{\bb Q_{x}^\alpha}  \big[ (b+S_n) f(b+S_n) \mathds{1}_{\{b+\min_{1 \leq k \leq n} S_k \geq 0\}}  \big] \\
 & \le \limsup_{b \to \infty} \frac{c_g C s^*}{b} + \epsilon \limsup_{b \to \infty} \frac{C (b-s^*+c_1)}{b} = C \epsilon,
\end{align*}
which proves the assertion.
\end{proof}

\begin{appendix}	
	
\section{List of symbols}

\begin{itemize}
	\item[] \textbf{Basic notation}
	\item[$g^*$] Transpose of $g$ 
		\item[$\norm{\cdot}$] For a vector $v \in \bb R^d$: the 1-norm $\norm{v}=\sum_{i=1}^d \abs{\langle v, e_i \rangle}$. For a matrix $g \in \bb M_+$: the operator norm
	$ \norm{g}=\sup_{v \in \bb R_+^d\setminus\{0\}} \frac{\norm{gv}}{\norm{v}}$
	\item[$\iota(g)$] $\iota(g) : = \inf_{v \in \bb R^d_+ \setminus \{0\} } \frac{\| g v \|}{\|v\|}$ 
	\item[$\bb M_+$] semigroup of $d\times d$ nonnegative matrices that are {\em allowable}
	\item[allowable] A nonnegative matrix $g$ is called allowable, if every row and every column has at least one positive entry
	\item[$\bb R_+$] $\bb R_+=[0,\infty)$
	\item[$\bb S_+^{d-1}$] $\bb S_+^{d-1}=\bb R_+^d \cap \bb S^{d-1}$
	\item[$A^\circ$] denotes the interior of a set $A$
	\item[$\overline{A}$] denotes the (topological) closure of a set $A$
	\item[$\sigma$] norm cocycle: $\sigma(g,x)=\log \norm{gx}$, see \eqref{Def-dot-cocycle}
	\item[$g \cdot x$] for allowable $g$ and $x \in \bb S^{d-1}_+$: $g \cdot x= \norm{gx}^{-1} gx \in \bb S_+^{d-1}$
		\item[$\widetilde{f}$] defined in \eqref{eq:ftilde},
	$$\widetilde{f}(x_1, \dots, x_n) =\sup \{f(y_1, \dots y_n) \, : \, |y_i-x_i| \le c_1 \text{ for all } 1 \le i \le n\}$$
		\item[$\Gamma$] the smallest closed subsemigroup of $\bb M_+$ such that $ \bb P\big(\{G_u:  |u|=1\} \subset \Gamma\big)=1$
			\item[$\mu$] intensity measure of $\mathscr{N}$, $\mu(A)= \bb E \big[ \sum_{ |u| = 1 } \mathds{1}_A(G_u) \big]$	
	\item[$\lambda_g, v_g$]  Perron-Frobenius eigenvalue $\lambda_g$ and eigenvector $v_g \in \bb S_+^{d-1}$ of $g \in \bb M_+^\circ$
	\item[$\Lambda(\Gamma)$] $\Lambda (\Gamma) = \overline{\{v_{g} \in \bb S_+^{d-1}: g\in \Gamma \cap \bb M_+^{\circ} \}}$
			\item[$I$]  $I = \Big\{ s \in \bb R_+ \, : \,  \bb E \Big( \sum_{|u| = 1}  \norm{G_u}^s \Big) < \infty  \Big\}$ 

       \item[]
	\item[] \textbf{Basic notation for branching processes}
		\item[$\bb U$] $\bb U:= \bigcup_{n=0}^\infty \bb N^n$
	\item[$|u|=n$] node $u \in \bb U$ of generation $n$: $u \in \bb N^n$ 
	\item[$u|k$] ancestor of node $u$ in generation $k$
	\item[$\mathscr N$] point process on $\bb M_+$
	\item[$\bb T$] random subtree, underlying Galton-Watson tree of the matrix branching random walk
	\item[$u < v$] for $u,v \in \bb T$: $u < v$ if there is $k < |v|$ with $u = v|k$
	and $u \nless v$ otherwise 
	\item[$\overset{\leftarrow}{u}$] parent of  $u \in \bb T \setminus \o$
	\item[${\rm brot}(u)$] the set of brothers of $u$: $v \in {\rm brot}(u)$ iff $v$ and $u$ have the same parent
	and $v \neq u$ 
	\item[{$[\cdot ]_v$}] shift operator, see \eqref{eq:shift-operator}
		\item[$\mathscr S$] survival event

  \item[]
	\item[] \textbf{Transfer operators $P_s$}
		\item[$\scr C(\bb S_+^{d-1})$] space of continuous functions on $\bb S_+^{d-1}$
	\item[$P_s$] $P_s \varphi(x)  = \int e^{s\sigma(g,x)} \varphi(g \cdot x) \mu(dg)$, see \eqref{Def-Ps}
	\item[$P_s^*$] $P_s^* \varphi(x)  = \int e^{s \sigma(g^*,x)} \varphi(g^* \cdot x) \mu(dg)$, see \eqref{eq:defnPsstar}
	\item[$\mathfrak{m}(s)$] spectral radius of $P_s$
	\item[$\nu_s$] unique probability measure on $\bb S_+^{d-1}$  satisfying $P_s \nu_s = \mathfrak{m}(s) \nu_s$
	\item[$r_s$] (up to scaling) unique positive eigenfunction of $P_s$ satisfying $P_s r_s = \mathfrak{m}(s) r_s$
	\item[$\nu_s^*$, $r_s^*$] eigenmeasure and eigenfunction of $P_s^*$, see Lemma \ref{lem:ExistenceEigenfunctionsPs}

	\item[$\mathfrak{M}$] $\mathfrak M(s) =  \log \mathfrak m(s)$, see \eqref{def-psi-s}
	\item[$m(s)$] $m(s)= \lim_{n \to \infty} \Big( \bb E \big( \sum_{|u| = n}  \norm{G_u}^s \big) \Big)^{1/n}$, see Lemma \ref{lem:ExistenceEigenfunctionsPs}
	\item[$q_n^s(x,g)$] 
	$q_n^s(x,g):= \frac{|gx|^s r_s(g \cdot x)}{\mathfrak{m}(s)^n r_s(x)}$, see \eqref{eq:qnsx}

  \item[]
	\item[] \textbf{Constants}
		\item[$c_0$] defined in Lemma \ref{Lem-contractivity} 
	\item[$c_1$] defined in \eqref{eq:c1}, $c_1=c_0+ \bar{\varkappa} + \log d$
	\item[$\delta$] constant fixed by \ref{Condi_harmonic}
		\item[$\varkappa$] constant fixed by assumption \ref{Condi-Furstenberg-Kesten}
		\item[$\overline{\varkappa}$] $\overline{\varkappa} = 2 \log \varkappa$, see \ref{Condi_harmonic} and Lemma \ref{Lem-Sn-Snstar}
		\item[$\alpha$] constant fixed by \ref{Condi_ms}

	  \item[]
	\item[] \textbf{Probability spaces and filtrations}
		\item[$(\Omega, \mathscr{F},\bb P)$] canonical product probability space for the family $(\mathscr{N}_u)_{u \in \bb U}$ of i.i.d.\ copies of $\mathscr{N}$
		\item[$\scr{F}_n$] $\scr{F}_n=\sigma ( X_{\o}, S_{\o}, \{ u, g_u \, : \, u \in \bb T, |u| \leq n\} )$, see \eqref{def-filtration-Fn}
		\item[$\Omega^{(1)}$] $ \Omega^{(1)} = \Omega \times \bb S^{d-1}_+ \times \bb R$ with corresponding product Borel $\sigma$-field $\mathscr{F}^{(1)}$
	\item[$\mathscr{G}_n$] natural filtration of the spine particles and their brothers, see \eqref{eq:filtrationGn}
	$$\mathscr{G}_n = \sigma \left\{ w|k, X_{w|k}, S_{w|k}, (X_u, S_u)_{ u \in {\rm brot}(w|k) } ; 0 \le k \le n \right\}$$ 
	\item[$ \mathscr{G}_{\infty}$]natural filtration of the spine particles and their brothers, see \eqref{def-scr-G-infty} 
	$$	\mathscr{G}_{\infty} = \sigma \left\{ w|k, X_{w|k}, S_{w|k}, (X_u, S_u)_{ u \in {\rm brot}(w|k) } ; k \geq 0 \right\}$$
	\item[$\mathscr{H}_n$] filtration including the spine particles and their children, see \eqref{filtration-spine-child}
	$$\mathscr{H}_n = \sigma \left\{ w|k, X_{w|k}, S_{w|k}, (X_v, S_v)_{ |v| = k+1, \overset{\leftarrow}{v} = w|k };  0 \le k \le n \right\}$$
	\item[$ \mathscr{H}_{\infty}$] 
	the $\sigma$-algebra generated by
	 the spine particles and their children, see \eqref{def-scr-H-infty}
	$$		\mathscr H_{\infty} = \sigma \left\{ w|k, (X_{w|k}, S_{w|k}), (X_v, S_v);  \overset{\leftarrow}{v} = w|k,  k \geq 0 \right\}
	$$	

  \item[]
	\item[] \textbf{Probability measures}
	\item[$\bb P_{x,b}$] fixing initial values for $X_{\o}, S_{\o}$: $\bb P_{x,b}(X_{\o}=x, S_{\o}=b)=1$; $\bb P_x=\bb P_{x,0}$
	\item[$\bf P$] $\bf P=\bb P(\cdot|\mathscr{S})$ conditioned on survival
		\item[$\bf P_{x,b}$] fixing initial values for $X_{\o}, S_{\o}$: $\bf P_{x,b}(X_{\o}=x, S_{\o}=b)=1$;
		 $\bf P_{x,0} = \bf P_x$ 
	\item[$\wh{\bb P}_{x,b}$] probability measure on $(\Omega^{(1)},\mathscr{F}^{(1)})$: For $A \in \mathscr{F}_n$, 
	$\wh{\bb P}_{x,b} (A) 
	=  \frac{ e^{\alpha b} }{  r_{\alpha}(x)  }  \bb E_{x,b}  \left( W_n\mathds 1_A \right) $, see \eqref{def-hat-P-x-b-intro}
	extended onto $\Omega^{(2)}:= \Omega^{(1)} \times \bb N^{\bb N}$ to carry a random sequence $w \in \bb N^{\bb N}$, called the {\em spine}, the law of which is given by
	$$
	\wh{\bb P}_{x, b} \big( w|n = z  \big| \mathscr F_n \big)
	:=  \frac{ r_{\alpha} (X_z)  e^{-\alpha S_z} }{ W_n },  
	$$
	\item[$\wh{\bb P}^h_{x, b}$] probability measure on $\mathscr{F}_\infty$ 
	$$\wh{\bb P}^h_{x, b} (A) 
	=  \frac{1}{  H_{\alpha}(x, b) }  \bb E_{x,b}  \left( M_n^h\mathds 1_A \right), $$ see \eqref{def-hat-P-x-b}.
	Extended on a larger probability space to carry a random variable $w \in \bb N^{\bb N}$ called the spine,  
	the law of which  is  given by the formula
	$$\wh{\bb P}^h_{x, b} \big( w|n = z  \big| \mathscr F_n \big)
	=  \frac{ H_{\alpha} \left( X_z, S_z \right)   \mathds{1}_{ \left\{ S_{z|k} \in B,  \, \forall  \,   0 \leq k \leq n  \right\} } }{ M_n^h}, $$ see \eqref{law-of-spine-new}
	\item[$\wt{\bb P}_{x,b}$] Law of the branching process with spine $(\left( X_u, S_u, N_u \right)_{u \in \bb T}, w )$, constructed in Section \ref{sect:spinal-decomp}
	\item[$\wh{\bb P}^{(\beta)}_{x, b}$] probability measure on $\mathscr F_{\infty}$ such that 
	$$\wh{\bb P}^{(\beta)}_{x, b} (A) 
	=  \frac{ e^{\alpha b} }{  r_{\alpha}(x) V_{\alpha}^{\beta} (x, b)   }  \bb E_{x,b}  \left( M_n^V(\beta) \mathds 1_A \right)$$
	for any $A \in \mathscr F_n$ and $n \geq 0$, see \eqref{def-P-beta}
	\item[$\bb Q_{x,b}^\alpha$] probability measure on $\bb S_+^{d-1} \times \bb R \times (\bb M_+)^{\bb N}$ with $\bb Q_{x,b}^\alpha(X_0=x, S_0=b)=1$ and the law of $(g_i)_{i \ge 1}$ is fixed by \eqref{def:Qsxn}
	\item[$\bb Q^{\alpha,*}_{x}$] Law of the dual matrix branching random walk, defined in \eqref{def:Qsstarxn}

  \item[]
	\item[] \textbf{Random variables and their expectations}
		\item[$\bb (g_u)_{u \in \bb T}$] the matrix branching random walk generated by $\mathscr{N}$, $g_{\o}$ is the identity matrix
		\item[$G_u$] $G_u= g_u g_{u|n-1}\ldots g_{u|2} g_{u|1}$, $u\in \bb T$
			\item[$(X_u, S_u)$] $X_u=G_u \cdot X_{\o}$, $S_u=-\log \norm{G_u X_{\o}} + S_{\o}$, see \eqref{Def-Xku_Sku}
			\item[$h_{n, k}$] $h_{n, k}= g_{n-k+1}^{ *}$
		\item[$(X_{n, k}^{*}, S_{n, k}^{*})$] reversed process, see \eqref{def-Xni-Sni}
		$$X_{n, k}^{*} =  (h_{n, k} \cdots h_{n, 1}) \cdot X_{n, 0}^*,  \quad   S_{n, k}^{*} =  -\log \| h_{n, k} \cdots h_{n, 1} X_{n, 0}^* \|$$
			\item[$(X_n, S_n)$] Markov random walk under $\bb Q_{x,b}^\alpha$, see \eqref{def-Sn-Xn}
		$$S_n = S_0 - \log \| g_n \cdots g_1 X_0 \| \  \text{ and } \ X_n=(g_n \cdots g_1) \cdot X_0$$ 
		\item[$(X_n^*, S_n^*)$] Dual Markov random walk under $\bb Q_{x,b}^{\alpha,*}$, see \eqref{def:dualprocess}
		$$S_n^*= -\log \norm{h_n \cdots h_1 X_0^*} 
		\quad \mbox{and} \quad 
		X_n^*= (h_n \cdots h_1) \cdot X_0^*$$ 
		\item[$L_n$] $L_n = \min_{0 \leq k \leq n} S_k$, see Proposition \ref{thm:SparreAndersen}
		\item[$\ell_\alpha$] $ \ell_{\alpha}(x) = \lim_{n \to \infty} \bb E_{\bb Q_x^\alpha} S_n $, see \eqref{def-ell-alpha-x}
		\item[$\sigma_\alpha^2$] $\sigma^2_{\alpha} = \lim_{n \to \infty} \frac{1}{n} \bb E_{\bb Q_x^\alpha} S_n^2$, 
		positive under \ref{CondiNonarith}, see \eqref{def-sigma-alpha}
			\item[$\Theta_{x,b}$] Point process on $\bb S_+^{d-1} \times \bb R$, see \eqref{def-point-process-root}
		$$\Theta_{x,b} = \sum_{|u|=1} \delta_{ \left( g_u \cdot x,  \,  b + \sigma (g_u, x) \right) }$$
		\item[$\wt{\Theta}_{x, b} $] biased point process, see \eqref{def-wt-Theta-xb}:
		$$\wt{\Theta}_{x, b} (dy, ds)
		= \frac{ H_{\alpha}(y, s)  }{ H_{\alpha}(x, b) } \mathds 1_{ \{ s \in B \} }
		\Theta_{x, b} (dy, ds) $$

  \item[]
	\item[] \textbf{Martingales and related processes}
		\item[$W_n$] additive martingale, see \eqref{def-addi-martigale}, with $\alpha$ from condition \ref{Condi_ms}, 
	$$W_n  = W_n(\alpha) =  \frac{1}{\mathfrak m (\alpha)^n}  \sum_{|u| = n} e^{-\alpha S_u} r_{\alpha}(X_u),   \quad  n \geq 1$$
	\item[$W_\infty$] $\bb P_x$-a.s. limit of $W_n$
	\item[$\widetilde{W}_n$] $\widetilde{W}_n =  \sum_{|u| = n}  r_{\alpha}(X_u)  e^{- \alpha S_u}  \mathds 1_{ \{ \min_{v \leq u} S_v \geq 0 \} }$, see \eqref{def-Wn-tilde}
	\item[$\widetilde{\widetilde{W}}_{n, k}$] $\widetilde{\widetilde{W}}_{n, k} =  \sum_{|u| = n}  r_{\alpha}(X_u)  e^{-\alpha S_u}  \mathds 1_{ \{ \min_{v \leq u, |v| \geq k} S_v \geq 0 \} }$, see \eqref{def-Wn-tilde-nk}
		\item[$D_n$] derivative martingale, see \eqref{def-derivative-martingale}
	$$D_n = \sum_{|u|=n} \big( S_u + \ell_{\alpha} (X_u) \big) e^{- \alpha S_u}  r_{\alpha}(X_u)$$
	\item[$D_\infty$] $\bb P_x$-a.s. limit of $D_n$
		\item[$M_n^h$] $\bb P_x$-martingale defined in terms of an harmonic function $h$, see \eqref{def-martingale-Dn} 
	$$M_n^h= \sum_{ |u| = n }    H_{\alpha} \left( X_u, S_u \right)  \mathds{1}_{ \left\{ S_{u|k} \in B,  \, \forall  \,   0 \leq k \leq n  \right\} }$$  
	\item[$M_n^V(\beta)$] martingale defined in terms of the harmonic 
	function $V_\alpha^\beta$, see \eqref{def-martingale-Dn-V}
	$$M_n^V(\beta) = \sum_{ |u| = n }  r_{\alpha}(X_u)  V_{\alpha}^{\beta}(X_u, S_u)  e^{-\alpha S_u}
	\mathds{1}_{ \left\{  S_{u|k} \geq -\beta,  \, \forall  \,   0 \leq k \leq n  \right\} }$$
	\item[$\widetilde{D}_n^V(\beta)$] is defined in \eqref{def-tilde-D-n-V-beta} 
	$$ \widetilde{D}_n^V(\beta)= \sum_{ |u| = n }   r_{\alpha}(X_u)  V_{\alpha}^{\beta}\left( X_u, S_u \right)  e^{- \alpha S_u}$$

  \item[]
	\item[] \textbf{Stopping times}
	\item[$\tau_y^-, \tau_y^+$] delayed first entrance times into $(-\infty,0)$ and $(0,\infty)$, respectively,  see \eqref{def-tau-y-plus-minus}
	$$ \tau_y^- =\inf\{ k \ge 1:  y + S_k <0\}, \qquad \tau_y^+ =\inf\{ k \ge 1: S_k-y>0\}$$ 
	\item[$T_n^+, T_n^-$] weakly ascending (descending) ladder epochs for $S_n^*$, see \eqref{def-Tn-plus-minus} 
	$$ T_n^+ = \inf\{ j > T_{n-1}^+ \, :\, S_j^* \geq S_{T_{n-1}^+}^* \}, \qquad T_n^- = \inf\{ j > T_{n-1}^- \, :\, S_j^* \leq S_{T_{n-1}^-}^* \}$$ 
	\item[$\tau_y^*$] $\tau_y^* = \min \{ k \geq 1: S_k^*-y > 0 \}$, see Proposition \ref{thm:SparreAndersen}
		\item[$\scr{T}_n$] weakly ascending ladder times for $S_n$, $\scr{T}_n = \inf\{ j > \scr{T}_{n-1} \, :\, S_j \geq S_{\scr{T}_{n-1}} \},$ see \eqref{def-T-n-001} 

  \item[]
	\item[] \textbf{Harmonic functions}
	\item[$h$] harmonic function for $(X_n,S_n)$ under $\bb Q_{x,b}^\alpha$, killed when $S_n$ leaves a set $B$, see \eqref{harmonicity-h}
	\item[$H_\alpha$] Given a harmonic function $h$, $H_{\alpha}(y, s) = r_{\alpha}(y)  h(y,s)  e^{- \alpha s}$, see \eqref{def-H-alpha-ys}

	\item[$V_{\alpha}(x, y)$] harmonic function for $S_n$ started at $y$, killed when going below 0, see \eqref{def-V-alpha-xy} 
			$$V_{\alpha}(x, y) = \lim_{n \to \infty} \bb E_{\bb Q_{x}^{\alpha}} \left( y + S_n; \tau_y > n \right)$$
	\item[$V_{\alpha}^{\beta}$] $V_{\alpha}^{\beta} (x, y)  = V_{\alpha} (x, y + \beta)$, see \eqref{def-V-alpha-beta-xy} 
\end{itemize}
	
\end{appendix}


\end{document}